\documentclass[a4paper,11pt,reqno]{amsart}

\usepackage[T1]{fontenc}
\usepackage{latexsym,amsmath,amsfonts,amscd,amssymb,amsthm,mathrsfs}
\usepackage{tikz}
\usepackage{pifont}
\usepackage{commath}
\usepackage{caption}
\usepackage{mathabx}
\usepackage{enumerate}
\usepackage[all]{xy}
\usepackage{color}
\definecolor{cobalt}{RGB}{61,89,171}
\usepackage[colorlinks,citecolor=cobalt,linkcolor=cobalt,urlcolor=cobalt,pdfpagemode=UseNone,backref = page]{hyperref}

\usepackage{soul}
\usepackage{pdflscape}
\usepackage{booktabs}
\usepackage{longtable}
\usepackage[symbol]{footmisc}

\newcommand{\referenza}{}

\theoremstyle{plain}
\newtheorem{theorem}{Theorem}[section]
\newtheorem*{theorem*}{Theorem}
\newtheorem{proposition}[theorem]{Proposition}
\newtheorem*{proposition*}{Proposition}
\newtheorem{lemma}[theorem]{Lemma}
\newtheorem{corollary}[theorem]{Corollary}
\newtheorem*{corollary*}{Corollary \referenza}

\theoremstyle{definition}

\newtheorem*{definition1}{Def{}inition}
\newtheorem{example}[theorem]{Example}
\theoremstyle{remark}
\newtheorem{remark}[theorem]{Remark}

\newcommand{\ad}{\mathrm{ad}}
\newcommand{\Aut}{\mathrm{Aut}}
\newcommand{\diag}{\mathrm{diag}}
\newcommand{\Heis}{\mathrm{Heis}}

\newcommand{\bR}{\mathbb{R}}
\newcommand{\bQ}{\mathbb{Q}}
\newcommand{\bZ}{\mathbb{Z}}

\newcommand{\cF}{\mathcal{F}}

\newcommand{\cZ}{\mathcal{Z}}

\newcommand{\fd}{\mathfrak{d}}
\newcommand{\fg}{\mathfrak{g}}
\newcommand{\fh}{\mathfrak{h}}
\newcommand{\fj}{\mathfrak{j}}
\newcommand{\fk}{\mathfrak{k}}
\newcommand{\fn}{\mathfrak{n}}
\newcommand{\fr}{\mathfrak{r}}
\newcommand{\fs}{\mathfrak{s}}

\newcommand{\heis}{\mathfrak{heis}}

\newcommand{\la}{\langle}
\newcommand{\ra}{\rangle}

\newcommand{\vt}{\vartheta}
\newcommand{\ve}{\varepsilon}
\newcommand{\om}{\omega}
\newcommand{\Om}{\Omega}

\begin{document}
\title[lcs Lie algebras]{Structure of locally conformally symplectic Lie algebras and solvmanifolds}

\author[D. Angella]{Daniele Angella}
\address[D. Angella]{Dipartimento di Matematica e Informatica "Ulisse Dini"\\
Universit\`a degli Studi di Firenze\\
viale Morgagni 67/a\\
50134 Firenze, Italy
}
\email{daniele.angella@unifi.it, daniele.angella@gmail.com}

\author[G. Bazzoni]{Giovanni Bazzoni}
\address[G. Bazzoni]{Fachbereich Mathematik und Informatik, Philipps-Universit\"{a}t Marburg, Hans-Meerwein-Str. 6 (Campus Lahnberge), 35032 Marburg, Germany}
\email{bazzoni@mathematik.uni-marburg.de}

\author[M. Parton]{Maurizio Parton}
\address[M. Parton]{
Dipartimento di Economia\\
Universit\`a di Chieti-Pescara\\
viale della Pineta 4\\
65129 Pescara, Italy
} 
\email{parton@unich.it}


\begin{abstract}
We obtain structure results for locally conformally symplectic Lie algebras. We classify locally conformally symplectic structures on four-dimensional Lie algebras and construct locally conformally symplectic structures on compact quotients of all four-dimensional connected and simply connected solvable Lie groups.
\end{abstract}

\maketitle

\date{\today}

\section*{Introduction}

A {\em locally conformally symplectic} (shortly, {\em lcs}) structure \cite{Lee-1} on a differentiable manifold $M$ consists of an open cover $\{U_j\}_j$ of $M$ and a non-degenerate $2$-form $\Omega$ such that $\Omega_j\coloneq \imath_j^*\Omega$ is closed (hence symplectic), up to a conformal change, on each open set $\imath_j\colon U_j\to M$. If $f_j\in C^\infty(U_j)$ is a smooth function such that $\exp(-f_j)\Omega_j$ is symplectic, then $d\Omega_j-df_j\wedge\Omega_j=0$ on $U_j$. Since $df_j=df_k$ on $U_j\cap U_k$, the local 1-forms $\{df_j\}_j$ satisfy the cocycle condition and piece together to a global $1$-form $\vartheta$ on $M$, the {\em Lee form}, and $(\Omega,\vartheta)$ satisfies the equations
\begin{equation}\label{eq:lcs}
d\vartheta=0, \qquad d\Omega-\vartheta\wedge\Omega=0 .
\end{equation}
By Poincar\'e Lemma, every closed $1$-form is locally exact. Hence a lcs structure is given, equivalently, by a non-degenerate $2$-form $\Omega$ and a $1$-form $\vartheta$ satisfying \eqref{eq:lcs}.
The ``limit'' case $\vartheta=0$ recovers a symplectic structure, while the case $[\vartheta]=0$ means that $\Om$ is globally conformal to a symplectic structure, {\itshape i.e.} globally conformally symplectic. Hence, in a sense, lcs structures can be seen as a generalization of symplectic structures. As shown in \cite{Vaisman}, for instance, lcs manifolds are natural phase spaces of Hamiltonian dynamical systems. They also appear as even-dimensional transitive leaves in Jacobi manifolds, see \cite{Gue_Lich}.

In this paper, however,  we focus on ``genuine'' lcs structures, those whose Lee form $\vt$ satisfies $[\vt]\neq 0$. This condition prevents some manifolds which are lcs from being symplectic. Lcs geometry is currently an active research area, see \cite{ATO,Banyaga,El_Murph,Savelyev,VLV}.

The purpose of this note is to investigate the structure of Lie groups endowed with left-invariant lcs structures and to show, under certain assumptions, how to construct them. Since we consider left-invariant structures, Lie algebras are the natural object of study. In particular, we revisit and extend, in an algebraic setting, some results of Banyaga \cite{Banyaga} and of the second-named author with J.\ C.\ Marrero \cite{Bazzoni_Marrero}. We also adapt to the lcs case some ideas of Ovando \cite{Ovando} on the structure of symplectic Lie algebras. Moreover, we classify left-invariant lcs structures on four-dimensional Lie groups and construct lcs structures on their compact quotients.

Recall that a Hermitian structure $(J,g)$ on a manifold $M$ is \emph{locally conformally K\"ahler}, lcK for short, if its fundamental form $\Om$, defined by $\Om(X,Y)=g(JX,Y)$, satisfies $d\Omega=\vt\wedge\Omega$, where $\vt$ is the Lee form of the Hermitian structure, see \cite{Gray_Hervella}. LcK geometry has received a great deal of attention over the last years, both from the mathematical and from the physical community (see for instance \cite{ACHK,GMO,OrneaVerbitsky-2,Parton1,Parton2,Shahbazi2} and the monograph \cite{DO}). A lcK structure is \emph{Vaisman} if $\nabla\vt=0$. Every lcK structure is a lcs structure in a natural way. In this sense, our results can be seen as the lcs equivalent of the work of Belgun \cite{Belgun} and Hasegawa \emph{et\ al.}\ \cite{HaseKami} on lcK structures on compact complex surfaces modeled on Lie groups.

Let us recollect some definitions of lcs geometry. If $(\Omega,\vartheta)$ is a lcs structure on a manifold $M$, the {\em characteristic field} $V\in\mathcal{X}(M)$ is the dual of the Lee form $\vartheta$ with respect to the non-degenerate form $\Omega$, namely,
$$ \imath_V \Omega = \vartheta . $$
This terminology is due to Vaisman \cite{Vaisman}. If $(J,g,\Om,\vt)$ is a lcK structure, the \emph{Lee vector field} is the metric dual of the Lee form; hence, if the lcs structure comes from an lcK structure, the Lee field equals $J(V)$.

We consider the Lie subalgebra $\mathcal{X}_\Omega(M)\subset\mathcal{X}(M)$ of infinitesimal automorphisms of the lcs structure $(\Omega,\vartheta)$, {\itshape i.e.} $\mathcal{X}_\Omega(M)=\{X\in\mathcal{X}(M) \mid L_X\Omega=0\}$, from which $L_X\vartheta=0$ follows (here $L$ denotes the Lie derivative and $M$ is assumed to be connected of dimension $2n\geq4$). In particular, $V\in\mathcal{X}_\Omega(M)$. For $X\in\mathcal{X}_\Omega(M)$, the function $\imath_X\vartheta$ is constant, hence there is a well-defined morphism of Lie algebras
\[
\ell \colon \mathcal{X}_\Omega(M)\to\mathbb{R}, \quad X\mapsto\imath_X\vartheta\,, 
\]
called the \emph{Lee morphism}; clearly $V\in\ker\ell$. Either $\ell$ is surjective, and we say that the lcs structure is {\em of the first kind} \cite{Vaisman}; or $\ell=0$, and the lcs structure is {\em of the second kind}. If the lcs structure is of the first kind, one can choose $U\in\mathcal{X}_\Omega(M)$ with $\vartheta(U)=1$; we refer to $U$ as a {\em transversal field}. The choice of a transversal field $U$ determines a 1-form $\eta$ by the condition $\eta=-\imath_U\Om$. Clearly $\vt(V)=\eta(U)=0$, while $\vt(U)=-\Omega(U,V)=\eta(V)$; moreover, one has $\Om=d\eta-\vt\wedge\eta$. Lcs structures of the first kind exist on four-manifolds satisfying certain assumptions, see \cite[Corollary 4.12]{Bazzoni_Marrero}.

Given a smooth manifold $M$ and a diffeomorphism $\varphi\colon M\to M$, the \emph{mapping torus} of $M$ and $\varphi$ is the quotient space of $M\times\bR$ by the equivalence relation $(x,t)\sim (\varphi(x),t+1)$. It is a fibre bundle over $S^1$ with fibre $M$. A result of Banyaga (see \cite[Theorem 2]{Banyaga-1}) says that a compact manifold endowed with a lcs structure of the first kind is diffeomorphic to the mapping torus of a contact manifold and a strict contactomorphism. A similar result has been proved by the second-named author and J.\ C.\ Marrero \cite[Theorem 4.7]{Bazzoni_Marrero} for lcs manifolds of the first kind with the property that the foliation $\cF=\{\vt=0\}$ admits a compact leaf. It is possible to see that the lcs structure underlying a Vaisman structure is of the first kind. In \cite{OrneaVerbitsky-2} Ornea and Verbitsky proved that a compact Vaisman manifold is diffeomorphic to the mapping torus of a Sasakian manifold and a Sasaki automorphism. Thus the structure of compact lcs manifolds of the first kind and of compact Vaisman manifolds, as well as their relationships with other notable geometric structures, is well understood. Nothing is known, however, for lcs structures of the second kind, and this was one of the motivations for our research.

Lcs structures can be distinguished according to another criterion. Given a smooth manifold $M$ endowed with a closed 1-form $\vt$, one can define a differential $d_\vt$ on $\Omega^\bullet(M)$ by setting $d_\vt=d-\vt\wedge\_$. The cohomology of the complex $(\Omega^\bullet(M),d_\vt)$, denoted $H^\bullet_\vt(M)$, is known as \emph{Morse-Novikov} or \emph{Lichnerowicz} cohomology of $(M,\vt)$, see \cite{Gue_Lich}. If $(\Omega,\vt)$ is a lcs structure on $M$, the lcs condition is equivalent to $d_\vt\Omega=0$, hence $\Omega$ defines a cohomology class $[\Omega]\in H^2_\vt(M)$. If $[\Omega]=0$, the lcs structure is \emph{exact}, otherwise it is \emph{non-exact}. Notice that a lcs structure of the first kind is automatically exact. As shown in \cite{El_Murph} (see also \cite[Theorem 2.15]{Cha_Murph}), exact lcs structures exist on every closed manifold $M$ with $H^1(M;\bR)\neq 0$ endowed with an almost symplectic form.

As announced, in this paper we restrict our attention to left-invariant lcs structures on Lie groups. Such a structure can be read in the Lie algebra of the Lie group and it is natural to give the following

\begin{definition1}
A \emph{locally conformally symplectic (lcs)} structure on a Lie algebra $\fg$ with $\dim\fg=2n\geq 4$ consists of $\Omega\in\Lambda^2\fg^*$ and $\vt\in\fg^*$ such that\footnote[1]{Hereafter $d$ denotes the Chevalley-Eilenberg differential of 
$\fg$.} $\Omega^n\neq 0$, $d\vt=0$ and $d\Omega=\vt\wedge\Om$. The \emph{characteristic vector} $V$ of the lcs structure is defined by $\imath_V\Om=\vt$.
\end{definition1}

Lcs structures on almost abelian Lie algebras have recently been studied in \cite{AO}. Given a lcs Lie algebra $(\fg,\Omega,\vt)$ we set $\fg_\Om=\{X\in\fg \mid L_X\Om=0\}$; notice that $V\in\fg_\Om$. We have an algebraic analogue of the Lee morphism, $\ell\colon\fg_\Omega\to\bR$, $\ell(X)=\vt(X)$. If it is non-zero then the lcs structure $(\Omega,\vt)$ is of the first kind, otherwise it is of the second kind. An element $U\in\fg_\Om$ with $\vt(U)=1$ is called a \emph{transversal vector} and, as above, the choice of $U$ determines $\eta\in\fg^*$ by the condition $\eta=-\imath_U\Om$. One has $\vt(V)=\eta(U)=0$, $\vt(U)=-\Omega(U,V)=\eta(V)$ and $\Om=d\eta-\vt\wedge\eta$.

The algebraic analogue of the structure result for compact manifolds endowed with lcs structure of the first kind has been proved in \cite[Theorem 5.9]{Bazzoni_Marrero}. Let $(\fg,\Om,\vt)$ be a $2n$-dimensional lcs Lie algebra of the first kind with transversal vector $U$. Then the ideal $\fh\coloneq\ker\vartheta$ is endowed with the contact form $\eta\big|_{\fh}$, denoted again by $\eta$, and with a contact derivation $D$, {\itshape i.e.} $D^*\eta=0$, induced by $\mathrm{ad}_U$ (here our convention is that, given a linear map $D\colon\fg\to\fg$, the dual map $D^*\colon\fg^*\to\fg^*$ is defined by $(D^*\alpha)(X)=\alpha(DX)$). Moreover $\fg\simeq\fh \rtimes_D \bR$, the {\itshape semidirect product} of $\mathfrak{h}$ and $\bR$ by $D$; this is just $\fh\oplus\bR$ with Lie bracket
\[
\left[ (X,a),\, (Y,b) \right] \coloneq  \left( aD(Y)-bD(X)+[X,Y]_{\mathfrak{h}} ,\, 0 \right)\,;
\]
in particular we get an exact sequence of Lie algebras $0 \to \fh\to \fg \to \bR \to 0$. 
Recall that a contact Lie algebra is a $(2n-1)$-dimensional Lie algebra $\fh$ with a 1-form $\eta\in\fh^*$ such that $\eta\wedge d\eta^{n-1}\neq 0$. Conversely, the datum of a contact Lie algebra $(\fh,\eta)$ with a contact derivation $D$ defines a lcs structure of the first kind on $\fh\rtimes_D\bR$. 

A lcs structure on a nilpotent Lie algebra is necessarily of the first kind. In \cite{Bazzoni_Marrero} the authors introduce the notion of {\em lcs extension} and characterize (see \cite[Theorem 5.16]{Bazzoni_Marrero}) every lcs nilpotent Lie algebra of dimension $2n+2$ as the lcs extension of a nilpotent symplectic Lie algebra of dimension $2n$ by a symplectic nilpotent derivation; such nilpotent symplectic Lie algebra can in turn be obtained by a sequence of $n-1$ {\em symplectic double extensions} \cite{Medina_Revoy} by nilpotent derivations from the Abelian $\bR^2$.

As in the geometric case, lcs structures on Lie algebras can be distinguished according to another criterion. Given a Lie algebra $\fg$ and $\vt\in\fg^*$, one can define a differential $d_\vt$ on $\Lambda^\bullet\fg^*$ by setting $d_\vt=d-\vt\wedge\_$. The cohomology of $(\Lambda^\bullet\fg^*,d_\vt)$, denoted $H^\bullet_\vt(\fg)$, is the \emph{Morse-Novikov} or \emph{Lichnerowicz} cohomology of $(\fg,\vt)$. If $(\fg,\Omega,\vt)$ is a lcs Lie algebra, the lcs condition is equivalent to $d_\vt\Omega=0$, hence $\Omega$ defines a cohomology class $[\Omega]\in H^2_\vt(\fg)$. If $[\Omega]=0$, the lcs structure is {\em exact}, otherwise it is {\em non-exact}. As above, a lcs structure of the first kind is automatically exact. The converse is true when the Lie algebra is unimodular, \cite[Proposition 5.5]{Bazzoni_Marrero}. However, there exist exact lcs Lie algebras which are not of the first kind. Therefore, the results of \cite{Bazzoni_Marrero} do not apply to them. In the exact case, a primitive of $\Om$, that is, $\eta\in\fg^*$ such that $d_\vt\eta=\Omega$, determines a unique vector $U\in\fg$ by the equation $\eta=-\imath_U\Om$.

In \cite{Ovando}, Ovando classifies all symplectic structures on four-dimensional Lie algebras up to equivalence, describing them either as solutions of the {\em cotangent extension problem} (see \cite{Boyom}), or as a symplectic double extension of $\mathbb{R}^2$. 

Inspired by the results contained in \cite{Bazzoni_Marrero} and \cite{Ovando}, we study the structure of lcs Lie algebras.

Our first result extends \cite[Theorem 5.9]{Bazzoni_Marrero} to exact lcs structures, not necessarily of the first kind --- see Theorem \ref{corr:exact}.

\begin{theorem*}
 There is a one-to-one correspondence 
 \[
\left\{\begin{array}{c}
\textrm{exact lcs Lie algebras } (\fg,\Om=d\eta-\vt\wedge\eta,\vt),\\
\textrm{dim }\fg=2n, \textrm{ such that }\vt(U)\neq 0,\\
\textrm{ where }\eta=-\imath_U\Om\end{array}\right\}\leftrightarrow
\left\{\begin{array}{c}
\textrm{contact Lie algebras } (\fh,\eta),\\ 
\textrm{dim }\fh=2n-1, \textrm{ with a derivation}\\
D \textrm{ such that } D^*\eta=\alpha\eta, \alpha\neq1
\end{array}\right\}
\]
The correspondence sends $(\fg,\Omega,\vt)$ to $(\ker\vt,\eta,\ad_U)$; conversely, $(\fh,\eta,D)$ is sent to $(\fh\rtimes_D\bR,d\eta-\vt\wedge\eta,\vt)$, where $\vt(X,a)=-a$. The exact lcs structure is of the first kind if and only if $\vt(U)=1$ if and only if $\alpha=0$.
\end{theorem*}

Notice that there exist four-dimensional lcs Lie algebras which are not exact, hence do not fall in the hypotheses of the previous theorem. There exist also four-dimensional exact lcs Lie algebras for which the hypothesis $\vt(U)\neq 0$ is not fulfilled, see Section \ref{Sec:4.2}.

Our second result as well displays certain lcs Lie algebras as a semidirect product. More precisely, we consider in Section \ref{another} a lcs Lie algebra $(\fg,\Omega,\vt)$ and write
\begin{equation}\label{eq:mixed}
\Om=\om+\eta\wedge\vt\,,
\end{equation}
for some $\om\in\Lambda^2\fg^*$ and $\eta\in\fg^*$. The non-degeneracy of $\Omega$ provides us with a vector $U\in\fg$ determined by the condition $\imath_U\Omega=-\eta$. We assume that
\[
\imath_V\omega=0 \quad \textrm{and} \quad \imath_U\omega=0\,,
\]
where $V$ is the characteristic vector. We write $\fg=\fh\rtimes_D\bR$ where $\fh\coloneq \ker\vartheta$ with $\vt$ corresponding to the linear map $(X,a)\mapsto a$, and $D$ is given by $\mathrm{ad}_U$. Imposing $d\Om=\vt\wedge\Om$, \eqref{eq:mixed} yields the equations
\[
d^{\fh}\omega=0 \quad \textrm{and} \quad \omega+D^*\omega-d^{\fh}\eta=0\,,
\]
where $d^\fh$ denotes the Chevalley-Eilenberg differential on $\fh$. We can solve the above equations at least under some specific {\itshape Ans\"atze}. For example, assuming $\omega=d^{\fh}\eta$ and $D^*\eta=0$, we are back to Theorem \ref{corr:exact} in case of lcs structures of the first kind. Another possible {\itshape Ansatz} is $d^{\fh}\omega=0$, $d^{\fh}\eta=0$ and $D^*\omega=-\omega$; the first two conditions define a \emph{cosymplectic structure} $(\eta,\omega)$ on $\fh$. If $R$ denotes the Reeb vector of the cosymplectic structure, determined by $\imath_R\omega=0$ and $\imath_R\eta=1$, we obtain the following result (see Proposition \ref{cosymp_lcs}):
\begin{theorem*}
Let $(\fh,\eta,\omega)$ be a cosymplectic Lie algebra of dimension $2n-1$, endowed with a derivation $D$ such that $D^*\omega=\alpha\omega$ for some $\alpha\neq 0$. Then $\fg=\fh\rtimes_D\bR$ admits a natural lcs structure. The Lie algebra $\fg$ is unimodular if and only if $\fh$ is unimodular and $D^*\eta=-\alpha(n-1)\eta+\zeta$ for some $\zeta\in\langle R\rangle^\circ$. If $\fh$ is unimodular then the lcs structure $(\Om,\vt)$ on $\fg$ is not exact.
\end{theorem*}
This result is, up to the authors' knowledge, the first construction of non-exact lcs structures on Lie algebras. Notice that, according to \cite[Corollary 4.3]{AO}, a lcs almost abelian Lie algebra of dimension $\geq 6$ is necessarily of the second kind. A relation between cosymplectic Lie algebras and lcs Lie algebras of the first kind was implicitly discussed in \cite{Marrero_Padron}.

In Section \ref{sec:cotangente-ext} we consider the cotangent extension problem in the lcs setting. As we mentioned above, its symplectic aspect was studied by Ovando, with special emphasis on four-dimensional symplectic Lie algebras; a symplectic Lie algebra is just a $2n$-dimensional Lie algebra $\fs$ with a closed 2-form $\omega\in\Lambda^2\fs^*$ such that $\omega^n\neq 0$. Solutions of this problem in the symplectic case are related to the existence of Lagrangian ideals in $\fs$, {\itshape i.e.} $n$-dimensional ideals $\fh\subset\fs$ such that $\om\big|_{\fh\times\fh}\equiv 0$. In general, Lagrangian ideals play an essential role in the study of symplectic Lie algebras, see \cite{CB}.

Let $\fh$ be a Lie algebra with a closed $1$-form $\hat{\vt}\in\fh^*$; we set $\fg=\fh^*\oplus\fh$ and extend $\hat{\vt}$ to a $1$-form $\vt\in\fg^*$ defined by $\vt(\varphi,X)=\hat{\vt}(X)$. We define $\Om_0\in\Lambda^2\fg^*$ by
\begin{equation}\label{lcs_standard}
\Om_0((\varphi,X),(\psi,Y))\coloneq \varphi(Y)-\psi(X)\,.
\end{equation}
A \emph{solution} of the cotangent extension problem in the lcs context is a Lie algebra structure on $\fg$ such that
\begin{itemize}
\item $\fg$ is an extension $0\longrightarrow\fh^*\longrightarrow\fg\longrightarrow\fh\longrightarrow 0$, where $\fh^*$ is endowed with the structure of an abelian Lie algebra;
\item $(\Omega_0,\vt)$ is a lcs structure on $\fg$, {\itshape i.e.} $d\vt=0$ and $d_\vt\Omega_0=0$.
\end{itemize}
The Lie algebra structure on $\fg$ is encoded in a representation $\rho\colon\fh\to\mathrm{End}(\fh^*)$ and a cocycle $\alpha\in Z^2(\fh,\fh^*)$, by setting
\begin{itemize}
\item $[(\varphi,0),(\psi,0)]_\fg=0$;
\item $[(\varphi,0),(0,X)]_\fg=(-\rho(X)(\varphi),0)$ and
\item $[(0,X),(0,Y)]_\fg=(\alpha(X,Y),[X,Y]_\fh)$.
\end{itemize}
Notice that $\fh^*$ is an abelian ideal contained in $\ker\vt$.
The fact that \eqref{lcs_standard} is a lcs structure yields the following result (compare with Corollary \ref{cor:cotg-ext}):
\begin{theorem*}
 Let $\fh$ be a Lie algebra and let $\vt\in\fh^*$ be a closed 1-form. The Lie algebra structure on $\fg=\fh^*\oplus\fh$ attached to the triple $(\fh,\rho,[\alpha])$, where 
 $\rho\colon\fh\to\mathrm{End}(\fh^*)$ is a representation satisfying
 $$ \rho(X)(\varphi)(Y)-\rho(Y)(\varphi)(X)=d^\fh_\vt\varphi(X,Y) $$
and $[\alpha]\in H^2(\fh,\fh^*)$ satisfies
$$ \alpha(X,Y)(Z)+\alpha(Y,Z)(X)+\alpha(Z,X)(Y)=0 , $$
is a solution to the cotangent extension problem in the locally conformally symplectic context.
\end{theorem*}
The following result relates a special kind of Lagrangian ideals in a lcs Lie algebra with solutions of the cotangent extension problem --- see Proposition \ref{Lag:id}:
\begin{theorem*}
Let $(\fg,\Om,\vt)$ be a $2n$-dimensional lcs Lie algebra with a Lagrangian ideal $\fj\subset\ker\vt$. Then $\fg$ is a solution of the cotangent extension problem.
\end{theorem*}

The results of \cite{Bazzoni_Marrero} deal with lcs Lie algebras such that the characteristic vector is central. However, there exist lcs algebras with trivial center, see Example \ref{ex:2}. In Section \ref{sec:center} we study the center of lcs Lie algebras and characterize it completely in the nilpotent case, see Corollary \ref{center_nilpotent}. We also study the center of reductive lcs Lie algebras, with an eye toward their classification in the subsequent section.

Indeed, in Section \ref{sec:reductive} we turn to reductive lcs Lie algebras. They are all of the first kind (Theorem \ref{theo:1}), so \cite[Theorem 5.9]{Bazzoni_Marrero} applies. It turns out that there are only two of those (Corollary \ref{cor:reductive-4-dim}): either $\fg=\mathfrak{su}_2\oplus\bR$ (see Proposition \ref{prop:su_2+R}) or $\fg=\mathfrak{sl}_2\oplus\bR$ (see Proposition \ref{prop:sl_2+R}). We classify lcs structure on such Lie algebras, up to automorphism. This yields a classification of left-invariant lcs structures on the manifold $S^3\times S^1$ and on every compact quotient of $\widetilde{\mathrm{SL}(2,\bR)}\times\bR$.

Theorem \ref{thm:lcs-4} and Table \ref{table:lcs-4} provide a classification of lcs structures on $4$-dimensional solvable Lie algebras up to automorphism. Computations have been performed with the help of Maple and of Sage \cite{sage}. We show that every structure in the table can be recovered thanks to at least one of the three constructions detailed above, hence obtaining a complete picture of the four-dimensional case.

Let $(\fg,\Omega,\vt)$ be a solvable lcs Lie algebra and let $G$ denote the connected, simply connected solvable Lie group that integrates $\fg$; clearly $G$ is endowed with a left-invariant lcs structure. If there exists a discrete and co-compact subgroup $\Gamma\subset G$, the left-invariant lcs structure on $G$ induces a left-invariant lcs structure on $M=\Gamma\backslash G$ (left-invariance refers here to the lift with respect to the left-translations on the universal cover). In Section \ref{sec:compact}, we construct left-invariant lcs structures on compact quotients of connected simply connected four-dimensional solvable Lie groups and explain how these are related to the structure results for lcs Lie algebras discussed above.

\bigskip

\noindent{\bf Acknowledgements.} The first author is supported by the SIR2014 project RBSI14DYEB ``Analytic aspects in complex and hypercomplex geometry (AnHyC)'', by ICUB Fellowship for Visiting Professor, and by GNSAGA of INdAM. The third author is supported by Project PRIN ``Variet\`a reali e complesse: geometria, topologia e analisi armonica'' and by GNSAGA of INdAM. We are grateful to Vicente Cort\'es for his interesting comments.

\noindent{\itshape Notation.} Throughout the paper, the structure equations for Lie algebras are written following the Salamon notation: {\itshape e.g.}
$$\mathfrak{rh}_{3}=(0,0,-12,0)$$
means that the four-dimensional Lie algebra $\mathfrak{rh}_{3}$ admits a basis $(e_1,e_2,e_3,e_4)$ such that $[e_1,e_2]=e_3$, the other brackets being trivial; equivalently, the dual $\mathfrak{rh}_{3}^\ast$ admits a basis $(e^1, e^2,e^3,e^4)$ such that $de^1=de^2=de^4=0$ and $de^3=-e^1\wedge e^2$. Hereafter, we shorten $e^{12}\coloneq e^1\wedge e^2$.

\section{Structure results for lcs Lie algebras}

In this section we consider different structure results for lcs Lie algebras. In particular, we obtain a quite complete picture for exact lcs Lie algebras. The first two results represent certain lcs Lie algebras $\fg$ as semidirect products $\fh\rtimes_D\bR$, where the Lie algebra $\fh$ is endowed with a certain structure and the derivation $D$ is adapted to the structure. The third result is of a different kind and is related to the existence of Lagrangian ideals in the kernel of the Lee form.

\subsection{Exact lcs Lie algebras}

Let $(\fg,\Om,\vt)$ be an exact lcs Lie algebra and let $\eta\in\fg^*$ be a primitive of $\Om$, namely $\Om=d\eta-\vt\wedge\eta$; moreover, let $U\in\fg$ be determined by $\imath_U\Om=-\eta$. Clearly $\vt(V)=\eta(U)=0$ and we have the following lemma:

\begin{lemma}\label{lem:plane}
In the hypotheses above, 
\begin{itemize}
\item the plane $\la U,V\ra$ is symplectic if and only if $\vt(U)\neq 0$;
\item the lcs structure $(d\eta-\vt\wedge\eta,\vt)$ is of the first kind if and only if $\vt(U)=1$.
\end{itemize}
\end{lemma}
\begin{proof}
For the first claim, simply notice that, by definition, $\vt(U)=\imath_U\imath_V\Om=-\Om(U,V)$. For the second one, we compute
\begin{align*}
L_U\Omega&=d(\imath_U\Omega)+\imath_Ud\Omega=-d\eta+\imath_U(\vt\wedge\Omega)\\[5pt]
&=-d\eta+\vt(U)\Om-\vt\wedge\imath_U\Om=\vt(U)\Omega-(d\eta-\vt\wedge\eta)\\[5pt]
&=(\vt(U)-1)\Om\,,
\end{align*}
completing the proof.
\end{proof}

If $\la U,V\ra$ is symplectic, then the same holds for $\la U,V\ra^\Om$; this is the key observation for the next proposition:

\begin{proposition}
Let $(\fg,\Om,\vt)$ be an exact lcs Lie algebra, $\Om=d_\vt\eta$; write $\fh\coloneq\ker\vt$. Assume that $\vt(U)\neq 0$. Then $\eta$ restricts to a contact form on $\fh$. The contact Lie algebra $(\fh,\eta)$ has derivation $D$ such that $D^*\eta=(1-\vt(U))\eta$ and $\fg\cong \fh\rtimes_D\bR$.
\end{proposition}

\begin{proof}
Since $\vt(U)\neq 0$, then $\la U,V\ra$ is a symplectic plane by Lemma \ref{lem:plane}. As a vector space, $\fh=\la U,V\ra^\Om\oplus \la V\ra$ and $\eta(V)\neq 0$. The restriction of $\Omega$ to $\la U,V\ra^\Om$ coincides with the restriction of $d\eta$ to $\la U,V\ra^\Om$, hence $\eta$ restricts to a contact form on $\fh$. Consider the linear map $D\colon\fh\to\fh$ given by $X\mapsto [U,X]$. Notice that $D$ really maps $\fh$ to $\fh$, since $\fh$ is an ideal because $d\vt=0$, and that it is a derivation, thanks to the Jacobi identity. We claim that $D^*\eta=(1-\vt(U))\eta$. We notice first that
\begin{equation}\label{eq:1}
\imath_Ud\eta=\imath_U(\Om+\vt\wedge\eta)=-\eta+\vt(U)\eta-\eta(U)\vt=(\vt(U)-1)\eta\,.
\end{equation}
For $X\in\fh$, we compute
\[
 (D^*\eta)(X)=\eta([U,X])=-d\eta(U,X)=-(\imath_Ud\eta)(X)\stackrel{\eqref{eq:1}}{=}(1-\vt(U))\eta(X)\,
\]
which proves the first assertion. The isomorphism $\fg\cong\fh\rtimes_D\bR$ is obtained by sending $X$ to $\left(X-\frac{\vt(X)}{\vt(U)}U,\frac{\vt(X)}{\vt(U)}\right)$.
\end{proof}

Let $\fh$ be a Lie algebra endowed with a derivation $D$. Form the semidirect product $\fg=\fh\rtimes_D\bR$ and define $\vt\in\fg^*$ by $\vt(X,a)=-a$. Identify $\eta\in\fh^*$ with the pre-image $\eta\in\fg^*$, under the projection $\fg^*\to\fh^*$, obtained by setting $\eta(X,a)=\eta(X)$. The Chevalley-Eilenberg differentials $d$ on $\fg^*$ and $d^\fh$ on $\fh^*$ are related by the formula
\begin{equation}\label{eq:Chevalley}
d\eta=d^{\fh}\eta-D^*\eta\wedge\vt\,.
\end{equation}

Consider now a contact Lie algebra $(\fh,\eta)$ of dimension $2n-1$ endowed with a derivation $D\colon\fh\to\fh$ such that $D^*\eta=\alpha\eta$ for some $\alpha\neq 1$ and consider $\fg=\fh\rtimes_D\bR$. Extend $\eta$ to an element of $\fg^*$, define $\vt\in\fg^*$ as above and set
\[
\Omega=d\eta-\vt\wedge\eta=d^{\fh}\eta-D^*\eta\wedge\vt-\vt\wedge\eta=d^{\fh}\eta-(1-\alpha)\vt\wedge\eta\,.
\]
Now one has
\[
\Om^n=-(1-\alpha)\eta\wedge (d^\fh\eta)^{n-1}\wedge\vt\neq 0\,,
\]
hence $\Om$ is non-degenerate and we obtain

\begin{proposition}
In the above hypotheses, $(\Om,\vt)$ is an exact lcs structure on $\fg$.
\end{proposition}

We compute now
\[
\imath_{(\xi,0)}\Om=\imath_\xi d^\fh\eta-(1-\alpha)\imath_{(\xi,0)}(\vt\wedge\eta)=(1-\alpha)\vt\,,
\]
hence the characteristic vector of this lcs structure is $V=\left(\frac{1}{1-\alpha}\xi,0\right)$. Moreover,
\[
\imath_{(0,1)}\Om=\imath_{(0,1)}d^\fh\eta-(1-\alpha)\imath_{(0,1)}(\vt\wedge\eta)=(1-\alpha)\eta\,,
\]
hence the symplectic dual of $\eta$ is $U=\left(0,\frac{1}{\alpha-1}\right)$. Notice in particular that $\vt(U)=\frac{1}{1-\alpha}\neq 0$.

Combining the two propositions, we obtain a structure result for exact lcs Lie algebras:
\begin{theorem}\label{corr:exact}
 There is a one-to-one correspondence 
 \[
\left\{\begin{array}{c}
\textrm{exact lcs Lie algebras } (\fg,\Om=d\eta-\vt\wedge\eta,\vt),\\
\textrm{dim }\fg=2n, \textrm{ such that }\vt(U)\neq 0,\\
\textrm{ where }\eta=-\imath_U\Om\end{array}\right\}\leftrightarrow
\left\{\begin{array}{c}
\textrm{contact Lie algebras } (\fh,\eta),\\ 
\textrm{dim }\fh=2n-1, \textrm{ with a derivation}\\
D \textrm{ such that } D^*\eta=\alpha\eta, \alpha\neq 1
\end{array}\right\}
\]
The correspondence sends $(\fg,\Omega,\vt)$ to $(\ker\vt,\eta,\ad_U)$; conversely, $(\fh,\eta,D)$ is sent to $(\fh\rtimes_D\bR,d\eta-\vt\wedge\eta,\vt)$, where $\vt(X,a)=-a$. The exact lcs structure is of the first kind if and only if $\vt(U)=1$ if and only if $\alpha=0$.
\end{theorem}

\begin{example}
Consider the Lie algebra $\mathfrak{d}_{4}=(14,-24,-12,0)$ endowed with the exact lcs structure $\vt=e^4$ and $\Om=d_\vt e^3=-e^{12}+ e^{34}$. Then $U=e_4$ and $\vt(U)=1$, thus $(\Om,\vt)$ is of the first kind by Lemma \ref{lem:plane}. $\mathfrak{h}\coloneq\ker\vt\cong(0,0,-12)$ is isomorphic to the Heisenberg algebra and $\fg\cong\fh\rtimes_D\bR$, where $D=\mathrm{ad}_U\colon \fh\to\fh$ is the derivation
\[
D(e_1)=-e_1, \quad D(e_2)=e_2 \quad \textrm{and} \quad D(e_3)=0.
\]
$\fh$ is endowed with the contact form $\eta\coloneq e^3$ and $D^*\eta=0$.
\end{example}

\begin{example}
Consider the Lie algebra $\mathfrak{d}_{4,1}=(14,0,-12+34,0)$ endowed with the exact lcs structure $\vt=e^4$ and $\Om=d_\vt e^3=-e^{12}+ 2e^{34}$. Then $U=\frac{1}{2}e_4$, hence $\vt(U)=\frac{1}{2}$ and $(\Om,\vt)$ is not of the first kind by Lemma \ref{lem:plane}. $\mathfrak{h}\coloneq\ker \vt\cong(0,0,-12)$ is isomorphic to the Heisenberg algebra and $\fg\cong\fh\rtimes_D\bR$, where $D=\mathrm{ad}_U\colon \fh\to\fh$ is the derivation
\[
D(e_1)=\frac{1}{2}e_1, \quad D(e_2)=0 \quad \textrm{and} \quad D(e_3)=\frac{1}{2}e_3.
\]
$\fh$ is endowed with the contact form $\eta\coloneq e^3$ and $D^*\eta=\frac{1}{2}\eta$.
\end{example}

\begin{example}
Consider the Lie algebra $\fr_2'=(0,0,-13+24,-14-23)$ endowed with the lcs structure $\vt=e^2$ and $\Om=d_\vt(-e^3+e^4)=e^{13}-e^{14}-2e^{24}$. Then $U=e_1$, $V=\frac{e_3+e_4}{2}$, $\vt(U)=0$ and $\la U,V\ra$ is not symplectic. The lcs structure is exact but we can not apply Theorem \ref{corr:exact}. The Lie algebra $\fh=\ker\vt=\la e_1,e_3,e_4\ra$ has structure equations $(0,-13,-14)$ and is not a contact Lie algebra.
\end{example}

\subsection{A ``mixed'' structure result}\label{another}

In this section we consider a partial structure result for lcs Lie algebras which recovers, under certain circumstances, a special case of Theorem \ref{corr:exact}. We begin with a lcs Lie algebra $(\fg,\Om,\vt)$ of dimension $2n$ with characteristic vector $V$. Assume that we can write $\Om=\om+\eta\wedge\vt$ where $\eta\in\fg^*$ and $\om\in\Lambda^2\fg^*$ is such that $\imath_V\omega=0$. Since $\omega$ can not have rank $n$, but $\Om^n\neq 0$, it follows that $\vartheta\wedge\eta\wedge\omega^{n-1}\neq 0$ hence $\eta\wedge\omega^{n-1}\neq 0$. The non-degeneracy of $\Om$ provides us with $U\in\fg$ determined by the condition $\imath_U\Omega=-\eta$. We assume further that $\imath_U\omega=0$. Then $\eta(V)=\vt(U)=1$ and the plane $\la U,V\ra$ is symplectic for $\Om$. Since $d\vt=0$, we can write $\fg$ as a semidirect product $\fh\rtimes_D\bR$, where $\fh=\ker\vt$ and $D\colon\fh\to\fh$ is given by $\ad_U$; under this isomorphism, $U\in\fg$ corresponds to $(0,1)\in\fh\rtimes_D\bR$ and $\vt$ corresponds to the linear map $(X,a)\mapsto a$ (notice the different sign convention with respect to the previous section). According to this identification, the Chevalley-Eilenberg differentials of $\fg$ and $\fh$ are related by the formula
\[
d\zeta=d^{\fh}\zeta+(-1)^{p+1}D^*\zeta\wedge\vt,
\]
where $\zeta\in\Lambda^p\fh^*$ is identified with a pre-image $\zeta\in\Lambda^p\fg^*$ under the projection $\fg^*\to\fh^*$. Of course, $d\vt=0$. Now since $\eta(U)=0$ (respectively $\imath_U\omega=0$), then $\eta$ (respectively $\omega$) can be identified with an element of $\fh^*$ (respectively $\Lambda^2\fh^*$). We denote such elements again by $\eta$ and $\omega$. We have the following chain of equalities:
\[
\vt\wedge\om=\vt\wedge\Om=d\Om=d(\om+\eta\wedge\vt)=d^\fh\om-D^*\om\wedge\vt+d^\fh\eta\wedge \vt\,.
\]
This implies
\[
d^\fh\om=0\quad \textrm{and} \quad \om+D^*\om-d^\fh\eta=0\,.
\]
To solve these equations on $\fh$ we can make different \emph{Ans\"atze}:
\begin{enumerate}
\item $\omega=d^\fh\eta$ and $d^\fh(D^*\eta)=0$; then $(\fh,\eta)$ is a contact Lie algebra endowed with a derivation $D\colon\fh\to\fh$ such that $D^*\eta$ is closed. As a special case of this instance, one can consider $D^*\eta=0$; then $\vt(U)=1$ implies that we are in the context of Theorem \ref{corr:exact} when the lcs structure is of the first kind;
\item $d^\fh\omega=0$, $d^\fh\eta=0$ and $D^*\omega=-\omega$; this leads to another kind of structure. In fact, the closedness of $\om$ and $\eta$ in $\fh$, together with $\eta\wedge\om^{n-1}\neq 0$, imply that $(\fh,\eta,\omega)$ is a \emph{cosymplectic} Lie algebra, endowed with a derivation $D\colon\fh\to\fh$ such that $D^*\omega=-\omega$.
\end{enumerate}

Recall that a cosymplectic structure $(\eta,\omega)$ on a Lie algebra $\fh$ of dimension $2n-1$ determines a vector $R\in\fh$ by the conditions $\imath_R\om=0$ and $\eta(R)=1$; moreover, there is a decomposition 
\begin{equation}\label{eq:1789}
\fh^*=\la\eta\ra\oplus \langle R\rangle^\circ\,,
\end{equation}
where $\langle R\rangle^\circ$ denotes the annihilator of $\langle R\rangle$, and the linear map $\_\wedge\om^{n-1}\colon\fh^*\to\Lambda^{2n-1}\fh^*$ is non-zero on $\la\eta\ra$, hence its kernel coincides with $\langle R\rangle^\circ$.

The second \emph{Ansatz} provides a kind of alternative structure result for lcs Lie algebras, corroborated by the following

\begin{proposition}\label{cosymp_lcs}
Let $(\fh,\eta,\omega)$ be a cosymplectic Lie algebra of dimension $2n-1$, endowed with a derivation $D$ such that $D^*\omega=\alpha\omega$ for some $\alpha\neq 0$. Then $\fg=\fh\rtimes_D\bR$ admits a natural lcs structure. The Lie algebra $\fg$ is unimodular if and only if $\fh$ is unimodular and $D^*\eta=-\alpha(n-1)\eta+\zeta$ for some $\zeta\in\langle R\rangle^\circ$. If $\fh$ is unimodular then the lcs structure $(\Om,\vt)$ on $\fg$ is not exact.
\end{proposition}

\begin{proof}
Set $\fg=\fh\rtimes_D\bR$ and define $\vt(X,a)=-\alpha a$; with respect to this choice of $\vt$, the formula 
\[
d\zeta=d^\fh\zeta+\frac{(-1)^p}{\alpha}D^*\zeta\wedge\vt
\]
relates the Chevalley-Eilenberg differentials on $\fg$ and $\fh$ for a $p$-form $\zeta$.
Setting $\Omega=\om+\eta\wedge \vt$, we see that $\Om^n=\om^{n-1}\wedge\eta\wedge\vt\neq 0$, hence $\Omega$ is non-degenerate. Moreover,
\[
d\Omega=d^\fh\omega+\frac{1}{\alpha}D^*\omega\wedge \vt=\om\wedge\vt=(\om+\eta\wedge\vt)\wedge\vt=\vt\wedge\Omega\,.
\]

An $n$-dimensional Lie algebra $\fk$ is unimodular if and only if $d(\Lambda^{n-1}\fk^*)=0$, {\itshape i.e.} if and only if a generator of $\Lambda^n\fk^*$ is not exact. Now $\omega^{n-1}\wedge\eta\wedge\vt$ generates $\Lambda^{2n}\fg^*$ and one has $\Lambda^{2n-1}\fg^*\cong\la\omega^{n-1}\wedge\eta\ra\oplus\Lambda^{2n-2}\fh^*\wedge\vt$. By hypothesis we have $D^*\omega=\alpha\omega$ and we can decompose $D^*\eta$ according to \eqref{eq:1789}, $D^*\eta=\beta\eta+\zeta$ for some $\beta\in\bR$ and $\zeta\in\langle R\rangle^\circ$. Now
\begin{align*}
d(\omega^{n-1}\wedge\eta)&=d^\fh(\omega^{n-1}\wedge\eta)-\frac{1}{\alpha}D^*(\omega^{n-1}\wedge\eta)\wedge \vt\\
&=-\frac{1}{\alpha}(\alpha(n-1)\omega^{n-1}\wedge\eta+\omega^{n-1}\wedge D^*\eta)\wedge\vt\\
&=-\frac{1}{\alpha}(\alpha(n-1)+\beta)\omega^{n-1}\wedge\eta\wedge \vt\,.
\end{align*}
This vanishes if and only if $\beta=-\alpha(n-1)$. If $\kappa\in\Lambda^{2n-2}\fh^*$, then
\[
d(\kappa\wedge\vt)=(d^\fh\kappa)\wedge\vt
\]
which vanishes if and only if $d^\fh\kappa=0$; since $\kappa$ is arbitrary, this happens if and only if $\fh$ is unimodular.

If $\fh$ is unimodular then $\omega$ can not be exact. If the lcs structure $(\Omega,\vt)$ is exact, there exists $\upsilon\in\fh^*$ with $\Omega=d\upsilon+\upsilon\wedge\vt=d^\fh\upsilon+\left(-\frac{1}{\alpha}D^*\upsilon+\upsilon\right)\wedge\vt=\omega+\eta\wedge\theta$. This is impossible, since it would imply that $\omega$ is exact.
\end{proof}

\begin{remark}
According to \cite[Proposition 10]{BFM}, cosymplectic Lie algebras $(\fh,\eta,\omega)$ in dimension $2n-1$ are in one-to-one correspondence with symplectic Lie algebras $(\fs,\omega)$ in dimension $2n-2$, endowed with a derivation $E$ such that $E^*\omega=0$. The correspondence is given by $(\fh,\eta,\omega)\mapsto(\ker\eta,\omega,\ad_R)$ and $(\fs,\omega,E)\mapsto (\fs\rtimes_E\bR,\eta,\omega)$, where $\eta$ generates the $\bR$-factor.
In principle one could use the above proposition to establish a link between non-exact lcs and symplectic Lie algebras.
\end{remark}

\begin{example}\label{ex:1}
We consider the abelian Lie algebra $\bR^3=\la f_1,f_2,f_3\ra$ endowed with the cosymplectic structures $\eta=f^3$, $\omega_\pm=\pm f^{12}$. For $\gamma>0$ we consider the derivation
\[
D=\begin{pmatrix}
\gamma & 1 & 0\\ -1 & \gamma & 0\\ 0 & 0 & 0
\end{pmatrix}\colon\bR^3\to\bR^3\,,
\]
which satisfies $D^*\eta=0$ and $D^*\omega_\pm=2\gamma\omega_\pm$. The Lie algebra $\fg=\bR^3\rtimes_D\bR$ is endowed with the lcs structure $(\Om,\vt)=(\pm f^{12}+f^{34},-2\gamma f^4)$, where $f^4$ generates the $\bR$-factor. The lcs structure is not exact and it is easy to see that $\fg$ is isomorphic to $\fr\fr'_{3,\gamma}$ (see Table \ref{algebreLie}).
\end{example}

\begin{example}\label{ex:2}
The abelian Lie algebra $\bR^3=\la e_1,e_2,e_3\ra$ is endowed with the cosymplectic structure $\eta=e^1$, $\omega=e^{23}$. For $\gamma\in\bR$ and $\delta>0$ we consider the derivation 
\[
D=\begin{pmatrix}
1 & 0 & 0\\ 0 & \gamma & \delta\\ 0 & -\delta & \gamma
\end{pmatrix}\colon\bR^3\to\bR^3
\]
which satisfies $D^*\eta=\eta$ and $D^*\omega=2\gamma\omega$. We assume that $\gamma\neq 0$. The Lie algebra $\fg=\bR^3\rtimes_D\bR$ is isomorphic to $\fr'_{4,\gamma,\delta}$ (see Table \ref{algebreLie}), which is endowed with the lcs structure $(\Om,\vt)=(e^{14}+e^{23},-2\gamma e^4)$. The lcs structure is not exact.
\end{example}

\subsection{Cotangent extensions and Lagrangian ideals}\label{sec:cotangente-ext}

In this section we extend to the locally conformally symplectic setting the  cotangent extension problem \cite{Boyom} studied by Ovando for four-dimensional symplectic Lie algebras in \cite{Ovando}. Let $\fh$ be a Lie algebra and let $\fh^*$ be its dual vector 
space. Consider the skew-symmetric 2-form $\Omega_0$ on $\fh^*\oplus\fh$, defined by
\begin{equation}\label{symplectic}
 \Om_0((\varphi,X),(\psi,Y))=\varphi(Y)-\psi(X).
\end{equation}
In the symplectic context, the \emph{cotangent extension problem} consists in finding a Lie algebra structure on $\fh^*\oplus\fh$ such that
\begin{itemize}
\item $0\longrightarrow\fh^*\longrightarrow\fh^*\oplus\fh\longrightarrow\fh\longrightarrow 0$ is a short exact sequence of Lie algebras, where $\fh^*$ is endowed with the structure of an abelian Lie algebra;
\item the 2-form $\Om_0$ defined in \eqref{symplectic} is closed.
\end{itemize}

Suppose $\rho\colon\fh\to\mathrm{End}(\fh^*)$ is a Lie algebra representation and define the skew-symmetric map $[ \ , \ ]\colon\fg\times\fg\to\fg$ on $\fg=\fh^*\oplus\fh$ by setting
\begin{itemize}
 \item $[(\varphi,0),(\psi,0)]_\fg=0$,
 \item $[(\varphi,0),(0,X)]_\fg=(-\rho(X)(\varphi),0)$ and
 \item $[(0,X),(0,Y)]_\fg=(\alpha(X,Y),[X,Y]_\fh)$,
\end{itemize}
where $\alpha\in C^2(\fh,\fh^*)$ is a 2-cochain (the module structure on $\fh^*$ is clearly given by $\rho$). Then $[ \ , \ ]$ is a Lie bracket if and only $\alpha\in Z^2(\fh,\fh^*)$. In this case, 
$\fh^*\subset\fg$ is an ideal and we have a short exact sequence
\[
 0\longrightarrow\fh^*\longrightarrow\fg\longrightarrow\fh\longrightarrow 0.
\]
The 2-form \eqref{symplectic} is closed if and only if
\begin{eqnarray}
 \alpha(X,Y)(Z)+\alpha(Y,Z)(X)+\alpha(Z,X)(Y)=0\label{conditions_1}\\
 \rho(X)(\varphi)(Y)-\rho(Y)(\varphi)(X)=-\varphi([X,Y]_\fh)\label{conditions_2}
\end{eqnarray}

Hence the Lie algebra $\fg$ attached to the triple $(\fh,\rho,[\alpha])$, satisfying \eqref{conditions_1} and \eqref{conditions_2}, where $[\alpha]$ is the class of $\alpha$ in $H^2(\fh,\fh^*)$, is a solution of the cotangent extension problem. 
In \cite{Ovando}, the author proves:
\begin{theorem}[{\cite[Theorem 3.6]{Ovando}}]\label{thm:Ovando}
 Let $\fg$ be a $2n$-dimensional Lie algebra.
 \begin{itemize}
  \item If $\fj\subset\fg$ is an abelian ideal of dimension $n$, then $\fg$ is a solution of the cotangent extension problem if and only if \eqref{conditions_1} and \eqref{conditions_2} are satisfied.
  \item If $\fg$ is a symplectic Lie algebra and $\fj\subset\fg$ is a Lagrangian ideal, then $\fg$ is a solution of the cotangent extension problem.
 \end{itemize}
\end{theorem}

The cotangent extension problem is related to the fact that the cotangent bundle of any smooth manifold has a canonical symplectic structure. However, the simply connected Lie group $G$ is the 
cotangent bundle of the simply connected Lie group $H$ if and only if $H$ is abelian (see \cite[Remark 3.2]{Ovando}).

The cotangent bundle of any smooth manifold has a locally conformally symplectic structure. In fact, suppose $M$ is a smooth manifold, let $\hat{\vt}\in\Omega^1(M)$ be a closed 1-form and let $\pi\colon T^*M\to M$ be the natural projection. Let $\lambda^{can}$ denote the 
canonical 1-form on $T^*M$, given by $\lambda^{can}_{(p,\varphi)}(v)=\varphi(d\pi_{(p,\varphi)}(v))$ for $(p,\varphi)\in T^*M$ and $v\in T_{(p,\varphi)}(T^*M)$. Then $\Om\coloneq d\lambda^{can}-\vt\wedge\lambda^{can}$ defines a locally conformally symplectic structure on $T^*M$ whose Lee form is $\vt=\pi^*\hat{\vt}$. In fact, as neatly explained in \cite{Vaisman}, locally conformally symplectic manifolds are natural phase spaces of Hamiltonian dynamics.

Motivated by these speculations, we consider a Lie algebra $\fh$ with a closed element $\hat{\vt}\in\fh^*$ and set $\fg=\fh^*\oplus\fh$. We extend $\hat{\vt}$ to an element $\vt\in\fg^*$ by setting $\vt(\varphi,X)=\hat{\vt}(X)$ and define a 2-form $\Om_0$ on $\fg$ precisely as in \eqref{symplectic}.

In the locally conformally symplectic context, a solution of the \emph{cotangent extension problem} is a Lie algebra structure on $\fg$ such that
\begin{itemize}
\item $0\longrightarrow\fh^*\longrightarrow\fg\longrightarrow\fh\longrightarrow 0$ is a short exact sequence of Lie algebras, $\fh^*$ endowed with the structure of an abelian Lie algebra;
\item the 1-form $\vt$ is closed and the 2-form $\Om_0$ defined in \eqref{symplectic} satisfies $d\Om_0=\vt\wedge\Om_0$.
\end{itemize}

Given a representation $\rho\colon\fh\to\mathrm{End}(\fh^*)$ and a cochain $\alpha\in C^2(\fh,\fh^*)$, we define a skew-symmetric bilinear map $[ \ , \ ]\colon\fg\times\fg\to\fg$ as we did above. 
Then $[ \ , \ ]$ is a Lie bracket on $\fg$ if and only if $\alpha\in Z^2(\fh,\fh^*)$. Assuming this, we have a short exact sequence $0\to\fh^*\to\fg\to\fh\to 0$ and $\fh^*$ is an abelian ideal.

\begin{lemma}
 $\Om_0$ satisfies $d\Om_0=\vt\wedge\Om_0$ if and only if
 \begin{eqnarray}
 \alpha(X,Y)(Z)+\alpha(Y,Z)(X)+\alpha(Z,X)(Y)=0\label{conditions_3}\\
 \rho(X)(\varphi)(Y)-\rho(Y)(\varphi)(X)=d^\fh_\vt\varphi(X,Y)\label{conditions_4}
\end{eqnarray}
\end{lemma}
\begin{proof}
 To save space, we write simply $\varphi$ (respectively $X$) instead of $(\varphi,0)$ (respectively $(0,X)$). However, $(\varphi,X)$ remains the same. We have
 \begin{align*}
  d\Om_0(X,Y,Z) = & -\Om_0([X,Y]_\fg,Z)-\Om_0([Y,Z]_\fg,X)-\Om_0([Z,X]_\fg,Y)\\[5pt]
  = &-\Om_0((\alpha(X,Y),[X,Y]_\fh),Z)-\Om_0((\alpha(Y,Z),[Y,Z]_\fh),X)\\[5pt]
  & -\Om_0((\alpha(Z,X),[Z,X]_\fh),Y)\\[5pt]
  = & -\alpha(X,Y)(Z)-\alpha(Y,Z)(X)-\alpha(Z,X)(Y)
 \end{align*}
and
\[
 (\vt\wedge\Om_0)(X,Y,Z)=\vt(X)\Om_0(Y,Z)+\vt(Y)\Om_0(Z,X)+\vt(Z)\Om_0(X,Y)=0,
\]
which imply \eqref{conditions_3}. Next,
 \begin{align*}
  d\Om_0(\varphi,X,Y) & = -\Om_0(-\rho(X)(\varphi),Y)-\Om_0([X,Y]_\fg,\varphi)-\Om_0(\rho(Y)(\varphi),X)\\[5pt]
  & = \rho(X)(\varphi)(Y)-\rho(Y)(\varphi)(X)-\Om_0((\alpha(X,Y),[X,Y]_\fh),\varphi)\\[5pt]
  & = \rho(X)(\varphi)(Y)-\rho(Y)(\varphi)(X)+\varphi([X,Y]_\fh)\\[5pt]
  & = \rho(X)(\varphi)(Y)-\rho(Y)(\varphi)(X)-d^\fh\varphi(X,Y),
 \end{align*}
 and
 \begin{align*}
 (\vt\wedge\Om_0)(\varphi,X,Y) & =\vt(\varphi)\Om_0(X,Y)+\vt(X)\Om_0(Y,\varphi)+\vt(Y)\Om_0(\varphi,X)=\\
 & = -\vt(X)\varphi(Y)+\vt(Y)\varphi(X)=\\
 & = -(\vt\wedge\varphi)(X,Y).
 \end{align*}
 Hence $d\Om_0(\varphi,X,Y)=(\vt\wedge\Om_0)(\varphi,X,Y)$ implies \eqref{conditions_4}. Clearly \eqref{conditions_3} and \eqref{conditions_4} are also sufficient.
\end{proof}

\begin{corollary}\label{cor:cotg-ext}
 Let $\fh$ be a Lie algebra and let $\vt\in\fh^*$ be a closed 1-form. The Lie algebra structure on $\fg=\fh^*\oplus\fh$ attached to the triple $(\fh,\rho,[\alpha])$, where 
 $\rho\colon\fh\to\mathrm{End}(\fh^*)$ is a representation satisfying \eqref{conditions_4} and $[\alpha]\in H^2(\fh,\fh^*)$ satisfies \eqref{conditions_3}, is a solution to the cotangent extension problem in the locally conformally symplectic context.
\end{corollary}

\begin{remark}
 The formul\ae \ for the Lie bracket on the Lie algebra $\fg=\fh^*\oplus\fh$ given above show that $\fh^*$ sits in $\fg$ as an abelian ideal. Moreover, this ideal is by construction contained in $\ker\vt$. We can then sum up what we said so far in the following way: {\itshape a $2n$-dimensional Lie algebra $\fg$ endowed with a closed $\vt\in\fg^*$ and an $n$-dimensional abelian ideal $\fj\subset\ker\vt$ is a solution of the cotangent extension problem if and only if \eqref{conditions_3} and \eqref{conditions_4} hold}. This is the equivalent, in the lcs context, of the first statement of Theorem \ref{thm:Ovando}.
\end{remark}

We show now that solutions of the cotangent extension problem in the lcs case are related to the existence of a special kind of Lagrangian ideals.

\begin{lemma}\label{Lag:ab}
Let $(\fg,\Om,\vt)$ be a lcs Lie algebra. If $\fj\subset\fg$ is a Lagrangian ideal contained in $\ker\vt$, then $\fj$ is abelian.
\end{lemma}
\begin{proof}
For $X,Y\in\fj$ and $Z\in\fg$ we compute
\begin{align*}
\Om([X,Y],Z)&=-d\Om(X,Y,Z)+\Om([X,Z],Y)-\Om([Y,Z],X)\\[5pt]
&=-(\vt\wedge\Om)(X,Y,Z)=-\vt(X)\Om(Y,Z)+\vt(Y)\Om(X,Z)-\vt(Z)\Om(X,Y)\\[5pt]
&=0\,,
\end{align*}
concluding the proof.
\end{proof}

\begin{proposition}\label{Lag:id}
Let $(\fg,\Om,\vt)$ be a $2n$-dimensional lcs Lie algebra with a Lagrangian ideal $\fj\subset\ker\vt$. Then $\fg$ is a solution of the cotangent extension problem.
\end{proposition}
\begin{proof}
Being contained in $\ker\vt$, $\fj$ is an abelian ideal by Lemma \ref{Lag:ab}; moreover, we have a short exact sequence of Lie algebras 
\begin{equation}\label{ses:1}
0\to\fj\to\fg\to\fh\to 0\,,
\end{equation}
where $\fh\coloneq\fg/\fj$. Since $\fj$ is Lagrangian, the non-degeneracy of $\Om$ identifies it with $\fh^*$. More precisely, the linear map $\sigma\colon\fj\to\fh^*$, $X\mapsto\imath_X\Om$ is injective, hence an isomorphism by dimension reasons. Choose a Lagrangian splitting $\fg=\fj\oplus\fk$ (such a splitting always exists, see \cite[Lecture 2]{Weinstein} for a proof). We use $\sigma$ and the isomorphism $\fh\cong\fk$ to get an isomorphism $\Sigma\colon \fg=\fj\oplus\fk\to\fh^*\oplus\fh$. Clearly $\fh^*$ sits in $\fg$ as an abelian Lie algebra. Moreover, dualizing \eqref{ses:1} we see that $\vt\in\fg^*$ comes from a closed element $\hat{\vt}\in\fh^*$. We endow $\fh^*\oplus\fh$ with the skew-symmetric form $\Om_0$ on $\fh^*\oplus\fh$ defined by \eqref{symplectic} and the closed 1-form $\vt_0$ obtained from $\hat{\vt}$ by setting equal to zero on $\fh^*$. We claim that $\Sigma$ provides an isomorphism of lcs Lie algebras between $(\fg,\Om,\vt)$ and $(\fh^*\oplus\fh,\Om_0,\vt_0)$. Indeed,
\[
(\Sigma^*\vt_0)(X,0)=\vt_0(\Sigma(X,0))=\vt_0(\sigma(X),0)=0=\vt(X,0)
\]
and
\[
(\Sigma^*\vt_0)(0,Y)=\vt_0(\Sigma(0,Y))=\hat{\vt}(Y)=\vt(0,Y)\,.
\]
Moreover,
\[
(\Sigma^*\Om_0)((X,0),(Y,0))=\Om_0(\Sigma(X,0),\Sigma(Y,0))=0=\Om((X,0),(Y,0))\,,
\]
\begin{align*}
(\Sigma^*\Om_0)((X,0),(0,Y))&=\Om_0(\Sigma(X,0),\Sigma(0,Y))=\Om_0((\imath_X\Om,0)(0,Y))=\Om(X,Y)\\
&=\Om((X,0),(0,Y))
\end{align*}
and
\[
(\Sigma^*\Om_0)((0,X),(0,Y))=\Om_0(\Sigma(0,X),\Sigma(0,Y))=0=\Om((0,X),(0,Y))\,,
\]
concluding the proof.
\end{proof}

\section{The center of a lcs Lie algebra}\label{sec:center}

In this section we obtain some results on the center of a lcs Lie algebra. In particular, we characterize the center of nilpotent lcs Lie algebras. The results of this section will also play a role in the next section.

\begin{lemma}\label{center:contact}
Let $(\fh,\eta)$ be a contact Lie algebra. Then $\dim \cZ(\fh)\leq 1$, with equality if and only if $\cZ(\fh)$ is spanned by the Reeb vector.
\end{lemma}
\begin{proof}
We assume that $\cZ(\fh)\neq 0$ and pick a central vector $X$. Then:
\begin{align*}
(\imath_Xd\eta)(Y)=d\eta(X,Y)=-\eta([X,Y])=0\,.
\end{align*}
$\eta$ being a contact form, the radical of $d\eta$ is spanned by the Reeb vector $\xi$, hence $X\in\langle\xi\rangle$.
\end{proof}

\begin{proposition}\label{Proposition:1}
 If $(\fg,\Om,\vt)$ is a lcs Lie algebra, then $\cZ(\fg)\subset\fg_\Om$. Moreover, if $\cZ(\fg)\not\subset\ker\vt$ then $\fh=\ker\vt$ is endowed with a contact structure, $(\Om,\vt)$ is of the first kind, $\fg$ splits as $\fh\oplus\bR$ and $\dim\cZ(\fg)\leq 2$.
\end{proposition}
\begin{proof}
 We prove first that $\cZ(\fg)\subset\fg_\Om$; pick $Z\in\cZ(\fg)$, then $[Z,X]=0$ for every $X\in\fg$ and we have:
 \begin{align}\label{eq:3}
  0&=(\imath_{[Z,X]}\Om)(Y)=-d\Om(Z,X,Y)-\Om([X,Y],Z)+\Om([Z,Y],X)\nonumber\\
   &=-(\vt\wedge\Om)(Z,X,Y)+(\imath_Z\Om)([X,Y])\nonumber\\
   &=-\vt(Z)\Om(X,Y)+\vt(X)\Om(Z,Y)-\vt(Y)\Om(Z,X)-d(\imath_Z\Om)(X,Y)\nonumber\\
   &=-\vt(Z)\Om(X,Y)-d(\imath_Z\Om)(X,Y)+(\vt\wedge\imath_Z\Om)(X,Y)\nonumber\\
   &=(-\imath_Z(\vt\wedge\Om)-d(\imath_Z\Om))(X,Y)=(-\imath_Z(d\Om)-d(\imath_Z\Om))(X,Y)\nonumber\\
   &=-(L_Z\Om)(X,Y)\,.
 \end{align}
If $\cZ(\fg)\not\subset\ker\vt$ then $\fg_\Om\not\subset\ker\vt$, hence we find $U\in\cZ(\fg)\subset\fg_\Om$ with $\vt(U)=1$ and the lcs is of the first kind. By \thmref{corr:exact}, $\fg\cong\fh\rtimes_D\bR$, where $\fh=\ker\vt$ and $D=\mathrm{ad}_U$. Since $U$ is central $D$ is the trivial derivation and $\fg\cong\fh\oplus\bR$. Moreover $\fh$ is a contact Lie algebra, hence $\dim \cZ(\fh)\leq 1$ by Lemma \ref{center:contact} and $\dim\cZ(\fg)\leq 2$.
\end{proof}

\begin{lemma}\label{Lemma:101}
 Let $(\fg,\Om,\vt)$ be a lcs Lie algebra with $\fg_\Om\subset\ker\vt$. Then the isomorphism $\Theta\colon\fg\to\fg^*$, $\Theta(X)=\imath_X\Om$, injects $\fg_\Om$ into $Z^1_\vt(\fg)=\{\alpha\in\fg^* \mid d_\vt\alpha=0\}$ and the kernel of the composition $\fg_\Om\xrightarrow{\Theta} Z^1_\vt(\fg)\to H^1_\vt(\fg)$ is generated by the characteristic vector.
\end{lemma}
\begin{proof}
Given $X\in\fg_\Om$, consider $\imath_X\Om\in\fg^*$. We compute
\[
d_\vt(\imath_X\Om)=d(\imath_X\Om)-\vt\wedge\imath_X\Om=-\imath_X(d\Om)-\vt\wedge\imath_X\Om=-\imath_X(\vt\wedge\Om)-\vt\wedge\imath_X\Om=0\,,
\]
hence $\imath_X\Om\in Z^1_\vt(\fg)$ and clearly $\Theta(V)=\vt$. Since $d_\vt\colon\Lambda^0\fg^*\to\Lambda^1\fg^*$ maps $1$ to $-\vt$, $\vt$ is the only $d_\vt$-exact element in $\fg^*$.
\end{proof}

\begin{example}
Consider the Lie algebra $\fr\fr_3=(0,-12-13,-13,0)$ endowed with the lcs structure $(\Om,\vt)=(e^{12}+e^{34},-e^1)$. The characteristic vector is $V=e_2$, we have $\cZ(\fr\fr_3)=\langle e_4\rangle$, $(\fr\fr_3)_\Om=\langle e_2,e_4\rangle$, $\ker\vt=\langle e_2,e_3,e_4\rangle$ and we get a strict sequence of inclusions 
\[
0\subset\cZ(\fr\fr_3)\subset(\fr\fr_3)_\Om\subset\ker\vt\subset\fr\fr_3\,.
\]
\end{example}

\begin{proposition}
Let $(\fg,\Om,\vt)$ be a lcs Lie algebra with $0\neq\cZ(\fg)\subset\ker\vt$. If $V\notin\cZ(\fg)$, then $\fg$ is solvable non-nilpotent.
\end{proposition}
\begin{proof}
Pick $Z\in \cZ(\fg)$; then $Z\in\fg_\Om$ by Proposition \ref{Proposition:1} and $[\Theta(Z)]\in H^1_\vt(\fg)\neq 0$ by Lemma \ref{Lemma:101}. Due to standard results in Lie algebra cohomology, this is impossible if $\fg$ is nilpotent (see \cite[Th\'eor\`eme 1]{Dixmier}) or semisimple (see \cite[Corollary 7.8.10]{Weibel}), hence also if $\fg$ 
is reductive. Then $\fg$ must be solvable non-nilpotent.
\end{proof}

\begin{corollary}\label{center_nilpotent}
Let $(\fg,\Om,\vt)$ be a nilpotent lcs Lie algebra. Then either $\cZ(\fg)=\la V\ra$, or $\cZ(\fg)=\la U,V\ra$ and $\fg\cong\fh\oplus\bR$, where $\fh$ is a contact Lie algebra.
\end{corollary}
\begin{proof}
Since $H^2_\vt(\fg)=0$ on a nilpotent Lie algebra, $(\Om,\vt)$ is exact. Moreover, $\fg$ is unimodular and $(\Om,\vt)$ is of the first kind (compare with \cite[Proposition 5.5]{Bazzoni_Marrero}). Also, $\cZ(\fg)\neq 0$ since $\fg$ is nilpotent, hence $1\leq \dim\cZ(\fg)\leq 2$ by Proposition \ref{Proposition:1}. If $\cZ(\fg)\subset\fh\coloneq\ker\vt$ then $\cZ(\fg)$ is contained in the center of $\fh$, a nilpotent contact Lie algebra, hence $\dim\cZ(\fg)=\dim\cZ(\fh)=1$ by Lemma \ref{center:contact}. Since $\fg$ is nilpotent, $H^1_\vt(\fg)=0$ and then $\cZ(\fg)=\langle V\rangle$. Otherwise, again by Proposition \ref{Proposition:1}, $\cZ(\fg)=\la U,V\ra$ and $\fg=\fh\oplus\bR$.
\end{proof}

\begin{remark}
There exist lcs solvable Lie algebras with trivial center, see for instance Example \ref{ex:2}.
\end{remark}

\section{Reductive lcs Lie algebras}\label{sec:reductive}

If a reductive Lie group $G$ is endowed with a left-invariant locally conformally pseudo-K\"ahler structure, then $\fg$ is isomorphic to $\mathfrak{u}_2=\mathfrak{su}_2\oplus\bR$ or $\mathfrak{gl}_2(\bR)=\mathfrak{sl}_2(\bR)\oplus\bR$, see \cite[Theorem 4.15]{ACHK}, and all such structures are classified. In this section we generalize this result to left-invariant locally conformally symplectic structures: we prove that the only reductive lcs Lie algebras are $\mathfrak{u}_2=\mathfrak{su}_2\oplus\bR$ and $\mathfrak{gl}_2(\bR)=\mathfrak{sl}_2(\bR)\oplus\bR$ and classify such lcs structures (a classification was already obtained in \cite[Propositions 4.5 and 4.9]{ACHK}).

\begin{theorem}\label{theo:1}
Let $\fg$ be a reductive Lie algebra endowed with a lcs structure $(\Om,\vt)$. Then $\fh=\ker\vt$ is a semisimple Lie algebra endowed with a contact form $\eta$, $\fg=\fh\oplus\bR$ and $\Omega=d\eta-\vt\wedge\eta$. In particular $(\Om,\vt)$ is of the first kind.
\end{theorem}

\begin{proof}
We notice first that $\fg$ cannot be semisimple: if it were so, then $b_1(\fg)$ would vanish, contradicting the fact that $\vt$ is a nontrivial $1$-cocycle. Then we can write $\fg=\fs\oplus\cZ(\fg)$ with $\dim\cZ(\fg)\geq 1$. The Lie algebra structure is given by
\[
 [(S_1,Z_1),(S_2,Z_2)]=([S_1,S_2],0), \quad (S_j,Z_j)\in\fs\oplus\cZ(\fg), \ j\in\{1,2\}.
\]
We compute
\[
 0=d\vt((S_1,Z_1),(S_2,Z_2))=-\vt([(S_1,Z_1),(S_2,Z_2)])=-\vt([S_1,S_2],0),
\]
hence $[\fs,\fs]=\fs\subset\fh\coloneq\ker\vt$ and $\vt\in\cZ(\fg)^*$.
We pick a vector $U\in\cZ(\fg)$ with $\vt(U)=1$. We apply Proposition \ref{Proposition:1} and conclude that $\fg=\fh\oplus \la U\ra$, the lcs structure $(\Om,\vt)$ is of the first kind, and $\eta=-\imath_U\Om$ is a contact form on $\fh$.
Since $(\Om,\vt)$ is of the first kind, $\cZ(\fg)$ has dimension $\leq 2$ again by Proposition \ref{Proposition:1}, hence $1\leq \dim \cZ(\fg)\leq 2$. If $\dim\cZ(\fg)=1$, then $\fs\subset\fh$ implies $\fs=\fh$, hence $(\fh,\eta)$ is a semisimple contact Lie algebra. Assume $\dim\cZ(\fg)=2$. Again $\fs\subset\fh$ implies that $\fg=\fs\oplus\cZ(\fg)$ 
induces a splitting $\fh=\fs\oplus\cZ(\fh)$, where $\cZ(\fh)=\cZ(\fg)\cap\fh$ is the center of $(\fh,\eta)$, hence $\cZ(\fh)=\langle\xi\rangle$, where $\xi$ is the Reeb vector.
Moreover, $\la\eta\ra\cap\cZ(\fh)^*\neq 0$, hence they coincide for dimension reasons. The Lie algebra structure on $\fh$ is then
\[
[(S_1,a_1\xi),(S_2,a_2\xi)]=([S_1,S_2],0)
\]
Now $d\eta$ must be non-degenerate on $\fs$; however,
\[
d\eta((S_1,0),(S_2,0))=-\eta([S_1,S_2],0)=0.
\]
Hence $\cZ(\fg)$ is 1-dimensional and $(\fh,\eta)$ is a semisimple contact Lie algebra.
\end{proof}

\begin{corollary}\label{cor:reductive-4-dim}
 Let $(\fg,\Om,\vt)$ be a reductive lcs Lie algebra. Then either $\fg=\mathfrak{su}_2\oplus\bR$ or $\fg=\mathfrak{sl}_2(\bR)\oplus\bR$.
\end{corollary}
\begin{proof}
By \cite[Theorem 5]{Boothby-Wang}, a semisimple Lie group with a left-invariant contact structure is locally isomorphic either to $\mathrm{SU}(2)$ or to $\mathrm{SL}(2;\bR)$. Hence 
a semisimple Lie algebra with a contact structure is isomorphic either to $\mathfrak{su}_2$ or to $\mathfrak{sl}_2$. By \thmref{theo:1}, $\fg=\fh\oplus\bR$, with $\fh$ a semisimple contact Lie 
algebra. We conclude that $\fh\cong\mathfrak{su}_2$ or $\fh\cong\mathfrak{sl}_2$.
\end{proof}

\begin{remark}
Note that, by a result of Chu, \cite[Theorem 9]{Chu}, a four-dimensional symplectic Lie algebra must be solvable. However, there exist by Corollary \ref{cor:reductive-4-dim} four-dimensional reductive lcs Lie algebras - interestingly enough, dimension $4$ is also the only one in which this can happen.
\end{remark}

\subsection{Lcs structures on $\mathfrak{su}_2\oplus\bR$} We fix a basis $\{e_1,e_2,e_3\}$ of $\mathfrak{su}_2$ in such a way that the brackets are
\[
[e_1,e_2]=-e_3, \quad [e_1,e_3]=e_2 \quad \textrm{and} \quad [e_2,e_3]=-e_1\,.
\]
With respect to the dual basis $\{e^1,e^2,e^3\}$ of $\mathfrak{su}_2^*$, the structure equations are
\[
de^1=e^{23}, \quad de^2=-e^{13} \quad \textrm{and} \quad de^3=e^{12}\,.
\]
\begin{proposition}\label{prop:su_2+R}
Up to automorphisms, every lcs structure on $\mathfrak{su}_2\oplus\bR$ is equivalent to
\[
(\Om_r,\vt)=(r(e^{12}+e^{34}),e^4), \ r>0\,,
\]
where $e^4$ is a generator of $\bR$.
\end{proposition}
\begin{proof}
In order to classify lcs structure on $\mathfrak{su}_2\oplus\bR$ it is enough to classify contact structures on $\mathfrak{su}_2$. A generic 1-form $\eta=\alpha_1e^1+\alpha_2e^2+\alpha_3e^3$ on $\mathfrak{su}_2$ is contact if and only if $\alpha_1^2+\alpha_2^2+\alpha_3^2>0$. $\eta$ is a point on the sphere of radius $r=\sqrt{\alpha_1^2+\alpha_2^2+\alpha_3^2}$ in $\mathfrak{su}_2^*$. Since the action of $\textrm{SU}(2)$ on such sphere is transitive, we find a change of basis such that $\eta=re^1$. This means that every contact form on $\mathfrak{su}_2$ is contactomorphic to $\eta=re^1$ for some $r>0$. The corresponding lcs structure on $\mathfrak{su}_2\oplus\bR$ is then given by setting $e^4=\vt$, then $\Omega=d\eta-e^4\wedge\eta=re^{23}-re^{41}=r(e^{14}+e^{23})$.
\end{proof}

\begin{remark}
$\mathfrak{su}_2\oplus\bR$ is the Lie algebra of the Lie group $S^3\times\bR$. The lcs structures on $\mathfrak{su}_2\oplus\bR$ give therefore left-invariant lcs structures on $S^3\times S^1$, which is an important example of a compact locally conformally symplectic manifold. It is also a complex manifold and  admits Vaisman metrics, see \cite{Belgun}.
\end{remark}

\subsection{Lcs structures on $\mathfrak{sl}_2\oplus\bR$} We fix a basis $\{e_1,e_2,e_3\}$ of $\mathfrak{sl}_2$ so that the brackets are
\[
[e_1,e_2]=-2e_3, \quad [e_1,e_3]=2e_2 \quad \textrm{and} \quad [e_2,e_3]=2e_1\,.
\]
With respect to the dual basis $\{e^1,e^2,e^3\}$ of $\mathfrak{sl}_2^*$ the structure equations are
\[
de^1=-2e^{23}, \quad de^2=-2e^{13} \quad \textrm{and} \quad de^3=2e^{12}\,.
\]
\begin{proposition}\label{prop:sl_2+R}
Up to automorphisms, every lcs structure on $\mathfrak{sl}_2\oplus\bR$ is equivalent to
\begin{itemize}
\item $(\Om_r,\vt)=(\pm r(e^{14}-2e^{23}),e^4)$;
\item $(\Om_r,\vt)=(r(-2e^{13}+e^{24}),e^4)$,
\end{itemize}
for a constant $r>0$.
\end{proposition}
\begin{proof}
As above, we classify contact structures on $\mathfrak{sl}_2$. A generic 1-form $\eta=\alpha_1e^1+\alpha_2e^2+\alpha_3e^3$ on $\mathfrak{sl}_2$ is contact if and only if $-2\alpha_1^2+2\alpha_2^2+2\alpha_3^2\neq 0$. Thus we need to classify the coadjoint orbits of $\mathfrak{sl}_2^*$ of hyperbolic and elliptic type. Such orbits are the hyperboloids $\alpha_1^2-\alpha_2^2-\alpha_3^2=r\neq 0$; for $r>0$, it is a one-sheeted hyperboloid while, for $r<0$, we get a two-sheeted hyperboloid. The group $\textrm{SL}(2,\bR)$ acts transitively on each of these hyperboloids. There are therefore 3 normal forms: $\eta=\pm re^1$ and $\eta=re^2$, which give the lcs structures
\begin{itemize}
\item $\vt=e^4$, $\Omega=\pm r(e^{14}-2e^{23})$
\item $\vt=e^4$, $\Omega=r(-2e^{13}+e^{24})$
\end{itemize}
\end{proof}

\begin{remark}
$\mathfrak{sl}_2\oplus\bR$ is the Lie algebra of the Lie group $\widetilde{\mathrm{SL}(2,\bR)}\times\bR$. The lcs structures on $\mathfrak{sl}_2\oplus\bR$ give therefore left-invariant lcs structures on $\widetilde{\mathrm{SL}(2,\bR)}\times\bR$ and on any compact quotient. Some of these quotients form a class of compact complex surfaces, called \emph{properly elliptic surfaces}, and admit lcK metrics, see \cite{Belgun}.
\end{remark}

\begin{landscape}
\section{Locally conformally symplectic structures on \texorpdfstring{$4$}{4}-dimensional solvable Lie algebras}

In this section we classify lcs structures on four-dimensional solvable Lie algebras.

\def\arraystretch{1.5}
\begin{tabular}{llccc} \toprule
    Lie algebra & Structure equations & $\mathcal{Z}(\fg)$ & completely solvable & nilpotent \\
    \hline
    $\bR^4$ & $(0,0,0,0)$ & $\bR^4$ & \checkmark & \checkmark\\
    \hline
    $\fr\fh_3$ & $(0,0,-12,0)$ & $\la e_3, e_4\ra$ & \checkmark & \checkmark\\
    \hline
    $\fr\fr_3$ & $(0,-12-13,-13,0)$ & $\la e_4\ra$ & \checkmark & $\times$ \\
    \hline
    $\fr\fr_{3,\lambda}$, $\lambda\in[-1,1]$ & $(0,-12,-\lambda 13,0)$ & $\la e_4\ra$ & \checkmark & $\times$\\
    \hline
    $\fr\fr'_{3,\gamma}$, $\gamma\geq 0$ & $(0,-\gamma 12-13,12-\gamma 13,0)$ & $\la e_4\ra$ & $\times$ & $\times$ \\
    \hline
    $\fr_2\fr_2$ & $(0,-12,0,-34)$ & $0$ & \checkmark & $\times$\\
    \hline
    $\fr'_2$ & $(0,0,-13+24,-14-23)$ & $0$ & $\times$ & $\times$ \\
    \hline
    $\fn_4$ & $(0,14,24,0)$ & $\la e_3\ra$ & \checkmark & \checkmark\\
    \hline
    $\fr_4$ & $(14+24,24+34,34,0)$ & $0$ & \checkmark & $\times$\\
    \hline
    $\fr_{4,\mu}$, $\mu\in\bR$ & $(14,\mu 24+34,\mu 34,0)$ & $0$ ($\mu\neq 0$), $\la e_2\ra$ ($\mu=0$) & \checkmark & $\times$\\
    \hline
    $\fr_{4,\alpha,\beta}$, $-1<\alpha\leq\beta\leq 1$, $\alpha\beta\neq 0$ & $(14,\alpha 24,\beta 34,0)$ & $0$ & \checkmark & $\times$ \\
    \hline
    $\hat{\fr}_{4,\beta}$, $-1\leq\beta< 0$ & $(14,-24,\beta 34,0)$ & $0$ & \checkmark & $\times$\\
    \hline
    $\fr'_{4,\gamma,\delta}$, $\gamma\in\bR$, $\delta>0$ & $(14,\gamma 24+\delta 34,-\delta 24+\gamma 34,0)$ & $0$ & $\times$ & $\times$\\
    \hline
    $\fd_4$ & $(14,-24,-12,0)$ & $\la e_3\ra$ & \checkmark & $\times$\\
    \hline
    $\fd_{4,\lambda}$, $\lambda\geq\frac{1}{2}$ & $(\lambda 14,(1-\lambda)24,-12+34,0)$ & $0$ & \checkmark & $\times$\\
    \hline
    $\fd'_{4,\delta}$, $\delta\geq 0$ & $(\frac{\delta}{2}14+24,-14+\frac{\delta}{2}24,-12+\delta 34,0)$ & $0$ ($\delta\neq0$), $\la e_3\ra$ ($\delta=0$) & $\times$ & $\times$\\
    \hline
    $\fh_4$ & $(\frac{1}{2}14+24,\frac{1}{2}24,-12+34,0)$ & $0$ & \checkmark & $\times$\\ 
    \bottomrule
\end{tabular}
\vskip 0.25 cm
\captionof{table}{Solvable Lie algebras in dimension $4$, following \cite{Ovando}; structure equations are written using the Salamon notation.}\label{algebreLie}
\end{landscape}

\begin{landscape}
\def\arraystretch{1.5}
\footnotesize
\begin{longtable}{ll|lccc}
\label{table:lcs-4}
\endfirsthead
\toprule
   Lie algebra & parameters & lcs structure & exact & 1\textsuperscript{st} kind & Lagrangian ideal $\subset\ker\vt$\\    
    \hline

    $\fr\fh_3$ & $\times$ & $(e^{12}-e^{34},e^4)$ & $\eta=-e^3$, $U=e_4$ & \checkmark &  $\la e_1, e_3 \ra$  \\
    \hline
    $\fr\fr_3$ & $\times$ & $(e^{12}+e^{34},-e^1)$ & $\times$ & $\times$ & $\la e_2,e_4\ra$ \\
    \cline{3-6}
     & & $(e^{14}\pm e^{23},-2e^1)$ & $\times$ & $\times$ & $\la e_2,e_4\ra$ \\
    \cline{3-6}
     & & $(e^{12}+e^{13}-e^{24},e^4)$ & $\eta=-e^2$, $U=e_4$ & \checkmark & $\la e_2, e_3 \ra$\\
    \hline
    $\fr\fr_{3,\lambda}$ & $\lambda=0$ & $(-e^{12}+e^{23}+e^{34},e^3)$ & $\eta=e^2-e^4$, $U=e_3$ & \checkmark & $\la e_2, e_4 \ra$\\
    \cline{2-6}
    & $\lambda\notin\{-1,1\}$ & $(e^{13}+e^{24},-e^1)$ & $\times$ & $\times$ & $\la e_2,e_3\ra$ \\ \cline{2-6}
    & $\lambda\neq 0$ & $(e^{12}+e^{34},-\lambda e^1)$ & $\times$ & $\times$ & $\la e_2,e_4\ra$ \\ \cline{2-6}
    & $\lambda\notin\{-1,0\}$ & $(e^{14}+e^{23},-(1+\lambda)e^1)$ & $\times$ & $\times$ & $\la e_2,e_4\ra$ \\ \cline{2-6}
    & $\lambda\notin\{0,1\}$ & $(e^{12}-e^{13}-e^{24}+\frac{1}{\lambda}e^{34},e^4)$ & $\eta=-e^2+\frac{e^3}{\lambda}$, $U=e_4$ & \checkmark &  $\la e_2, e_3 \ra$\\
    \hline
    $\fr\fr'_{3,\gamma}$ & $\gamma> 0$ & $(e^{14}\pm e^{23},-2\gamma e^1)$ & $\times$ & $\times$ & $\times$ \\
    \cline{2-6}
    & $\forall \gamma$ & $(\gamma e^{12}+e^{13}-e^{24},e^4)$ & $\eta=-e^2$, $U=e_4$ & \checkmark & $\la e_2, e_3 \ra$\\
    \hline
    $\fr_2\fr_2$ & $\times$ & $(-e^{12}+e^{34}+\sigma e^{23},\sigma e^3)$, $\sigma\not\in \{0,-1\}$ & 
    $\eta=e^2-\frac{e^4}{\sigma+1}$, $U=\frac{e_1+e_3}{\sigma+1}$ & $\times$ & $\la e_2, e_4 \ra$\\
    \cline{3-6}
    & & $(-e^{12}+e^{34}- e^{23},- e^3)$ & $\times$ & $\times$ & $\la e_2, e_4 \ra$\\
    \cline{3-6}
     & & $(e^{12}+e^{14}+e^{23}+\sigma e^{34},-e^3)$, $\sigma\neq -1$ & $\times$ & $\times$ & $\la e_2,e_4\ra$ \\
    \cline{3-6}
     & & $(\sigma e^{13}+e^{24},-e^1-e^3)$, $\sigma>0$ & $\times$ & $\times$ & $\times$ \\
     \cline{3-6}
     & & \hspace{-0.2 cm}\begin{tabular}{l}$\left(-\frac{\sigma+1}{\tau}e^{12}+e^{14}+e^{23}+\frac{\tau+1}{\sigma}e^{34},\sigma e^1+\tau e^3\right)$ \\ $\sigma\tau\neq 0$, $\sigma+\tau\neq -1$, $\sigma\leq\tau$\end{tabular}  &  \begin{tabular}{l}$\eta=\frac{e^2}{\tau}-\frac{e^4}{\sigma}$, \\ $U=\frac{e_1+e_3}{\sigma+\tau+1}$\end{tabular} &  $\times$ & $\la e_2, e_4 \ra$\\
     \hline
     	$\fr_2'$ & $\times$ & $(e^{13}-\tau e^{14}-\frac{1+\tau^2}{1+\sigma} e^{24},\sigma e^1+\tau e^2)$, $\sigma\not\in\{-1,0\}$, $\tau>0$ & $\eta=\frac{-e^3+\tau e^4}{\sigma+1}$, $U=\frac{e_1}{\sigma+1}$ & $\times$ & $\la e_3, e_4 \ra$ \\
    \cline{3-6}
    & & $(e^{13}-\tau e^{14}-(1+\tau^2) e^{24}, \tau e^2)$, $\tau>0$ & $\eta=-e^3+\tau e^4$, $U=e_1$ & $\times$ & $\la e_3, e_4 \ra$ \\
    \cline{3-6}
     & & $(\sigma e^{12}+e^{34},-2e^1)$, $\sigma\neq 0$ & $\times$ & $\times$ & $\times$ \\
    \cline{3-6}
     & & $\left(e^{13}-\frac{1}{\sigma+1}e^{24},\sigma e^1\right)$, $\sigma\notin\{0,-1\}$ & $\eta=-\frac{e^3}{\sigma+1}$, $U=\frac{e_1}{\sigma+1}$ & $\times$ & $\la e_3, e_4 \ra$\\
\hline
$\fn_4$ & $\times$ & $\left(\pm(e^{13}-e^{24}),e^1\right)$ & $\eta=\mp e^3$, $U= e_1$ & \checkmark & $\la e_2, e_3 \ra$\\
\hline
    $\fr_4$ & $\times$ & $(e^{14}+\sigma e^{23},-2e^4)$, $\sigma\neq 0$ & $\times$ & $\times$ & $\la e_1,e_2\ra$ \\
    \hline
    $\fr_{4,\mu}$ & $\mu=0$ & $(e^{13}+e^{24}+\sigma e^{23},e^3)$, $\sigma\neq 0$ & $\times$ & $\times$ & $\la e_1,e_2\ra$ \\
    \cline{2-6}
     & $\mu=1$ & $(e^{13}+e^{24}+\sigma e^{23},-2e^4)$, $\sigma\in\bR$ & $\times$ & $\times$ & $\la e_1,e_2\ra$ \\
     \cline{2-6}
     & $\mu\notin\{-1,1\}$ & $\left(e^{13}+e^{24},-(\mu+1) e^4\right)$ & $\times$ & $\times$ & $\la e_1,e_2\ra$ \\
     \cline{2-6}
     & $\mu\notin\{0,1\}$ & $(e^{14}\pm e^{23},-2\mu e^4)$ & $\times$ & $\times$ & $\la e_1,e_2\ra$ \\
     \hline
     $\fr_{4,\alpha,\beta}$ & $\alpha\neq\beta$ & $\left(e^{13}+e^{24},-(1+\beta)e^4\right)$ & $\times$ & $\times$ & $\la e_1,e_2\ra$ \\
    \cline{2-6}
     & $\beta\neq 1$ & $\left(e^{14}+e^{23},-(\alpha+\beta)e^4\right)$ & $\times$ & $\times$ & $\la e_1,e_2\ra$ \\
    \cline{2-6}
     & $\forall \alpha,\beta$ & $\left(e^{12}+e^{34},-(1+\alpha)e^4\right)$ & $\times$ & $\times$ & $\la e_1,e_3\ra$ \\
     \hline
     $\hat{\fr}_{4,\beta}$ & $\beta\neq -1$ & $\left(e^{13}+e^{24},(-1-\beta) e^4\right)$ & $\times$ & $\times$ & $\la e_1,e_2\ra$\\
     \cline{2-6}
     & $\forall \beta$ & $\left(e^{14}+e^{23},(1-\beta) e^4\right)$ & $\times$ & $\times$ & $\la e_1,e_2\ra$\\
     \hline
     $\fr'_{4,\gamma,\delta}$ & $\gamma\neq 0$ & $(e^{14}\pm e^{23},-2\gamma e^4)$ & $\times$ & $\times$ & $\times$ \\
    \hline
     $\fd_4$ & $\times$ & $(e^{12}-\sigma e^{34},\sigma e^4)$, $\sigma>0$ & $\eta=-e^3$, $U=\frac{e_4}{\sigma}$ & \checkmark & $\la e_1, e_3 \ra$\\
     \cline{3-6}
     & & $(e^{12}-e^{34}+e^{24},e^4)$ & $\times$ & $\times$ & $\la e_1,e_3\ra$\\
     \cline{3-6}
     & & $(\pm e^{14}+e^{23},e^4)$ & $\times$ & $\times$ & $\la e_1,e_3\ra$\\
     \hline
         $\fd_{4,\lambda}$ & $\lambda\neq 2$ & $(e^{12}-(\sigma+1)e^{34},\sigma e^4)$, $\sigma\notin\{0,-1\}$ & $\eta=-e^3$, $U=\frac{e_4}{\sigma+1}$ & $\times$ & $\la e_1,e_3\ra$ \\
    \cline{2-6}
    & $\lambda\neq 1$ & $\left(e^{12}-(1-\lambda)e^{34}+ e^{14},-\lambda e^4\right)$ & $\times$ & $\times$ & $\la e_1,e_3\ra$\\
    \cline{2-6}
    & $\lambda\notin\{\frac{1}{2},1\}$ & $\left(e^{12}- \lambda e^{34}+e^{24},(\lambda-1)e^4\right)$ & $\times$ & $\times$ & $\la e_1,e_3\ra$\\
    \cline{2-6}
    & $\lambda\notin\{\frac{1}{2},2\}$ & $\left(e^{14}\pm e^{23},(\lambda-2)e^4\right)$ & $\times$ & $\times$ & $\la e_1,e_3\ra$\\
    \cline{2-6}
    & $\lambda=1$ & $\left(e^{14}-\frac{1}{\sigma+1}e^{23}+e^{34},e^2+\sigma e^4\right)$, $\sigma\neq -1$ & $\eta=\frac{e^1+e^3}{\sigma+1}$, $U=\frac{e_4}{\sigma+1}$ & $\times$ & $\la e_1,e_3\ra$ \\
    \cline{2-6}
    & $\forall \lambda$ & $\left(\pm e^{13}+e^{24},-(\lambda+1)e^4\right)$ & $\times$ & $\times$ & $\la e_2,e_3\ra$\\
    \hline
    $\fd^{\prime}_{4,\delta}$ & $\delta=0$ & $(e^{12}-\sigma e^{34},\sigma e^4)$, $\sigma>0$ & $\eta=-e^3$, $U=\frac{e_4}{\sigma}$ & \checkmark & $\times$\\
    \cline{2-6}
    & $\delta>0$ & $\left(\pm(e^{12}-(\delta+\sigma)e^{34}),\sigma e^4\right)$, $\sigma\notin\{0,-\delta\}$ & $\eta=\mp e^3$, $U=\frac{e_4}{\delta+\sigma}$ & $\times$ & $\times$\\
    \hline
    $\fh_4$ & $\times$ & $\left(\pm\left(e^{12}-(\sigma+1)e^{34},\sigma e^4\right)\right)$, $\sigma\notin\left\{0,-1\right\}$ & $\eta=\mp e^3$, $U=\frac{e_4}{\sigma+1}$ & $\times$ & $\la e_1, e_3 \ra$\\
    \cline{3-6}
     & & $\left(\pm \left(e^{12}-\frac{1}{2}e^{34}\right)+\sigma e^{14},-\frac{1}{2}e^4\right)$, $\sigma> 0$ & $\times$ & $\times$ & $\la e_1,e_3\ra$ \\
     \cline{3-6}
     & & $\left(e^{14}+\sigma e^{23},-\frac{3}{2}e^4\right)$ & $\times$ & $\times$ & $\la e_1,e_3\ra$ \\
\bottomrule
\caption{Locally conformally symplectic (non-symplectic) structures on $4$-dimensional solvable Lie algebras, up to automorphisms of the Lie algebra.}
\end{longtable}
\normalsize

\end{landscape}

\begin{theorem}\label{thm:lcs-4}
Table \ref{table:lcs-4} contains, up to automorphisms of the Lie algebra, the lcs structures $(\Om,\vt)$ with $\vt\neq 0$ on four-dimensional solvable Lie algebras.
\end{theorem}

\subsection{Non-existence of Lagrangian ideals in $\ker\vt$}
In the statement of Theorem \ref{thm:lcs-4} we claimed that some lcs Lie algebras $(\fg,\Omega,\vt)$ do not have a Lagrangian ideal $\fj$ contained in $\ker\vt$. Here we prove this claim.

\begin{proposition}
The following lcs Lie algebras do not have a Lagrangian ideal $\fj\subset\ker\vt$:
\begin{itemize}
\item $(\fr\fr'_{3,\gamma},e^{14}\pm e^{23},-2\gamma e^1)$, $\gamma>0$;
\item $(\fr_2\fr_2,\sigma e^{13}+e^{24},-e^1-e^3)$, $\sigma>0$;
\item $(\fr'_2,\sigma e^{12}+e^{34},-2e^1)$, $\sigma\neq 0$;
\item $(\fr'_{4,\gamma,\delta},e^{14}\pm e^{23},-2\gamma e^4)$, $\gamma\neq 0$;
\item $(\fd'_{4,\delta},e^{12}-\sigma e^{34},\sigma e^4)$, $\delta=0, \sigma>0$;
\item $(\fd'_{4,\delta},\pm(e^{12}-(\delta+\sigma) e^{34}),\sigma e^4)$, $\delta>0, \sigma\not\in\{0,-\delta\}$.
\end{itemize}
\end{proposition}

\begin{proof}
Let $\fj\subset\fr\fr'_{3,\gamma}$ be a Lagrangian ideal contained in $\ker\vt$. The condition of $\fj$ being Lagrangian implies that $\dim(\fj\cap\la e_2,e_3\ra)=1$. Computing $\ad_{e_1}$, we see that $\dim(\ad_{e_1}(\fj)\cap\la e_2,e_3\ra)=2$, contradicting the hypothesis that $\fj$ is an ideal.

Let $\fj\subset\fr_2\fr_2$ be a Lagrangian ideal contained in $\ker\vt$. The condition $\fj\subset \ker\vt$ implies that $\fj\cap\la e_1,e_3\ra$, which must be non-empty, is spanned by $e_1-e_3$. But $[e_2,e_1-e_3]=-e_2$ and $[e_4,e_1-e_3]=-e_4$, hence $e_2$ and $e_4$ must both be in $\fj$ in order for $\fj$ to be an ideal. This is clearly absurd.

Let $\fj\subset\fr'_2$ be a Lagrangian ideal contained in $\ker\vt$. The condition of $\fj$ being Lagrangian implies that $\dim(\fj\cap\la e_3,e_4\ra)=1$. Computing $\ad_{e_1}$, we see that $\dim(\ad_{e_1}(\fj)\cap\la e_3,e_4\ra)=2$, contradicting the hypothesis that $\fj$ is an ideal.

Let $\fj\subset\fr'_{4,\gamma,\delta}$ be a Lagrangian ideal contained in $\ker\vt$. The condition of $\fj$ being Lagrangian implies that $\dim(\fj\cap\la e_2,e_3\ra)=1$. Computing $\ad_{e_4}$, we see that $\dim(\ad_{e_4}(\fj)\cap\la e_2,e_3\ra)=2$, contradicting the hypothesis that $\fj$ is an ideal.

Let $\fj\subset\fd'_{4,\delta}$ be a Lagrangian ideal contained in $\ker\vt$. The condition of $\fj$ being Lagrangian implies that $\dim(\fj\cap\la e_1,e_2\ra)=1$. Computing $\ad_{e_4}$, we see that $\dim(\ad_{e_4}(\fj)\cap\la e_1,e_2\ra)=2$, contradicting the hypothesis that $\fj$ is an ideal.
\end{proof}

\subsection{Other remarks concerning Table \ref{table:lcs-4}}\label{Sec:4.2}

In Table \ref{table:lcs-4} there are four examples of lcs Lie algebras whose structure can not be deduced from the results contained in Theorem \ref{corr:exact} and Proposition \ref{Lag:id}, namely
\begin{itemize}
\item $(\fr\fr'_{3,\gamma},e^{14}\pm e^{23},-2\gamma e^1)$, $\gamma>0$;
\item $(\fr_2\fr_2,\sigma e^{13}+e^{24},-e^1-e^3)$, $\sigma>0$;
\item $(\fr'_2,\sigma e^{12}+e^{34},-2e^1)$, $\sigma\neq 0$;
\item $(\fr'_{4,\gamma,\delta},e^{14}\pm e^{23},-2\gamma e^4)$, $\gamma\neq 0$.
\end{itemize}

The first and the last were treated in Examples \ref{ex:1} and \ref{ex:2} respectively, in view of the construction of Section \ref{another}. We use the same construction to show how to recover the second and the third one.

For $(\fr_2\fr_2,\sigma e^{13}+e^{24},-e^1-e^3)$, $\sigma>0$, we set $\omega=e^{24}$ and $\eta=-\frac{\sigma}{2}(e^1-e^3)$, so that $\Omega=\omega+\eta\wedge\vartheta$; moreover, $V=-\frac{1}{\sigma}(e_1-e_3)$ and $U=-\frac{1}{2}e_1-\frac{1}{2}e_3$. The Lie algebra $\fh=\ker\vt\cong\la e_1-e_3,e_2,e_4\ra$ is isomorphic to $\fr_{3,-1}=(df^1=0,df^2=-f^{12},df^3=f^{13})$ and is endowed with the cosymplectic structure $(\eta,\omega)$. The derivation $D=\ad_U\colon\fh\to\fh$ satisfies $D^*\eta=0$ and $D^*\omega=-\omega$.

For $(\fr'_2,\sigma e^{12}+e^{34},-2e^1)$, $\sigma\neq 0$, we set $\omega=e^{34}$ and $\eta=\frac{\sigma}{2}e^2$, so that $\Omega=\omega+\eta\wedge\vartheta$; we compute $V=\frac{2}{\sigma}e_2$ and $U=-\frac{1}{2}e_1$. The Lie algebra $\fh=\ker\vt\cong\la e_2,e_3,e_4\ra$ is isomorphic to $\fr'_{3,0}=(df^1=0,df^2=f^{13},df^3=-f^{12})$ and is endowed with the cosymplectic structure $(\eta,\omega)$. The derivation $D=\ad_U\colon\fh\to\fh$ satisfies $D^*\eta=0$ and $D^*\omega=-\omega$.

\section{Compact four-dimensional lcs solvmanifolds}\label{sec:compact}

In this section we consider connected, simply connected four-dimensional solvable Lie groups which admit a compact quotient and study their left-invariant lcs structures. Such groups have been studied by Bock and the next proposition is a distillation of the pertinent results contained in \cite{Bock}. We put particular emphasis on solvmanifolds which are models for compact complex surfaces and for symplectic fourfolds.

\begin{proposition}[{\cite[Table A.1]{Bock}}]
Table \ref{table:solvmanifolds} contains all four-dimensional Lie algebras whose corresponding connected, simply connected solvable Lie groups admit a compact quotient.
\end{proposition}
\begin{center}
\def\arraystretch{1.75}
\begin{table}[h!]
\begin{tabular}{lccc} \toprule
    Lie algebra & \cite{Bock} & Complex surface & Symplectic\\
    \hline
    $\bR^4$ & $4\fg_1$ & Torus & \checkmark\\
    \hline
    $\fr\fh_3$ & $\fg_{3.1}\oplus\fg_1$ & Primary Kodaira surface & \checkmark\\
    \hline
    $\fr\fr_{3,-1}$ & $\fg_{3.4}^{-1}\oplus\fg_1$ & $\times$ & \checkmark \\
    \hline
    $\fr\fr'_{3,0}$ & $\fg_{3.5}^0\oplus\fg_1$ & Hyperelliptic surface & \checkmark\\
    \hline
    $\fn_4$ & $\fg_{4.1}$ & $\times$ & \checkmark\\
    \hline
    $\fr_{4,\alpha,-(1+\alpha)}$, $-1<\alpha<-\frac{1}{2}$ & $\fg^{-(1+\alpha),\alpha}_{4.5}$ & $\times$ & $\times$ \\
    \hline
    $\fr'_{4,-\frac{1}{2},\delta}$, $\delta>0$ & $\fg^{-\frac{1}{\delta},\frac{1}{2\delta}}_{4.6}$ & Inoue surface $S^0$ & $\times$\\
    \hline
    $\fd_4$ & $\fg^{-1}_{4.8}$ & Inoue surface $S^+$  & $\times$ \\
    \hline
    $\fd'_{4,0}$ & $\fg^0_{4.9}$ & Secondary Kodaira surface & $\times$\\
    \bottomrule
\end{tabular}
\vskip .25 cm
\caption{Four-dimensional Lie algebras associated to compact solvmanifolds.}
\label{table:solvmanifolds}
\end{table}
\end{center}

Suppose $\Gamma\backslash G$ is a compact solvmanifold. It is known (see \cite{AO, Hattori,mostow,raghunathan}) that if $G$ is completely solvable Lie group (that is, the eigenvalues of the endomorphisms given by the adjoint representation of the corresponding Lie algebra are all real) or, more generally, if it satisfies the Mostow condition (that is, $\mathrm{Ad}(\Gamma)$ and $\mathrm{Ad}(G)$ have the same Zariski closure in $\mathrm{GL}(\mathfrak{g})$, the group of the linear isomorphisms of $\mathfrak{g}$), then we have isomorphisms
\begin{itemize}
\item $H^\bullet(\fg)\cong H^\bullet_{dR}(\Gamma\backslash G)$, where $H^\bullet(\fg)$ is the Lie algebra cohomology of $\fg$;
\item $H^\bullet_\vt(\fg)\cong H^\bullet_{\vt}(\Gamma\backslash G)$. Here $\vt\in\fg^*$ is a closed 1-form mapping to itself under the natural inclusion $\fg^*\hookrightarrow\Omega^1(\Gamma\backslash G)$ (this is well-defined since $H^\bullet_{\vt}(\Gamma\backslash G)$ depends only on the cohomology class of $\vt$) and $H^\bullet_\vt(\fg)$ is the cohomology of the complex $(\Lambda^\bullet\fg^*,d_\vt)$.
\end{itemize}

\begin{corollary}\label{Hattori}
Let $G$ be a connected, simply connected solvable Lie group. Assume that $G$ satisfies the Mostow condition and let $\Gamma\backslash G$ be a compact solvmanifold, quotient of $G$. Then
\[
H^\bullet_\vt(\fg) \cong H^\bullet_{\vt}(\Gamma\backslash G)\,.
\]
\end{corollary}

\subsection{$\bR^4$}
Clearly $\bR^4$ does not have any left-invariant lcs structure. However, a result of Martinet \cite{Mar} implies that every oriented compact 3-manifold admits a contact structure. This is the case for $T^3$, hence $T^4=T^3\times S^1$ admits a lcs structure of the first kind. Notice that contact structures exist on all odd-dimensional tori (see \cite{Bourgeois}), hence all even-dimensional tori of dimension $\geq 4$ admit a lcs structure of the first kind. It is not clear whether $T^4$ can admit a locally conformally K\"ahler metric, but it certainly carries no Vaisman metric since $b_1(T^4)$ is even, see \cite{KashiSato}.

\subsection{$\fr\fh_3$}
Notice that $\fr\fh_3$ is a nilpotent Lie algebra. The only lcs structure on $\fr\fh_3$ is of the first kind, hence the same is true for the left-invariant lcs structure on any nilmanifold, quotient of the connected, simply connected nilpotent Lie group with Lie algebra $\fr\fh_3$. Every nilmanifold quotient of this Lie group carries a left-invariant Vaisman metric (see \cite{Bazzoni,Ugarte}).

\subsection{$\fr\fr_{3,-1}$}
This Lie algebra admits two non-equivalent lcs structures, namely
\[
(e^{12}-e^{13}-e^{24}-e^{34},e^4) \qquad \textrm{and} \qquad (e^{12}+e^{34},e^1)\,.
\]

The first one is of the first kind, hence the same is true for a left-invariant lcs structure on any solvmanifold, quotient of $\bR\times\mathfrak{R}_{3,-1}$, the connected and simply connected completely solvable Lie group with Lie algebra $\fr\fr_{3,-1}$. Such a solvmanifold is the product of a 3-dimensional contact solvmanifold, quotient of the Lie group with Lie algebra $\fr_{3,-1}$, with $S^1$.

We consider now the second structure.
We compute $H^2_\vt(\fr\fr_{3,-1})=\la[e^{13}],[e^{34}]\ra$,
hence the lcs structure $(e^{12}+e^{34},e^1)$ is not exact. The characteristic vector is $V=-e_2$; according to Section \ref{another}, we set $\eta=-e^2$ and $\om=e^{34}$\, so to have $\Omega=\omega+\eta\wedge\vt$. The condition $-\imath_U\Omega=\eta$ yields $U=e_1$, hence $\imath_U\om=0=\imath_V\omega$. The Lie algebra $\ker\vt=\la e_2,e_3,e_4\ra$ is abelian, hence we denote it $\bR^3$; it is endowed with the cosymplectic structure $(\eta,\omega)=(-e^2,e^{34})$. The derivation $D=\ad_U\colon\bR^3\to\bR^3$ is given by the matrix $\mathrm{diag}(1,-1,0)$. Exponentiating, we obtain a 1-parameter subgroup of automorphisms of $\bR^3$, $\rho\colon\bR\to\mathrm{Aut}(\bR^3)$, $t\mapsto\mathrm{diag}(e^t,e^{-t},1)$. Since $\bR^3$ is abelian, the exponential map $\exp\colon\bR^3\to\bR^3$ is the identity and $\rho$ is a 1-parameter subgroup of automorphisms of $\bR^3$ (seen as a Lie group). We consider the semidirect product $\bR^3\rtimes_\rho\bR$ (clearly $\bR^3\rtimes_\rho\bR\cong\bR\times\mathfrak{R}_{3,-1}$); to construct a lattice in $\bR^3\rtimes_\rho\bR$ compatible with the semidirect product structure we need to find some value $t_0$ for which $\rho(t_0)$ is conjugate to a matrix in $\mathrm{SL}(3,\bZ)$. To do so, we consider the characteristic polynomial of $\rho(t)$: it is $-\lambda^3+(1+\exp(t)+\exp(-t))\lambda^2-(1+\exp(t)+\exp(-t))\lambda+1$. In particular, $\rho(t_0)$ is conjugated to a matrix in $\mathrm{SL}(3,\bZ)$ only if $1+\exp(t_0)+\exp(-t_0)=n$ for some $n\in\bZ$, and in this case the characteristic polynomial $-\lambda^3+n\lambda^2-n\lambda+1$ is the same as the one of the matrix
$$
\left(\begin{matrix}
0 & 1 & 0 \\
-1 & n-1 & 0 \\
0 & 0 & 1
\end{matrix}\right)
\in \mathrm{SL}(3,\bZ) .
$$
The equation $1+\exp(t)+\exp(-t)=n$ has solution for $n\geq3$, that is,
$$ t_0(n) = \log \frac{n-1+\sqrt{(n-1)^2-4}}{2}\,. $$
For example, for $n=4$, we get that $t_0=\log\left(\frac{3+\sqrt{5}}{2}\right)$ gives $\rho(t_0)$ is conjugated to the matrix
\[
A=\begin{pmatrix}
0 & 1 & 0\\-1 & 3 & 0\\0 & 0 & 1
\end{pmatrix}\in\mathrm{SL}(3,\bZ)\,,
\]
{\itshape i.e.} there exists $P\in\mathrm{GL}(3,\bR)$ such that $PA=\rho(t_0)P$. Let $\bZ^3$ denote the standard lattice in $\bR^3$ and set $\Gamma_0=P(\bZ^3)$. Then $\rho(t_0)$ preserves $\Gamma_0$ and $\Gamma_0\rtimes_{\rho}(t_0\bZ)$ is a lattice in $\bR^3\rtimes_\rho\bR$.
The group $\bR^3$ is endowed with the left-invariant cosymplectic structure $(\eta,\omega)=(-e^2,e^{34})$. By construction $\rho(t_0)$ descends to a diffeomorphism $\psi$ of $T^3=\Gamma_0\backslash \bR^3$ which satisfies $\psi^*\eta=e^{t_0}\eta$ and $\psi^*\om=e^{-t_0}\om$. The solvmanifold $(\Gamma_0\rtimes_{\rho}(t_0\bZ))\backslash(\bR^3\rtimes_\rho\bR)$ is identified with the mapping torus $(T^3)_\psi$ and is endowed with the lcs structure $(e^{12}+e^{34},e^1)$. Since $\bR\times\mathfrak{R}_{3,-1}$ is completely solvable, Corollary \ref{Hattori} yields an isomorphism
\[
H^*_{\vt}((\Gamma_0\rtimes_{\rho}(t_0\bZ))\backslash(\bR^3\rtimes_\rho\bR))\cong H^*_\vt(\fr\fr_{3,-1})\,,
\]
hence the lcs structure $(e^{12}+e^{34},e^1)$ on $(T^3)_\psi$ is not exact.

\subsection{$\fr\fr'_{3,0}$}
The only lcs structure on $\fr\fr'_{3,0}$ is $(e^{13}-e^{24},e^4)$; it is of the first kind, hence the same is true for the left-invariant lcs structure induced on any solvmanifold, quotient of the connected, simply connected nilpotent Lie group with Lie algebra $\fr\fr'_{3,0}$. The resulting solvmanifold, which is the product of a 3-dimensional solvmanifold with a circle, is a model for a compact complex surface, namely the hyperelliptic (or bi-elliptic) surface. It does not carry any lcK metric.

\subsection{$\fn_4$}
Both lcs structures on $\fn_4$ are of the first kind, hence the same is true for the corresponding left-invariant lcs structures on any nilmanifold, quotient of the connected, simply connected nilpotent Lie group with Lie algebra $\fn_4$. A nilmanifold quotient of such group provided the first example of a lcs 4-manifold, not the product of a 3-manifold and a circle, which does not carry any lcK metric, see \cite{Bazzoni_Marrero}.

\subsection{$\fr_{4,\alpha,-(1+\alpha)}$, $-1<\alpha<-\frac{1}{2}$}
One has $\beta=-1-\alpha$, hence $-\frac{1}{2}<\beta<0$. For such values of the parameters, this Lie algebra admits three non-equivalent lcs structures:
\begin{itemize}
\item $(e^{13}+e^{24},\alpha e^4)$;
\item $(e^{14}+e^{23}, e^4)$;
\item $(e^{12}+e^{34},-(1+\alpha)e^4)$.
\end{itemize}
We have
\[
H^2_{\alpha e^4}(\fr_{4,\alpha,-(1+\alpha)})=\la[e^{13}]\ra, \ H^2_{e^4}(\fr_{4,\alpha,-(1+\alpha)})=\la[e^{23}]\ra \ \textrm{and} \ H^2_{-(1+\alpha)e^4}(\fr_{4,\alpha,-(1+\alpha)})=\la[e^{12}]\ra\,,
\]
hence none of the above lcs structures is exact.

We start with the first structure. The characteristic vector of $(e^{13}+e^{24},\alpha e^4)$ is $V=\alpha e_2$; according to Section \ref{another}, we set $\eta=\frac{1}{\alpha}e^2$ and $\om=e^{13}$, so that $\Omega=\om+\eta\wedge\vartheta$. We compute $U=\frac{1}{\alpha}e_4$, hence $\imath_U\om=0=\imath_V\omega$. The Lie algebra $\ker\vt=\la e_1,e_2,e_3\ra$ is abelian, hence we denote it $\bR^3$; it is endowed with the cosymplectic structure $(\eta,\omega)=\left(\frac{1}{\alpha}e^2,e^{13}\right)$. The derivation $D_1=\ad_U\colon\bR^3\to\bR^3$ is given by the matrix $\mathrm{diag}\left(\frac{1}{\alpha},1,-\frac{1+\alpha}{\alpha}\right)$ and $\fr_{4,\alpha,-(1+\alpha)}\cong \bR^3\rtimes_{D_1}\bR$. Exponentiating $D_1$ we obtain a 1-parameter subgroup of automorphisms of $\bR^3$, $\rho_1\colon\bR\to\mathrm{Aut}(\bR^3)$, $t\mapsto\mathrm{diag}\left(\exp\left(\frac{t}{\alpha}\right),\exp\left(t\right),\exp\left(-\frac{(1+\alpha)t}{\alpha}\right)\right)$. The exponential map $\exp\colon\bR^3\to\bR^3$ is the identity and $\rho_1$ is a 1-parameter subgroup of automorphisms of the Lie group $\bR^3$. The only connected, simply connected solvable Lie group with Lie algebra $\fr_{4,\alpha,-(1+\alpha)}$ is isomorphic to $\bR^3\rtimes_{\rho_1}\bR$. For $\lambda>1$ consider the Lie algebra $\mathfrak{sol}^4_\lambda=(\lambda 14,24,-(1+\lambda) 34, 0)$. It is easy to see that $\mathfrak{sol}^4_\lambda\cong \fr_{4,\frac{1}{2\lambda}-1,-\frac{1}{2\lambda}}$. Let $\textrm{Sol}^4_\lambda$ denote the unique connected, simply connected solvable Lie group with Lie algebra $\mathfrak{sol}^4_\lambda$. A lattice in $\textrm{Sol}^4_\lambda$, compatible with the corresponding semidirect product structure $\textrm{Sol}^4_\lambda\cong \bR^3\rtimes_\varphi\bR$, $\varphi(s)=\diag(e^{\lambda s}, e^s,e^{-(1+\lambda)s})$, has been constructed in \cite[Proposition 2.1]{Lee} for a countable set of parameters $\lambda$. Using this isomorphism, we find (for a countable set of parameters $\alpha$) $t_1=t_1(\alpha)\in\bR$ such that $\rho_1(t_1)$ is conjugated to a matrix in $\textrm{SL}(3,\bZ)$, hence, arguing as above, a lattice of the form $\Gamma_1\rtimes_\rho(t_1\bZ)$ contained in $\bR^3\rtimes_{\rho_1}\bR$. The group $\bR^3$ is endowed with the left-invariant cosymplectic structure $(\eta,\omega)=\left(\frac{1}{\alpha}e^2,e^{13}\right)$. By construction $\rho_1(t_1)$ descends to a diffeomorphism $\psi_1$ of $T^3=\Gamma_1\backslash \bR^3$ which satisfies $\psi_1^*\eta=e^{t_1}\eta$ and $\psi_1^*\om=e^{-t_1}\om$. The solvmanifold $(\Gamma_1\rtimes_{\rho_1}(t_1\bZ))\backslash(\bR^3\rtimes_{\rho_1}\bR)$ is identified with the mapping torus $(T^3)_{\psi_1}$ and is endowed with the lcs structure $(e^{13}+e^{24},\alpha e^4)$. Since $\bR^3\rtimes_{\rho_1}\bR$ is completely solvable, Corollary \ref{Hattori} yields an isomorphism
\[
H^*_{\vt}((\Gamma_1\rtimes_{\rho_1}(t_1\bZ))\backslash(\bR^3\rtimes_{\rho_1}\bR))\cong H^*_\vt(\fr_{4,\alpha,-(1+\alpha)})\,,
\]
hence the lcs structure $(e^{12}+e^{34},e^4)$ on $(T^3)_{\psi_1}$ is not exact.

We continue with $(e^{14}+e^{23},e^4)$; the characteristic vector is $V=e_1$; we set $\eta=e^1$ and $\om=e^{23}$, so that $\Omega=\om+\eta\wedge\vartheta$. We compute $U=e_4$, hence $\imath_U\om=0=\imath_V\omega$. The abelian Lie algebra $\ker\vt=\la e_1,e_2,e_3\ra=\bR^3$ is endowed with the cosymplectic structure $(\eta,\omega)=(e^1,e^{23})$. We have the derivation $D_2=\ad_U\colon\bR^3\to\bR^3$, given by $D_2=\diag(1,\alpha,-(1+\alpha))$. The same argument as for the previous lcs structure provides 1-parameter subgroup of automorphisms $\rho_2\colon\bR\to\Aut(\bR^3)$ so that the given Lie group can be written as $\bR^3\rtimes_{\rho_2}\bR$; moreover, for some $t_2\in\bR$, $\rho_2(t_2)$ preserves a lattice $\Gamma_2\subset\bR^3$, hence we obtain a lattice $\Gamma_2\rtimes_{\rho_2}(t_2\bZ)\subset\bR^3\rtimes_{\rho_2}\bR$. The group $\bR^3$ is endowed with the left-invariant cosymplectic structure $(\eta,\omega)=(e^1,e^{23})$. By construction $\rho_2(t_2)$ descends to a diffeomorphism $\psi_2$ of $T^3=\Gamma_2\backslash \bR^3$ which satisfies $\psi_2^*\eta=e^{t_2}\eta$ and $\psi_2^*\om=e^{-t_2}\om$. The solvmanifold $(\Gamma_2\rtimes_{\rho_2}(t_2\bZ))\backslash(\bR^3\rtimes_{\rho_2}\bR)$ is identified with the mapping torus $(T^3)_{\psi_2}$ and is endowed with the lcs structure $(e^{14}+e^{23},e^4)$. Arguing as above, the lcs structure $(e^{14}+e^{23},e^4)$ on $(T^3)_{\psi_2}$ is not exact.

In the last case, $(e^{12}+e^{34},-(1+\alpha)e^4)$, the characteristic vector is $V=(-1-\alpha)e_3$ and we set $\eta=-\frac{1}{1+\alpha}e^3$ and $\om=e^{12}$, so that $\Omega=\om+\eta\wedge\vt$. We compute $U=-\frac{1}{1+\alpha}e_4$, hence $\imath_U\om=0=\imath_V\omega$. Moreover, we have a derivation $D_3=\ad_U\colon\bR^3\to\bR^3$, given by $D_3=\diag(-\frac{1}{1+\alpha},-\frac{\alpha}{1+\alpha},1)$. The same construction as above produces a solvmanifold $(\Gamma_3\rtimes_{\rho_3}(t_3\bZ))\backslash(\bR^3\rtimes_{\rho_3}\bR)$. The torus $T^3=\Gamma_3\backslash\bR^3$ is endowed with the left-invariant cosymplectic structure $(-\frac{1}{1+\alpha}e^3,e^{12})$ and a diffeomorphism $\psi_3$ satisfying $\psi_3^*\eta=e^{t_3}\eta$ and $\psi_3^*\om=e^{-t_3}\omega$. The solvmanifold $(\Gamma_3\rtimes_{\rho_3}(t_3\bZ))\backslash(\bR^3\rtimes_{\rho_3}\bR)$ is identified with the mapping torus $(T^3)_{\psi_3}$. By the same token, the lcs structure $(e^{12}+e^{34},-(1+\alpha)e^4)$ on $(T^3)_{\psi_3}$ is not exact.

\begin{proposition}
For $i\in\{1,2,3\}$, the solvmanifolds $(T^3)_{\psi_i}$ constructed above are examples of $4$-dimensional manifolds which admit lcs structures but neither symplectic nor complex structures. Moreover, $(T^3)_{\psi_i}$ are not products of a 3-dimensional manifold and a circle.
\end{proposition}

\begin{proof}
It is clear that $(T^3)_{\psi_i}$ are neither symplectic nor complex manifolds. That they are not products follows from the same argument as in \cite[Example 4.19, Proposition 4.21]{BFM2}.
\end{proof}

\subsection{$\fr'_{4,-\frac{1}{2},\delta}$, $\delta>0$} We choose $\gamma=-\frac{1}{2}$, otherwise the Lie algebra is not unimodular and can not admit compact quotients. This Lie algebra admits two non-equivalent lcs structures, namely
\[
(\Om_\pm,\vt)=(e^{14}\pm e^{23},e^4)\,,
\]
both non-exact since $[e^{23}]\neq 0$ in $H^2_{\vt}(\fr'_{4,-\frac{1}{2},\delta})$. The characteristic vector of $(\Om_\pm,\vt)$ is $V=e_1$; according to Section \ref{another}, we set $\eta=e^1$ and $\om_\pm=\pm e^{23}$ in order to have $\Omega=\omega+\eta\wedge\vartheta$. We compute $U=e_4$, hence $\imath_U\om_\pm=0=\imath_V\omega_\pm$. The Lie algebra $\ker\vt=\la e_1,e_2,e_3\ra$ is abelian, hence we denote it $\bR^3$; it is endowed with the cosymplectic structures $(\eta,\pm\om)=(e^1,\pm e^{23})$. The derivation $D=\ad_U\colon\bR^3\to\bR^3$ is given by the matrix 
\[
D=\begin{pmatrix}
1 & 0 & 0\\0 & -\frac{1}{2} & \delta\\0 &-\delta & -\frac{1}{2}
\end{pmatrix}\,.
\]

Exponentiating it, we obtain a 1-parameter subgroup of automorphisms of $\bR^3$, $\rho\colon\bR\to\mathrm{Aut}(\bR^3)$,
\[
t\mapsto\exp(tD)=\begin{pmatrix}
e^t & 0 & 0\\0 & e^{-\frac{t}{2}}\cos t\delta & e^{-\frac{t}{2}}\sin t\delta\\0 &-e^{-\frac{t}{2}}\sin t\delta & e^{-\frac{t}{2}}\cos t\delta
\end{pmatrix}\,.
\]
Since $\bR^3$ is abelian, the exponential map $\exp\colon\bR^3\to\bR^3$ is the identity. Hence $\rho$ lifts to a 1-parameter subgroup of automorphisms of $\bR^3$ and the only connected, simply connected solvable Lie group with Lie algebra $\fr'_{4,-\frac{1}{2},\delta}$ is isomorphic to $\bR^3\rtimes_\rho\bR$. We show how to construct a lattice of the form $\Gamma_0\rtimes_{\rho}(t_0\bZ)$ in $\bR^3\rtimes_\rho\bR$, where $t_0\in\bR$ is to be determined, for special choices of $\delta$. We determine $t_0$ in such a way that $\rho(t_0)$ is conjugated to a matrix in $\mathrm{SL}(3,\bZ)$. Choose $\delta$ of the form $\frac{\pi m}{t_0}$ for $m\in\bZ$ where $t_0$ has to be fixed such that $mt_0>0$. In this case, we are reduced to the diagonal matrix
$$
\exp(t_0D) =
\left(
\begin{matrix}
\exp(t_0) & 0 & 0 \\
0 & (-1)^m \exp(-\frac{1}{2}t_0) & 0 \\
0 & 0 & (-1)^m \exp(-\frac{1}{2}t_0)
\end{matrix}
\right) .
$$
We have to solve the equations
$$\left\{
\begin{array}{l}
e^{t_0}+(-1)^m 2 e^{-\frac{1}{2}t_0} = p \\[5pt]
2e^{\frac{1}{2}t_0}+(-1)^m e^{-t_0} = q
\end{array}
\right.
$$
where $m\in\bZ$ and $p,q\in\bZ$, for $t_0$ such that $mt_0>0$.
For example, we choose $m=-1$ and $p=-5$, $q=-3$, and we solve for $t_0<0$.
We consider the curves $\varphi_1(x)\coloneq x^3 + 5x - 2$ and $\varphi_2(x)\coloneq 2x^3 + 3x^2 - 1$. Since $\varphi_1(0)=-2<-1=\varphi_2(0)$ and $\varphi_1(\frac{1}{2})= \frac{5}{8} > 0 = \varphi_2(\frac{1}{2})$, there exists $0<x_0<\frac{1}{2}$ such that $\varphi_1(x_0)=\varphi_2(x_0)$. Then $t_0\coloneq 2\log x_0<0$ solves the above system. Therefore there is a lattice $\Gamma_0$ in $\bR^3$ such that $\rho(t_0)$ preserves $\Gamma_0$. We consider the solvmanifold $(\Gamma_0\rtimes_{\rho}(t_0\bZ))\backslash(\bR^3\rtimes_\rho\bR)=(\Gamma_0\backslash\bR^3)_\psi$, where $\rho(t_0)$ descends to a diffeomorphism $\psi$ of the torus $\Gamma_0\backslash \bR^3$ which acts on the cosymplectic structure $(\eta,\pm\omega)=(e^1,\pm e^{23}))$ as $\psi^*\eta=\exp(t_0)\eta$ and $\psi^*\omega=\exp(-t_0)\omega$. Then the solvmanifold is endowed with the lcs structures $(e^{14}\pm e^{23},e^4)$. The Lie algebra $\fr'_{4,-\frac{1}{2},\delta}$ is not completely solvable, hence we can not use Corollary \ref{Hattori} to determine whether the lcs structures on the solvmanifold are exact. Notwithstanding, the validity of the Mostow condition for the Inoue surface of type $S^0$ is confirmed in \cite{ATO} (see also \cite{Otiman}), hence we conclude, using Corollary \ref{Hattori}, that the resulting lcs structure is not exact.

\subsection{$\fd_4$}
On this Lie algebra there are many non-equivalent lcs structures:
\begin{enumerate}
\item $(e^{12}-\sigma e^{34},\sigma e^4)$, $\sigma>0$;
\item $(e^{12}-e^{34}+e^{24},e^4)$;
\item $(\pm e^{14}+e^{23},e^4)$.
\end{enumerate}
These lcs structures provide left-invariant lcs structures on any compact quotient of the connected, simply connected solvable Lie group with Lie algebra $\fd_4$ and have been investigated by Banyaga in \cite{Banyaga}. In particular, using such solvmanifold (which had been previously studied in \cite{ACFM}), Banyaga constructed the first example of a non $d_\vt$-exact lcs structure.

The first lcs structure is of the first kind: we have $V=-e_3$ and $U=\frac{e_4}{\sigma}$. Moreover $\ker\vt=\la e_1,e_2,e_3\ra$ is isomorphic to $\heis_3$, the 3-dimensional Heisenberg algebra; it is endowed with the contact form $\eta=-e^3$. The transversal vector $U$ induces the derivation $D_\sigma=\ad_U\colon \heis_3\to\heis_3$, $D=\diag\left(\sigma,-\sigma,0\right)$. Exponentiating it, we obtain a 1-parameter subgroup of automorphisms of $\heis_3$, $\tilde\rho_\sigma\colon\bR\to\mathrm{Aut}(\heis_3)$, $t\mapsto\mathrm{diag}(e^{\sigma t},e^{-\sigma t},1)$. Notice that $\tilde\rho_\sigma(t)^*\eta=\eta$. The connected, simply connected nilpotent Lie group with Lie algebra $\heis_3$ is
\[
\Heis_3=\left\{\begin{pmatrix}
1 & x & z\\0 & 1 & y\\0 & 0 & 1
\end{pmatrix} \mid x,y,z\in\bR\right\}\,.
\]
We can lift $\tilde\rho_\sigma$ to a 1-parameter subgroup of automorphisms of $\Heis_3$ as follows: since the exponential map $\exp\colon\heis_3\to\Heis_3$ is a bijection, we have a diagram
\[
\xymatrix{
\Heis_3 \ar[r]^\Phi & \Heis_3\\
\heis_3\ar[u]^\exp \ar[r]^\rho &\heis_3\ar[u]_\exp}
\]
which defines a family of Lie group automorphisms $\rho_\sigma\colon\Heis_3\to\Heis_3$ via $\rho_\sigma=\exp\circ\tilde\rho_\sigma\circ\exp^{-1}$. One computes $\rho_\sigma(t)(x,y,z)=(e^{\sigma t}x,e^{-\sigma t}y,z)$.
The connected, simply connected solvable Lie group with Lie algebra $\fd_4$ is thus isomorphic to $\Heis_3\rtimes_\rho\bR$. A lattice of the form $\Gamma_0\rtimes_\rho(t_0\bZ)\subset\Heis_3\rtimes_\rho\bR$, for a certain $t_0\in\bR$, was explicitly constructed in \cite[Theorem 2]{Sawai}. The group $\Heis_3$ is endowed with the left-invariant contact structure $\eta=-e^3$ which descends to the nilmanifold $\Gamma_0\backslash\Heis_3$. By construction $\rho(t_0)$ descends to a diffeomorphism $\psi$ of $\Gamma_0\backslash \Heis_3$ satisfying $\psi^*\eta=\eta$. Hence the solvmanifold $(\Gamma_0\rtimes_{\rho}(t_0\bZ))\backslash(\Heis_3\rtimes_\rho\bR)$ is identified with the contact mapping torus $(\Gamma_0\backslash \Heis_3)_\psi$.

For $\vt=e^4$ we have $H^2_\vt(\fd_4)=\la [e^{23}],[e^{24}]\ra$, hence the second and the third structure are not exact.

The characteristic field of the second one is $V=-e_3$. In this case, $\ker\vt=\la e_1,e_2,e_3\ra$ is isomorphic to $\heis_3$. We try to proceed as prescribed by Section \ref{another} and set $\eta=e^2-e^3$, $\om=e^{12}=d\eta$, hence $U=e_4$ and $(\heis_3,\eta,\om)$ is a contact Lie algebra endowed with the derivation $D=\ad_U$. This derivation satisfies $D^*\eta=-e^2$, hence we are in the general case of the first \emph{Ansatz} of Section \ref{another}, for which we have no structure results.

The characteristic field of the third lcs structure is $V=\pm e_1$. According to Section \ref{another} we set $\eta=\pm e^1$ and $\om=e^{23}$, giving $\Omega=\omega+\eta\wedge\vartheta$. Then $U=e_4$ and $\imath_U\om=0=\imath_V\om$. Again $\ker\vt=\la e_1,e_2,e_3\ra$ is isomorphic to $\heis_3$, this time endowed with the cosymplectic structure $(\pm e^1,e^{23})$. $U$ induces the derivation $D_1=\ad_U\colon \heis_3\to\heis_3$, $D=\diag\left(1,-1,0\right)$. Exponentiating it, we obtain a 1-parameter subgroup of automorphisms of $\heis_3$, $\rho\colon\bR\to\mathrm{Aut}(\heis_3)$, $t\mapsto\mathrm{diag}(e^t,e^{-t},1)$. We have
\[
\tilde\rho(t)^*\eta=e^t\eta \quad \mathrm{and} \quad \tilde\rho(t)^*\omega=e^{-t}\omega\,.
\]

$\tilde\rho$ lifts to a 1-parameter subgroup of automorphisms $\rho$ of $\Heis_3$, $\rho(t)(x,y,z)=(e^tx,e^{-t}y,z)$ and there exists a lattice $\Gamma_0\subset\Heis_3$ preserved by $\rho(t_0)$ for some $t_0\in\bR$. The left-invariant cosymplectic structures $(\pm e^1,e^{23})$ on $\Heis_3$ descend to cosymplectic structures on $\Gamma_0\subset\Heis_3$ and $\rho(t_0)$ gives a diffeomorphism $\psi\colon\Gamma_0\backslash\Heis_3\to\Gamma_0\backslash\Heis_3$ satisfying $\psi^*\eta=e^{t_0}\eta$ and $\psi^*\omega=e^{-t_0}\omega$. The solvmanifold $(\Gamma_0\rtimes_\rho(t_0\bZ))\backslash(\Heis_3\rtimes_\rho\bR)$ can be identified with mapping torus $(\Gamma_0\backslash\Heis_3)_\psi$ and is endowed with the lcs structures $(\pm e^{14}+e^{23}, e^4)$. Being $\fd_4$ completely solvable, we apply Corollary \ref{Hattori} to see that the lcs structures are not exact.

As already mentioned, using solvmanifolds $M_{n,k}$ quotients of the connected, simply connected completely solvable Lie group with Lie algebra $\fd_4$, Banyaga constructed in \cite{Banyaga} the first example of a lcs structure $(\Omega,\vt)$ which is not $d_\vt$-exact. In \cite[Question 3]{Banyaga} he asked whether the dimension of the spaces $H^{i}_\vt(M_{n,k})$ ($i=1,2,3$) with $\vt=e^4$ is exactly one. In view of Corollary \ref{Hattori}, we can answer this question producing explicit generators for the Morse-Novikov cohomology:
\begin{corollary}
Let $S$ be a solvmanifold quotient of the connected, simply connected completely solvable Lie group with Lie algebra $\fd_4$. For $\vt=e^4$ we have:
\[
H^0_\vt(S)=H^4_\vt(S)=0, \quad H^1_\vt(S)=\la[e^2]\ra, \quad H^2_\vt(S)=\la[e^{23}],[e^{24}]\ra \quad \textrm{and} \quad H^3_\vt(S)=\la[e^{234}]\ra\,.
\]
\end{corollary}
\begin{remark}
In \cite[Example 2.1]{BandeKotschick} the authors proved, using the vanishing of the Euler characteristic for the Morse-Novikov cohomology, that the dimension of $H^2_\vt(S)$ must be at least two. Analogous results to ours have been obtained, with a different method, in \cite{Otiman}.
\end{remark}

\subsection{$\fd'_{4,0}$} The lcs structures on this Lie algebra are given by
\[
(\Om,\vt)=(e^{12}-\sigma e^{34},\sigma e^4), \quad \sigma>0\,.
\]
All of them are of the first kind: we have $V=-e_3$ and $U=\frac{e_4}{\sigma}$. Moreover $\ker\vt=\la e_1,e_2,e_3\ra$ is isomorphic to $\heis_3$, the 3-dimensional Heisenberg algebra; the transversal vector $U$ induces the derivation $D=\ad_U\colon \heis_3\to\heis_3$, 
\[
D=\begin{pmatrix}
0 & \frac{1}{\sigma} & 0\\-\frac{1}{\sigma} & 0 & 0\\ 0 & 0 & 0
\end{pmatrix}\,.
\]
Exponentiating it, we obtain a 1-parameter subgroup of automorphisms of $\heis_3$, $\tilde\rho\colon\bR\to\mathrm{Aut}(\heis_3)$, 
\[
\tilde\rho(t)=\exp(tD)=\begin{pmatrix}
\cos \frac{t}{\sigma} & \sin \frac{t}{\sigma} & 0\\-\sin \frac{t}{\sigma} & \cos \frac{t}{\sigma} & 0\\ 0 & 0 & 1
\end{pmatrix}\,.
\]
We lift $\tilde\rho$ to a 1-parameter subgroup of automorphisms $\rho(t)\colon\Heis_3\to \Heis_3$, $\rho(t)(x,y,z)=(x\cos \frac{t}{\sigma}-y\sin \frac{t}{\sigma},x\cos \frac{t}{\sigma}+y\sin \frac{t}{\sigma},z)$. For $t_0=\frac{\pi}{2}\sigma$, $\rho_{t_0}$ maps the lattice
\[
\Gamma=\left\{\begin{pmatrix}
1 & x & z\\0 & 1 & y\\0 & 0 & 1
\end{pmatrix} \mid x,y,z\in\bZ\right\}\subset\Heis_3
\]
to itself, therefore $\Gamma\rtimes_\rho(t_0\bZ)$ is a lattice in $\Heis_3\rtimes_\rho\bR$. The group $\Heis_3$ is endowed with the left-invariant contact structure $\eta=e^3$. By construction $\rho(t_0)$ descends to a diffeomorphism $\psi$ of $\Gamma\backslash \Heis_3$ which satisfies $\psi^*\eta=\eta$. Hence the solvmanifold $(\Gamma\rtimes_{\rho}(t_0\bZ))\backslash(\Heis_3\rtimes_\rho\bR)$ is identified with the contact mapping torus $(\Gamma\backslash \Heis_3)_\psi$.

\section{Appendix: Proof of Theorem \ref{thm:lcs-4}}

\begin{proof}

We assume throughout the proof that the non-degeneracy condition for $\Omega$ holds. This means that, for  a generic $2$-form $\Omega=\sum_{1\leq j < k \leq 4} \omega_{jk}e^{jk}$, 
\[
\omega_{12}\omega_{34}-\omega_{13}\omega_{24}+\omega_{14}\omega_{23}\neq 0 \label{eq:non_deg} \tag{\ding{70}}
\]


\subsection{$\mathfrak{rh}_{3}$, $(0,0,-12,0)$}
Take a generic $1$-form $\vartheta=\sum_{j=1}^{4}\vartheta_j e^j$ and a generic $2$-form $\Omega=\sum_{1\leq j < k \leq 4} \omega_{jk}e^{jk}$. By computing $d(\vartheta)=-\vartheta_3 e^{12}$ we see that $\vartheta_3=0$ must hold, so the generic Lee form is $\vartheta=\vartheta_1 e^1+\vartheta_2 e^2+\vartheta_4 e^4$. We assume $\vartheta_1^2+\vartheta_2^2+\vartheta_4^2\neq0$, otherwise we are in the symplectic case. We compute $d_{\vartheta}\Omega=(-\vartheta_1\omega_{23}+\vartheta_2\omega_{13})e^{123} + (-\omega_{34}-\vartheta_1\omega_{24}+\vartheta_2\omega_{14}-\vartheta_4\omega_{12})e^{124} + (-\vartheta_1\omega_{34}-\vartheta_4\omega_{13})e^{134}+(-\vartheta_2\omega_{34}-\vartheta_4\omega_{23}) e^{234}$. The parameters must therefore obey the following conditions:
\begin{enumerate}
\item $\vartheta_1\omega_{23}-\vartheta_2\omega_{13}=0$
\item $\omega_{34}+\vartheta_1\omega_{24}-\vartheta_2\omega_{14}+\vartheta_4\omega_{12}=0$
\item $\vartheta_1\omega_{34}+\vartheta_4\omega_{13}=0$
\item $\vartheta_2\omega_{34}+\vartheta_4\omega_{23}=0$
\item $\vartheta_1^2+\vartheta_2^2+\vartheta_4^2\neq0$
\end{enumerate}

Assume first that $\vt_4= 0$. Equation (5) implies that either $\vt_1\neq 0$ or $\vt_2\neq 0$. Assume $\vt_1\neq 0$; then from (3) follows $\omega_{34}=0$ and, from (2), $\vartheta_1\omega_{24}=\vartheta_2\omega_{14}$, which gives $\omega_{24}=\frac{\vt_2\omega_{14}}{\vt_1}$. (1) gives $\omega_{23}=\frac{\vt_2\omega_{13}}{\vt_1}$. Plugging this into \eqref{eq:non_deg}, we get a contradiction. Thus $\vt_1=0$. Arguing in the same way, we see that $\vt_2$ must vanish.

As a consequence, we may assume $\vt_4\neq 0$. It follows from (3) that $\om_{13}=-\frac{\vt_1\om_{34}}{\vt_4}$ and from (4) that $\om_{23}=-\frac{\vt_2\om_{34}}{\vt_4}$. The lcs structure is therefore
\begin{eqnarray*}
\vt&=&\vt_1e^1+\vt_2e^2+\vt_4e^4 \;,\\[5pt]
\Omega&=&-\left(\frac{\omega_{34}+\vartheta_1\omega_{24}-\vartheta_2\omega_{14}}{\vt_4}\right)e^{12}-\frac{\vt_1\om_{34}}{\vt_4}e^{13}+\om_{14}e^{14}-\frac{\vt_2\om_{34}}{\vt_4}e^{23}+\om_{24}e^{24}+\om_{34}e^{34}\;,
\end{eqnarray*}
together with condition (1) and $\vt_4\neq0$. Furthermore, \eqref{eq:non_deg} gives $\om_{34}\neq 0$. We consider, in terms of the basis $\{e^1,e^2,e^3,e^4\}$ of $\mathfrak{rh}_3^*$, the automorphism given by the matrix
\[
\left(\begin{array}{cccc}
-\frac{\vt_4}{\om_{34}} & 0 & \frac{\omega_{14}\vt_4}{\om_{34}^2}& \frac{\vt_1}{\om_{34}} \\
0 & 1 & -\frac{\omega_{24}}{\om_{34}} & -\frac{\vt_2}{\vt_4} \\
0 & 0 & -\frac{\vt_4}{\om_{34}} & 0 \\
0 & 0 & 0 & \frac{1}{\vt_4}
\end{array}\right) \,.
\]

This gives the normal form
\[
\left\{ \begin{array}{ccl}
         \vt & = & e^4\\
        \Omega & = & e^{12}-e^{34}
        \end{array} \right.
\]


\subsection{$\mathfrak{rr}_3$, $(0,-12-13,-13,0)$}
The generic closed 1-form is $\vt=\vt_1 e^1+\vt_4 e^4$. Imposing furthermore the conformally closedness of the generic 2-form $\Omega=\sum_{1\leq j < k \leq 4} \omega_{jk} e^{jk}$ with respect to $\vt$, we obtain the following equations:
\begin{enumerate}
\item $(\vt_1+2)\omega_{23}=0$
\item $(\vt_1+1)\omega_{24}=-\vt_4\omega_{12}$
\item $(\vt_1+1)\omega_{34}=-\om_{24}-\vt_4\om_{13}$
\item $\vt_4\om_{23}=0$
\item $\vt_1^2+\vt_4^2\neq0$
\end{enumerate}

Assume first that $\vt_4=0$; by (5), $\vt_1\neq 0$. Assuming $\vt_1=-1$, (3) gives $\om_{24}=0$ and (1) gives $\om_{23}=0$. Under these hypotheses, the lcs structure is
\[
\vt=-e^1 \quad \textrm{and} \quad \Omega=\om_{12}e^{12}+\omega_{13}e^{13}+\om_{14}e^{14}+\omega_{34}e^{34}\;
\]
and the non-degeneracy condition \eqref{eq:non_deg} yields $\om_{12}\om_{34}\neq 0$.
In terms of the basis $\{e^1,e^2,e^3,e^4\}$ of $(\mathfrak{rr}_3)^*$, we consider the automorphism given by the matrix
\[
\left(\begin{array}{cccc}
1 & 0 & -\frac{\om_{14}}{\om_{34}} & \frac{\om_{13}}{\om_{34}} \\
0 & \frac{1}{\om_{12}} & 0 & 0 \\
0 & 0 & \frac{1}{\om_{12}} & 0 \\
0 & 0 & 0 & \frac{\om_{12}}{\om_{34}}
\end{array}\right) \;
\]
This gives the normal form
\[
\left\{ \begin{array}{ccl}
         \vt & = & -e^1\\
        \Omega & = & e^{12}+e^{34}
        \end{array} \right.
\]
Assuming $\vt_1=-2$ instead, (2) gives $\om_{24}=0$ and $\om_{34}=0$ follows then from (3). Under these hypotheses, the lcs structure is
\[
\vt=-2e^1 \quad \textrm{and} \quad \Omega=\om_{12}e^{12}+\omega_{13}e^{13}+\om_{14}e^{14}+\omega_{23}e^{23}
\]
and the non-degeneracy condition \eqref{eq:non_deg} yields $\om_{14}\om_{23}\neq 0$.
According to the sign of $\om_{23}$, we consider the automorphism given by the matrix

\[
\left(\begin{array}{cccc}
1 & -\frac{\om_{13}}{\om_{23}} & \frac{\om_{12}}{\om_{23}} & 0 \\
0 & \sqrt{\frac{1}{\pm\om_{23}}} & 0 & 0 \\
0 & 0 & \sqrt{\frac{1}{\pm\om_{23}}} & 0 \\
0 & 0 & 0 & \frac{1}{\om_{14}}
\end{array}\right) \;
\]
which provides the normal form
\[
\left\{ \begin{array}{ccl}
         \vt & = & -2e^1\\
        \Omega_\pm & = & e^{14}\pm e^{23}
        \end{array} \right.
\]
If $\vt_1\notin\{-2,-1,0\}$, then (1), (2) and (3) give $\om_{23}=\om_{24}=\om_{34}=0$, which contradicts the non-degeneracy.

A Gr\"obner basis computation shows that the two models above are distinct. More precisely, we proceed as follows. We consider a matrix $A=(a_{ij})\in\mathrm{Mat}(4,\mathbb{R})$. The condition for $A$ to yield an automorphism of $\mathfrak{rr}_3$, namely $d(A(e^k))=A^{\wedge 2}(d(e^k))$, produces the ideal
\begin{eqnarray*}
I_1 &=& \left(
- a_{21}, - a_{12} a_{21} -  a_{13} a_{21} + a_{11} a_{22} + a_{11} a_{23} -  a_{22}, - a_{13} a_{21} + a_{11} a_{23} -  a_{23}, - a_{24},\right.\\[5pt]
&& \left. - a_{21} -  a_{31}, - a_{12} a_{31} -  a_{13} a_{31} + a_{11} a_{32} + a_{11} a_{33} -  a_{22} -  a_{32}, \right.\\[5pt]
&& \left. - a_{13} a_{31} + a_{11} a_{33} -  a_{23} -  a_{33}, - a_{24} -  a_{34}, - a_{12} a_{41} -  a_{13} a_{41} + a_{11} a_{42} + a_{11} a_{43}, \right.\\[5pt]
&& \left. - a_{13} a_{41} + a_{11} a_{43}, - a_{22} a_{31} -  a_{23} a_{31} + a_{21} a_{32} + a_{21} a_{33}, - a_{23} a_{31} + a_{21} a_{33}, \right.\\[5pt]
&& \left. - a_{22} a_{41} -  a_{23} a_{41} + a_{21} a_{42} + a_{21} a_{43}, - a_{23} a_{41} + a_{21} a_{43}, - a_{33} a_{41} + a_{31} a_{43}, \right.\\[5pt]
&& \left. - a_{32} a_{41} -  a_{33} a_{41} + a_{31} a_{42} + a_{31} a_{43}
\right)\subset \mathbb{Q}[a_{ij}].
\end{eqnarray*}

In order for $A^{\wedge 2}$ to transform $\Omega_1=e^{14}+e^{23}$ into $\Omega_2=e^{14}-e^{23}$, we get the ideal

\begin{eqnarray*}
I_2 &=& \left(- a_{14} a_{21} -  a_{13} a_{22} + a_{12} a_{23} + a_{11} a_{24}, - a_{14} a_{21} -  a_{13} a_{22} + a_{12} a_{23} + a_{11} a_{24}, \right. \\[5pt]
&& \left. - a_{14} a_{21} -  a_{13} a_{22} + a_{12} a_{23} + a_{11} a_{24}, - a_{14} a_{21} -  a_{13} a_{22} + a_{12} a_{23} + a_{11} a_{24}, \right.\\[5pt]
&& \left. - a_{14} a_{31} -  a_{13} a_{32} + a_{12} a_{33} + a_{11} a_{34}, - a_{14} a_{31} -  a_{13} a_{32} + a_{12} a_{33} + a_{11} a_{34}, \right.\\[5pt]
&& \left. - a_{14} a_{31} -  a_{13} a_{32} + a_{12} a_{33} + a_{11} a_{34}, - a_{14} a_{31} -  a_{13} a_{32} + a_{12} a_{33} + a_{11} a_{34}, \right.\\[5pt]
&& \left.- a_{14} a_{41} -  a_{13} a_{42} + a_{12} a_{43} + a_{11} a_{44} - 1, - a_{14} a_{41} -  a_{13} a_{42} + a_{12} a_{43} + a_{11} a_{44} - 1, \right.\\[5pt]
&& \left. - a_{14} a_{41} -  a_{13} a_{42} + a_{12} a_{43} + a_{11} a_{44} - 1, - a_{14} a_{41} -  a_{13} a_{42} + a_{12} a_{43} + a_{11} a_{44} - 1, \right.\\[5pt]
&& \left. - a_{24} a_{31} -  a_{23} a_{32} + a_{22} a_{33} + a_{21} a_{34} + 1, - a_{24} a_{31} -  a_{23} a_{32} + a_{22} a_{33} + a_{21} a_{34} + 1, \right.\\[5pt]
&& \left. - a_{24} a_{31} -  a_{23} a_{32} + a_{22} a_{33} + a_{21} a_{34} + 1, - a_{24} a_{31} -  a_{23} a_{32} + a_{22} a_{33} + a_{21} a_{34} + 1, \right.\\[5pt]
&& \left. - a_{24} a_{41} -  a_{23} a_{42} + a_{22} a_{43} + a_{21} a_{44}, - a_{24} a_{41} -  a_{23} a_{42} + a_{22} a_{43} + a_{21} a_{44}, \right.\\[5pt]
&& \left. - a_{24} a_{41} -  a_{23} a_{42} + a_{22} a_{43} + a_{21} a_{44}, - a_{24} a_{41} -  a_{23} a_{42} + a_{22} a_{43} + a_{21} a_{44}, \right.\\[5pt]
&& \left. - a_{34} a_{41} -  a_{33} a_{42} + a_{32} a_{43} + a_{31} a_{44}, - a_{34} a_{41} -  a_{33} a_{42} + a_{32} a_{43} + a_{31} a_{44}, \right.\\[5pt]
&& \left. - a_{34} a_{41} -  a_{33} a_{42} + a_{32} a_{43} + a_{31} a_{44}, - a_{34} a_{41} -  a_{33} a_{42} + a_{32} a_{43} \right.\\[5pt]
&& \left. + a_{31} a_{44}\right)\subset\mathbb{Q}[a_{ij}].
\end{eqnarray*}
A Gr\"obner basis for $I_1+I_2$ is given by
$$ \left( a_{33}^2 + 1, a_{11} - 1, a_{12}, a_{13}, a_{21}, a_{22} - a_{33}, a_{23}, a_{24}, a_{31}, a_{34}, a_{41}, a_{42}, a_{43}, a_{44} - 1 \right) $$
and so the corresponding variety is empty.

We assume from now on that $\vt_4\neq 0$. From (4), $\om_{23}=0$, while, from (2) and (3),
\[
\om_{12}=-\frac{(\vt_1+1)\om_{24}}{\vt_4} \quad \textrm{and} \quad \om_{13}=-\frac{\om_{24}+(\vt_1+1)\om_{34}}{\vt_4}.
\]
The non-degeneracy condition becomes $\om_{24}\neq 0$ and the lcs structure is $\vt=\vt_1e^1+\vt_4e^4$,
\[
\Omega=-\frac{(\vt_1+1)\om_{24}}{\vt_4}e^{12}-\frac{\om_{24}+(\vt_1+1)\om_{34}}{\vt_4}e^{13}+\om_{14}e^{14}+\omega_{24}e^{24}+\omega_{34}e^{34}.
\]
The automorphism

\[
\left(\begin{array}{cccc}
1 & -\frac{\om_{14}}{\om_{24}} & 0 & -\frac{\vt_1}{\vt_4} \\
0 & -\frac{\vt_4}{\om_{24}} & 0 & 0 \\
0 & \frac{\vt_4\om_{34}}{\om_{24}^2} & -\frac{\vt_4}{\om_{24}} & 0 \\
0 & 0 & 0 & \frac{1}{\vt_4}
\end{array}\right) \;
\]
gives the normal form
\[
\left\{ \begin{array}{ccl}
         \vt & = & e^4\\
        \Omega & = & e^{12}+e^{13}-e^{24}.
        \end{array} \right.
\]

Finally, we have to exclude the existence of automorphisms of the Lie algebra transforming one of the above Lee forms to another. We take a generic matrix $A=(a_{jk})$ with respect to the basis $\{e^j\}$ and we require that it is a morphism of the Lie algebra. We consider $\vt_1\coloneq -e^1$, $\vt_2\coloneq -2e^1$, and $\vt_3\coloneq e^4$. If we require also that $A$ sends $\vt_1$ to $\vt_2$, we have to solve the Gr\"obner basis
$$ \left(a_{11} - 2, a_{21}, a_{22}, a_{23}, a_{24}, a_{31}, a_{32}, a_{33}, a_{34}, a_{41}, a_{42}, a_{43}\right) $$
On the other hand, if we require that $A$ sends $\vt_1$ to $\vt_3$, we have to solve the Gr\"obner basis
$$ \left(a_{11}, a_{12}, a_{13}, a_{21}, a_{22}, a_{23}, a_{24}, a_{31}, a_{32}, a_{33}, a_{34}, a_{41} + 1\right) $$
Finally, if we require that $A$ sends $\vt_2$ to $\vt_3$, we have to solve the Gr\"obner basis
$$ \left(a_{11}, a_{12}, a_{13}, a_{21}, a_{22}, a_{23}, a_{24}, a_{31}, a_{32}, a_{33}, a_{34}, a_{41} + \frac{1}{2}\right) $$
In all the cases $A$ is not invertible.\\

\subsection{$\mathfrak{rr}_{3,\lambda}$, $(0,-12,-\lambda 13,0)$, $\lambda\in[-1,1]$}
The generic $1$-form $\vt=\sum_{i=1}^4\vt_ie^i$ has differential $d\vt=-\vt_2 e^{12} - \lambda\vt_3 e^{13}$.

Consider first the case $\lambda\neq0$. Then the generic Lee form is $\vartheta=\vartheta_1 e^1+\vartheta_4 e^4$. Imposing $d_\vt\Omega=0$ we obtain the following equations: 
\begin{enumerate}
\item $(\vt_1+1+\lambda)\om_{23}=0$
\item $(\vt_1+1)\omega_{24}+\vt_4\omega_{12}=0$
\item $(\vt_1+\lambda)\om_{34}+\vt_4\omega_{13}=0$
\item $\vt_4\omega_{23}=0$
\item $\vt_1^2+\vt_4^2\neq0$
\end{enumerate}

Consider the case $\vartheta_4\neq0$. Then $\om_{23}=0$ by (4) and the generic lcs structure is 
\begin{eqnarray*}
\vt&=&\vt_1e^1+\vt_4e^4 \;,\\[5pt]
\Omega&=&-\frac{(\vt_1+1)\omega_{24}}{\vartheta_{4}}e^{12} -\frac{(\lambda+\vt_1)\omega_{34}}{\vartheta_{4}} e^{13} + \omega_{14}  e^{14} + \omega_{24}  e^{24}+\omega_{34}  e^{34}\;.
\end{eqnarray*}

The non-degeneracy condition \eqref{eq:non_deg} reads $(\lambda-1)\om_{24}\om_{34}\neq 0$, hence we exclude the Lie algebra $\mathfrak{rr}_{3,1}$ from this case and assume $\om_{24}\neq 0$ and $\om_{34}\neq 0$. If $\lambda\neq -1$, the automorphism
\begin{equation}\label{eq:1592}
\left(\begin{array}{cccc}
1 & -\frac{\om_{14}}{\om_{24}} & 0 & -\frac{\vt_1}{\vt_4} \\
0 & -\frac{\vartheta_{4}}{\omega_{24}} & 0 & 0 \\
0 & 0 & \frac{\vartheta_{4}}{\lambda \omega_{34}} & 0 \\
0 & 0 & 0 & \frac{1}{\vt_4}
\end{array}\right)
\end{equation}
yields the normal form (on $\fr\fr_{3,\lambda}$ with $\lambda\notin\{0,1,\}$)
\[
\left\{ \begin{array}{ccl}
         \vt & = & e^4\\
        \Omega & = & e^{12}-e^{13}-e^{24}+\frac{1}{\lambda}e^{34}\,.
        \end{array} \right.
\]

If $\vt_4=0$, then $\vt_1\neq 0$ by (5), and we obtain the following equations:
\begin{equation}\label{rr_{3,lambda}-non-deg}
(\vt_1+1+\lambda)\om_{23}=0, \quad (\vt_1+1)\om_{24}=0, \quad \textrm{and} \quad (\vt_1+\lambda)\om_{34}=0 .
\end{equation}
Assume $\lambda\neq\pm 1$. Suppose first $\vt_1=-1$; then $\om_{23}=\om_{34}=0$, \eqref{eq:non_deg} gives $\om_{13}\om_{24}\neq 0$ and the generic lcs structure is
\[
\vt=-e^1 \quad \textrm{and} \quad \Omega=\om_{12}e^{12}+\om_{13}e^{13}+\om_{14}e^{14}+\om_{24}e^{24}.
\]
The automorphism
$$
\left(\begin{array}{cccc}
1 & -\frac{\om_{14}}{\om_{24}} & 0 & \frac{\om_{12}}{\om_{24}} \\
0 & 1 & 0 & 0 \\
0 & 0 & \frac{1}{\om_{13}} & 0 \\
0 & 0 & 0 & \frac{1}{\om_{24}}
\end{array}\right)
$$
gives the normal form (on $\fr\fr_{3,\lambda}$ with $\lambda\notin\{0,1,-1\}$)
\[
\left\{ \begin{array}{ccl}
         \vt & = & -e^1\\
        \Omega & = & e^{13}+e^{24}.
        \end{array} \right.
\]
If $\vt_1=-\lambda$, $\om_{23}=\om_{24}=0$, \eqref{eq:non_deg} gives $\om_{12}\om_{34}\neq 0$ and the generic lcs structure is
\[
\vt=-\lambda e^1 \quad \textrm{and} \quad \Omega=\om_{12}e^{12}+\om_{13}e^{13}+\om_{14}e^{14}+\om_{34}e^{34}.
\]
The automorphism
$$
\left(\begin{array}{cccc}
1 & 0 & -\frac{\om_{14}}{\om_{34}} & \frac{\om_{13}}{\om_{34}} \\
0 & \frac{1}{\om_{12}} & 0 & 0 \\
0 & 0 & 1 & 0 \\
0 & 0 & 0 & \frac{1}{\om_{34}}
\end{array}\right)
$$
gives the normal form (on $\fr\fr_{3,\lambda}$ with $\lambda\notin\{0,1,-1\}$)
\[
\left\{ \begin{array}{ccl}
         \vt & = & -\lambda e^1\\
        \Omega & = & e^{12}+e^{34}.
        \end{array} \right.
\]

If $\vt_1=-1-\lambda$, then $\om_{24}=\om_{34}=0$, \eqref{eq:non_deg} gives $\om_{14}\om_{23}\neq 0$ and the generic lcs structure is
\[
\vt=-(1+\lambda)e^1 \quad \textrm{and} \quad \Omega=\om_{12}e^{12}+\om_{13}e^{13}+\om_{14}e^{14}+\om_{23}e^{23}.
\]
The automorphism
\begin{equation}\label{aut:cococo}
\left(\begin{array}{cccc}
1 & -\frac{\om_{13}}{\om_{23}} & \frac{\om_{12}}{\om_{23}} & 0 \\
0 & \frac{1}{\om_{23}} & 0 & 0 \\
0 & 0 & 1 & 0 \\
0 & 0 & 0 & \frac{1}{\om_{14}}
\end{array}\right)
\end{equation}
gives the normal form (on $\fr\fr_{3,\lambda}$ with $\lambda\notin\{0,1,-1\}$)
\[
\left\{ \begin{array}{ccl}
         \vt & = & -(1+\lambda)e^1\\
        \Omega & = & e^{14}+e^{23}.
        \end{array} \right.
\]

If $\vt_1\notin\{-1,-\lambda,-1-\lambda\}$ then \eqref{rr_{3,lambda}-non-deg} implies that $\om_{23}=\om_{24}=\om_{34}=0$ and the lcs structure is degenerate.

We exclude the existence of automorphisms sending one of the above Lee forms, $\vt_1\coloneq -e^1$, $\vt_2\coloneq -\lambda e^1$, $\vt_3\coloneq -(1+\lambda)e^1$, $\vt_4\coloneq e^4$, ($\lambda\not\in\{-1,0,1\}$) to another. Therefore, consider $A=(a_{jk})$ with respect to the basis $\{e^j\}$ and the ideal $I_0\subset\bQ[a_{jk}]$ assuring that $A$ yields a morphism of the Lie algebra.
The ideal $I_{jk}$ containing $I_0$ and assuring that $A$ sends $\vt_j$ to $\vt_k$ has Gr\"obner basis $B_{jk}$ given by:
\begin{eqnarray*}
B_{12} &=& \left(a_{23} a_{33} \lambda -  a_{23} a_{33}, a_{33} a_{43} \lambda, a_{23} \lambda^{2} -  a_{23}, a_{33} \lambda^{2} -  a_{33} \lambda, a_{43} \lambda^{2}, a_{22} a_{34}, a_{23} a_{34}, \right.\\[5pt]
&& \left. a_{22} a_{42}, a_{23} a_{42}, a_{22} a_{43}, a_{23} a_{43}, a_{22} \lambda -  a_{22}, a_{34} \lambda, a_{42} \lambda, a_{11} -  \lambda, a_{21}, a_{24}, a_{31}, a_{41}\right) \\[5pt]
B_{13} &=& \left(a_{33} a_{43} \lambda, a_{23} \lambda^{2} + a_{23} \lambda -  a_{23}, a_{33} \lambda^{2}, a_{43} \lambda^{2} + a_{43} \lambda, a_{22} a_{23}, a_{23} a_{33}, a_{23} a_{34}, a_{22} a_{42}, \right. \\[5pt]
&& \left. a_{23} a_{42}, a_{33} a_{42}, a_{34} a_{42}, a_{23} a_{43}, a_{22} \lambda, a_{34} \lambda, a_{42} \lambda + a_{42}, a_{11} -  \lambda - 1, a_{21}, a_{24}, a_{31}, a_{32}, a_{41}\right)
 \\[5pt]
B_{14} &=& \left( a_{13} \lambda, a_{33} \lambda, a_{34} \lambda, a_{11}, a_{12}, a_{21}, a_{22}, a_{23}, a_{24}, a_{31}, a_{32}, a_{41} + 1\right) \\[5pt]
B_{23} &=& \left( a_{32} \lambda^{2} -  a_{32} \lambda -  a_{32}, a_{11} a_{32} -  a_{32} \lambda, a_{11} a_{42}, a_{32} a_{42}, a_{11} a_{43}, a_{32} a_{43}, a_{11} \lambda -  \lambda - 1, \right. \\[5pt]
&& \left. a_{42} \lambda + a_{42}, a_{43} \lambda + a_{43}, a_{21}, a_{22}, a_{23}, a_{24}, a_{31}, a_{33}, a_{34}, a_{41} \right) \\[5pt]
B_{24} &=& \left( a_{41} \lambda + 1, a_{11}, a_{12}, a_{13}, a_{21}, a_{22}, a_{23}, a_{24}, a_{31}, a_{32}, a_{33}, a_{34}\right) \\[5pt]
B_{34} &=& \left( a_{13} a_{41} + a_{13}, a_{33} a_{41} + a_{33}, a_{34} a_{41} + a_{34}, a_{13} \lambda, a_{33} \lambda, a_{34} \lambda, a_{41} \lambda + a_{41} + 1, a_{11}, a_{12}, \right. \\[5pt]
&& \left. a_{21}, a_{22}, a_{23}, a_{24}, a_{31}, a_{32}\right) .
\end{eqnarray*}
In any case, we note that the zeroes of the ideal satisfy either $a_{21}=a_{22}=a_{23}=a_{24}=0$, or (possibly in case $B_{13}$) $a_{31}=a_{32}=a_{33}=a_{34}=0$; that is, $A$ is not invertible.

Assume now $\lambda=1$. In this special case we get the following extra automorphism of $(\fr\fr_{3,1})^*$:
\begin{equation}\label{eq:533}
\left(\begin{array}{cccc}
1 & 0 & 0 & 0 \\
0 & 0 & 1 & 0 \\
0 & 1 & 0 & 0 \\
0 & 0 & 0 & 1
\end{array}\right)
\end{equation}
Then \eqref{rr_{3,lambda}-non-deg} becomes
\begin{equation}\label{rr_{3,lambda}-non-deg:1}
(\vt_1+2)\om_{23}=0, \quad (\vt_1+1)\om_{24}=0 \quad \textrm{and} \quad (\vt_1+1)\om_{34}=0.
\end{equation}
If $\vt_1=-1$, then $\om_{23}=0$, \eqref{eq:non_deg} gives $\om_{12}\om_{34}-\om_{13}\om_{24}\neq 0$ and the generic lcs structure is
\[
\vt=-e^1 \quad \textrm{and} \quad \Omega=\om_{12}e^{12}+\om_{13}e^{13}+\om_{14}e^{14}+\om_{24}e^{24}+\om_{34}e^{34}.
\]
If $\om_{24}\om_{34}\neq 0$, then the automorphism
\[
\left(\begin{array}{cccc}
1 & -\frac{\om_{14}}{\om_{24}} & 0 & \frac{\omega_{13}}{\omega_{34}} \\
0 & \frac{\omega_{34}}{\omega_{12} \omega_{34} - \omega_{13} \omega_{24}} & 0 & 0 \\
0 & 0 & \frac{\omega_{24}}{\omega_{12} \omega_{34} - \omega_{13} \omega_{24}} & 0 \\
0 & 0 & 0 & \frac{\omega_{12} \omega_{34} - \omega_{13} \omega_{24}}{\omega_{24} \omega_{34}}
\end{array}\right)
\]
gives $\vt'=-e^1$ and $\Om'=e^{12}+e^{24}+e^{34}$.
Apply now the automorphism
$$
\left(\begin{array}{cccc}
1 & 0 & 0 & 0 \\
0 & 1 & -1 & 0 \\
0 & 0 & 1 & 0 \\
0 & 0 & 0 & 1
\end{array}\right)
$$
to get the normal form (on $\fr\fr_{3,1}$)
\begin{equation}\label{eq:normal-form-rr3.1}
\left\{ \begin{array}{ccl}
         \vt & = & -e^1\\
        \Omega & = & e^{12}+e^{34}.
        \end{array} \right.
\end{equation}

Suppose $\om_{24}=0$; then $\om_{12}\om_{34}\neq 0$ and the automorphism
\begin{equation}\label{eq:534}
\left(\begin{array}{cccc}
1 & 0 & -\frac{\omega_{14}}{\omega_{34}} & \frac{\om_{13}}{\om_{34}} \\
0 & \frac{1}{\omega_{12}} & 0 & 0 \\
0 & 0 & 1 & 0 \\
0 & 0 & 0 & \frac{1}{\omega_{34}}
\end{array}\right)
\end{equation}
gives again the normal form \eqref{eq:normal-form-rr3.1}. If $\om_{34}=0$, \eqref{eq:533} leads us back to the previous case.

If $\vt_1=-2$, then $\om_{24}=\om_{34}=0$, \eqref{eq:non_deg} gives $\om_{14}\om_{23}\neq 0$ and the generic lcs structure is
\[
\vt=-2e^1 \quad \textrm{and} \quad \Omega=\om_{12}e^{12}+\om_{13}e^{13}+\om_{14}e^{14}+\om_{23}e^{23}.
\]
The automorphism \eqref{aut:cococo} gives the normal form
\[
\left\{ \begin{array}{ccl}
         \vt & = & -2e^1\\
        \Omega & = & e^{14}+e^{23}.
        \end{array} \right.
\]

If $\vt_1\neq-1$ and $\vt_1\neq-2$, then $\omega_{23}=\omega_{24}=\omega_{34}=0$, therefore $\Omega$ is degenerate.

The lcs structures with Lee forms $\vt_1\coloneq -e^1$ and $\vt_2\coloneq -2e^1$ are not equivalent. Indeed, take $A=(a_{jk})$ with respect to the basis $\{e^j\}$. Assume that $A$ induces a morphism of the Lie algebra and that it sends $\vt_1$ to $\vt_2$. This amounts to solve an ideal with Gr\"obner basis
\begin{eqnarray*}
B &=&  \left( a_{11} - 2, a_{21}, a_{22}, a_{23}, a_{24}, a_{31}, a_{32}, a_{33}, a_{34}, a_{41}, a_{42}, a_{43}\right) .
\end{eqnarray*}
This implies $a_{21}=a_{22}=a_{23}=a_{24}=0$, hence $\det A=0$.

Assume next $\lambda=-1$. In this special case we get the following extra automorphism of $(\fr\fr_{3,-1})^*$:
\begin{equation}\label{eq:535}
\left(\begin{array}{cccc}
-1 & 0 & 0 & 0 \\
0 & 0 & 1 & 0 \\
0 & 1 & 0 & 0 \\
0 & 0 & 0 & 1
\end{array}\right)
\end{equation}
Then \eqref{rr_{3,lambda}-non-deg} becomes
\begin{equation}\label{rr_{3,lambda}-non-deg:2}
\vt_1\om_{23}=0, \quad (\vt_1+1)\om_{24}=0 \quad \textrm{and} \quad (\vt_1-1)\om_{34}=0.
\end{equation}
Since $\vt_1\neq 0$, $\om_{23}=0$ and \eqref{eq:non_deg} gives $\om_{12}\om_{34}-\om_{13}\om_{24}\neq 0$ and the generic lcs structure is
\[
\vt=\vt_1e^1 \quad \textrm{and} \quad \Omega=\om_{12}e^{12}+\om_{13}e^{13}+\om_{14}e^{14}+\om_{24}e^{24}+\om_{34}e^{34}.
\]
If $\vt_1\neq\pm 1$, then \eqref{rr_{3,lambda}-non-deg:2} gives $\om_{24}=\om_{34}=0$, which contradicts \eqref{eq:non_deg}. Assuming $\vt_1=1$, we get $\om_{24}=0$ and the lcs structure
\begin{equation}\label{eq:536}
\vt=e^1 \quad \textrm{and} \quad \Omega=\om_{12}e^{12}+\om_{13}e^{13}+\om_{14}e^{14}+\om_{34}e^{34},
\end{equation}
with $\om_{12}\om_{34}\neq 0$. Using
\begin{equation*}
\left(\begin{array}{cccc}
1 & 0 & -\frac{\omega_{14}}{\omega_{34}} & \frac{\om_{13}}{\om_{34}} \\
0 & \frac{1}{\omega_{12}} & 0 & 0 \\
0 & 0 & 1 & 0 \\
0 & 0 & 0 & \frac{1}{\omega_{34}}
\end{array}\right)\;,
\end{equation*}
we obtain the normal form (on $\fr\fr_{3,-1}$)
\[
\left\{ \begin{array}{ccl}
         \vt & = & e^1\\
        \Omega & = & e^{12}+e^{34}.
        \end{array} \right.
\]
If $\vt_1=-1$ then $\om_{34}=0$; however, using \eqref{eq:535}, we are led back to \eqref{eq:536}.

The lcs structures with Lee forms $\vt_1\coloneq e^4$ and $\vt_2\coloneq e^1$ are not equivalent. Indeed, take $A=(a_{jk})$ with respect to the basis $\{e^j\}$. Assume that $A$ induces a morphism of the Lie algebra and that it sends $\vt_1$ to $\vt_2$. This amounts to solve an ideal with Gr\"obner basis
\begin{eqnarray*}
B &=& \left( a_{11} a_{13} a_{42} -  a_{11} a_{12} a_{43}, a_{11} a_{22} -  a_{22}, a_{11} a_{23} + a_{23}, a_{22} a_{23}, a_{11} a_{32} + a_{32}, \right.\\[5pt]
&& a_{22} a_{32}, a_{11} a_{33} -  a_{33}, a_{23} a_{33}, a_{32} a_{33}, a_{12} a_{41} -  a_{11} a_{42}, a_{13} a_{41} -  a_{11} a_{43}, a_{22} a_{41}, \\[5pt]
&& a_{23} a_{41}, a_{32} a_{41}, a_{33} a_{41}, a_{22} a_{42}, a_{23} a_{42}, a_{32} a_{42}, a_{33} a_{42}, a_{22} a_{43}, a_{23} a_{43}, \\[5pt]
&& \left. a_{32} a_{43}, a_{33} a_{43}, a_{14} - 1, a_{21}, a_{24}, a_{31}, a_{34}, a_{44}\right) .
\end{eqnarray*}
Solving this Grobner basis, one sees that $\det A=0$.

The case $\lambda=0$ yields structure equations $(0,-12,0,0)$. If $\vt=\sum_{i=1}^4\vt_ie^i$ is a 1-form,
the condition $d\vt=0$ gives $\vt_2=0$, whence the generic closed $1$-form is $\vt=\vt_1 e^1+\vt_3 e^3+ \vt_4 e^4$ with $\vt_1^2+\vt_3^2+\vt_4^2\neq0$. Computing $d_\vt\Om=0$, we obtain the following equations:
\begin{enumerate}
\item $(\vt_1+1)\om_{23}+\vt_3\om_{12}=0$;
\item $(\vt_1+1)\om_{24}+\vt_4\om_{12}=0$;
\item $\vt_1\om_{34}-\vt_3\om_{14}+\vt_4\om_{13}=0$;
\item $\vt_3\om_{24}-\vt_4\om_{23}=0$.
\end{enumerate}

We consider first the case $\vt_3=\vt_4=0$; then $\vt_1\neq 0$; if $\vt_1\neq-1$, then (1), (2) and (3)  imply $\om_{23}=\om_{24}=\om_{34}=0$, hence $\Omega$ is degenerate. Therefore we are reduced to $\vt=-e^1$, with (3) implying $\om_{34}=0$. The non-degeneracy condition is $\Delta=\om_{14}\om_{23} - \om_{13}\om_{24}\neq0$. The lcs structure is then
\[
\vt=-e^1, \quad \Om=\om_{12}e^{12}+\om_{13}e^{13}+\om_{14}e^{14}+\om_{23}e^{23}+\om_{24}e^{24}.
\]
If $\omega_{24}\neq0$, use the automorphism
$$
\left(\begin{array}{cccc}
1 & 0 & 0 & \frac{\omega_{12}}{\omega_{24}} \\
0 & \frac{1}{\Delta} & 0 & 0 \\
0 & 0 & -\frac{\omega_{24}}{\Delta} & \frac{\omega_{23}}{\Delta} \\
0 & 0 & \omega_{14} & -\omega_{13}
\end{array}\right) \;.
$$
If $\omega_{24}=0$, so $\omega_{14}\omega_{23}\neq0$, use the automorphism
\[
\left(\begin{array}{cccc}
1 & \frac{\omega_{13}}{\omega_{23}} & \frac{\omega_{12}}{\omega_{23}} & 0 \\
0 & 1 & 0 & 0 \\
0 & 0 & 0 & \frac{1}{\omega_{14}} \\
0 & 0 & \frac{1}{\omega_{23}} & \frac{-2\omega_{13}}{\omega_{14}\omega_{23}}
\end{array}\right)
\;.
\]
In both cases, we get the normal form
$$
\left\{\begin{array}{l}
\vartheta \;=\;-e^1\\[5pt]
\Omega\;=\;e^{13}+e^{24}\;.
\end{array}\right.
$$

Consider now the case $\vt_3^2+\vt_4^2\neq0$. We assume first $\vt_3\neq 0$; the generic lcs structure is then
\begin{eqnarray*}
\vt &=& \vt_1 e^1 + \vt_3 e^3 + \vt_4 e^4 \;,\\[5pt]
\Omega&=&\left( -\frac{(\vt_1+1)\omega_{23}}{\vartheta_{3}} \right)  e^{12} + \omega_{13}  e^{13} + \left( \frac{\omega_{34} \vartheta_{1} + \omega_{13} \vartheta_{4}}{\vartheta_{3}} \right)  e^{14} + \omega_{23}  e^{23} + \frac{\omega_{23} \vartheta_{4}}{\vartheta_{3}}  e^{24} + \omega_{34}  e^{34}
\end{eqnarray*}
and the non-degeneracy condition reads $\omega_{23}\omega_{34}\neq0$. Use the automorphism
\[
\left(\begin{array}{cccc}
1 & -\frac{\omega_{13}}{\omega_{23}} & -\frac{\vt_{1}}{\vt_{3}} & 0 \\
0 & \frac{\vt_{3}}{\omega_{23}} & 0 & 0 \\
0 & 0 & \frac{1}{\vt_{3}} & 0 \\
0 & 0 & -\frac{\vt_{4}}{\omega_{34}} & \frac{\vt_{3}}{\omega_{34}}
\end{array}\right)
\]
to get the normal form
\[
\left\{\begin{array}{l}
\vartheta \;=\;e^3\\[5pt]
\Omega\;=\;-e^{12} + e^{23} + e^{34}\;.
\end{array}\right.
\]
If $\vt_3=0$ but $\vt_4\neq 0$, the automorphism
\[
\left(\begin{array}{cccc}
1 & 0 & 0 & 0 \\
0 & 1 & 0 & 0 \\
0 & 0 & 0 & 1 \\
0 & 0 & 1 & 0
\end{array}\right)\;,
\]
brings us back to the case we just treated.

The lcs structures with Lee forms $\vt_1\coloneq -e^1$ and $\vt_2\coloneq e^3$ are not equivalent. Indeed, take $A=(a_{jk})$ with respect to the basis $\{e^j\}$. Assume that $A$ induces a morphism of the Lie algebra and that it sends $\vt_1$ to $\vt_2$. This amounts to solve an ideal with Gr\"obner basis
\begin{eqnarray*}
B &=&  \left(a_{11}, a_{12}, a_{21}, a_{22}, a_{23}, a_{24}, a_{31} + 1, a_{41}, a_{42}\right) .
\end{eqnarray*}
The solution yields $a_{21}=a_{22}=a_{23}=a_{24}=0$, hence $\det A=0$.\\

\subsection{$\mathfrak{rr}^{\prime}_{3,\gamma}$, $(0,-\gamma 12-13,12-\gamma 13,0)$, $\gamma\geq 0$}
The closedness of a generic $1$-form $\vartheta=\sum_{j=1}^{4}\vt_j e^j$ gives the conditions $\gamma\vt_2 - \vt_3=0=\gamma\vt_3 + \vt_2$. Then $\vt_2=\vt_3=0$. So the generic non-zero closed $1$-form is $\vt=\vt_1e^1+\vt_4e^4$ with $\vt_1^2+\vt_4^2\neq0$.
The condition $d_{\vt}\Omega=0$ gives
\begin{enumerate}
\item $(2\gamma+\vt_1)\omega_{23}=0$
\item $(\gamma+\vt_1)\omega_{24} + \omega_{12}\vartheta_4 - \omega_{34}=0$
\item $(\gamma+\vt_1)\omega_{34} + \omega_{13}\vartheta_4 + \omega_{24}=0$
\item $\omega_{23}\vartheta_4=0$
\end{enumerate}

We consider first the case $\vt_4=0$. Then (2) gives $\omega_{34}=(\gamma+\vartheta_1)\omega_{24}$ and (3) gives $\omega_{24}=-(\gamma+\vt_1)\om_{34}$, implying $\om_{24}=\om_{34}=0$. If $\vt_1\neq-2\gamma$, then (1) gives also $\omega_{23}=0$, whence $\Omega$ is degenerate. Then $\vt_1=-2\gamma$ (so $\gamma\neq0$).
Then, in this case, the generic lcs structure is
$$
\left\{\begin{array}{l}
\vt=-2\gamma e^1\\
\Omega=\omega_{12}e^{12} + \omega_{13}e^{13} + \omega_{14}e^{14} + \omega_{23}e^{23}
\end{array}
\right.
$$
with the non-degeneracy condition $\omega_{14}\omega_{23}\neq0$. According to the sign of $\om_{23}$, the automorphism
$$
\left(\begin{array}{cccc}
1 & -\frac{\omega_{13}}{\omega_{23}} & \frac{\omega_{12}}{\omega_{23}} & 0 \\
0 & \frac{1}{\sqrt{\pm\om_{23}}} & 0 & 0 \\
0 & 0 & \frac{1}{\sqrt{\pm\om_{23}}} & 0 \\
0 & 0 & 0 & \frac{1}{\omega_{14}}
\end{array}\right)
$$
gives the normal form (on $\mathfrak{rr}^{\prime}_{3,\gamma}$ with $\gamma>0$)
$$
\left\{\begin{array}{l}
\vt=-2\gamma e^1\\
\Omega=e^{14}\pm e^{23}
\end{array}
\right.
$$

Arguing as above, to check that the two models above are not equivalent, we have to check that the variety associated to the Gr\"obner basis
\begin{eqnarray*}
& (a_{12} a_{33} a_{34} + \frac{1}{2} a_{13}^{2} + \frac{1}{2} a_{34}^{2},
a_{13} a_{33} a_{34} - \frac{1}{2} a_{12} a_{13} - \frac{1}{2} a_{24} a_{34}, a_{12}^{2} + a_{13}^{2}, a_{12} a_{22} -  a_{12} a_{33}, &\\&
a_{13} a_{22} -  a_{13} a_{33},
a_{22}^{2} -  a_{33}^{2},
a_{12} a_{23} -  a_{13} a_{33} + a_{24}, a_{13} a_{23} + a_{12} a_{33} + a_{34}, a_{22} a_{23} + a_{32} a_{33},&\\&
a_{23}^{2} + a_{33}^{2} + a_{44},
a_{12} a_{24} + a_{13} a_{34},
a_{13} a_{24} -  a_{12} a_{34}, a_{22} a_{24} + a_{32} a_{34} + a_{13}, a_{24}^{2} + a_{34}^{2},&\\&
a_{23} a_{24} + a_{33} a_{34} -  a_{12}, a_{12} a_{32} + a_{13} a_{33} -  a_{24},
a_{13} a_{32} -  a_{12} a_{33} -  a_{34},
a_{22} a_{32} + a_{23} a_{33}, &\\&
a_{23} a_{32} -  a_{22} a_{33} - 1, a_{24} a_{32} -  a_{33} a_{34} + a_{12},
a_{32}^{2} + a_{33}^{2} + a_{44},
a_{24} a_{33} + a_{32} a_{34} + a_{13},&\\&
a_{22} a_{34} -  a_{33} a_{34}, a_{23} a_{34} + a_{32} a_{34},
a_{12} a_{44} -  a_{12},
a_{13} a_{44} -  a_{13},
a_{22} a_{44} -  a_{33}, a_{23} a_{44} + a_{32},&\\&
a_{24} a_{44} -  a_{24},
a_{32} a_{44} + a_{23},
a_{33} a_{44} -  a_{22},
a_{34} a_{44} -  a_{34},
a_{44}^{2} - 1, a_{12} \gamma -  a_{13}, a_{13} \gamma + a_{12},&\\&
a_{22} \gamma -  a_{33} \gamma,
a_{23} \gamma + a_{32} \gamma,
a_{24} \gamma -  a_{34}, a_{34} \gamma + a_{24},
a_{44} \gamma -  \gamma,
a_{11} -  a_{44},
a_{21},
a_{31},
a_{41},
a_{42},
a_{43})&
\end{eqnarray*}
is empty. Indeed, by solving it, we are reduced to $a_{22}^2 + a_{23}^2 + 1=0$.\\

Consider now the case $\vt_4\neq0$. Then (4) yields $\omega_{23}=0$. We get the generic lcs structure
$$
\left\{\begin{array}{l}
\vt=\vt_1 e^1+\vt_4 e^4, \quad \vt_4\neq0\\
\Omega =
\left( \frac{\om_{34}-(\gamma+\vt_1)\omega_{24}}{\vartheta_{4}} \right)  e^{12} + \left( -\frac{\om_{24}+(\gamma+\vt_1)\omega_{34}}{\vartheta_{4}} \right)  e^{13} + \omega_{14}  e^{14} + \omega_{24}  e^{24} + \omega_{34}  e^{34}
\end{array}\right.$$
with the non-degeneracy condition $\omega_{24}^2 + \omega_{34}^2\neq0$. Apply the automorphism
$$
\left(\begin{array}{cccc}
1 & 0 & 0 & -\frac{\vartheta_{1}}{\vartheta_{4}} \\
0 & 1 & 0 & 0 \\
0 & 0 & 1 & 0 \\
0 & 0 & 0 & \frac{1}{\vartheta_{4}}
\end{array}\right)
$$
to get the structure
$$
\left\{\begin{array}{l}
\vartheta' = e^4 \\[5pt]
\Omega' = \left(\frac{\om_{34}-\gamma \omega_{24}}{\vartheta_{4}} \right)  e^{12} + \left( -\frac{\gamma \omega_{34} + \omega_{24}}{\vartheta_{4}} \right)  e^{13} + \frac{\omega_{14}}{\vartheta_{4}}  e^{14} + \frac{\omega_{24}}{\vartheta_{4}}  e^{24} + \frac{\omega_{34}}{\vartheta_{4}}  e^{34}
\end{array}\right.
$$
Consider first the case $\omega_{24}\neq 0$ and choose $z\in(-\frac{\pi}{2},\frac{\pi}{2})$ such that $\tan z=-\frac{\om_{34}}{\om_{24}}$. The automorphism
\begin{equation}\label{rot}
\left(\begin{array}{cccc}
1 & 0 & 0 & 0 \\
0 & \cos z & -\sin z & 0 \\
0 & \sin z & \cos z & 0 \\
0 & 0 & 0 & 1
\end{array}\right)
\end{equation}
gives the structure
\[
\left\{\begin{array}{l}
\vartheta'' = e^4 \\[5pt]
\Omega'' = \gamma\alpha  e^{12} + \alpha  e^{13} + \frac{\omega_{14}}{\vartheta_{4}}  e^{14} - \alpha  e^{24}
\end{array}\right.
\]
with $\alpha=-\frac{\sqrt{\om_{24}^2+\om_{34}^2}}{\vt_4}$. The automorphism
\[
\left(\begin{array}{cccc}
1 & \frac{\om_{14}}{\alpha\vt_4} & 0 & 0 \\
0 & \frac{1}{\alpha} & 0 & 0 \\
0 & 0 & \frac{1}{\alpha} & 0 \\
0 & 0 & 0 & 1
\end{array}\right)
\]
gives the normal form
\[
\left\{\begin{array}{l}
\vartheta = e^4 \\[5pt]
\Omega =
\gamma e^{12}+e^{13}-e^{24}
\end{array}\right.\,.
\]
If $\om_{24}= 0$ then $\om_{34}\neq 0$ and the lcs structure becomes
\[
\left\{\begin{array}{l}
\vartheta' = e^4 \\[5pt]
\Omega' = \frac{\om_{34}}{\vartheta_{4}} e^{12} -\frac{\gamma \omega_{34}}{\vartheta_{4}}  e^{13} + \frac{\omega_{14}}{\vartheta_{4}}  e^{14} + \frac{\omega_{34}}{\vartheta_{4}}  e^{34}
\end{array}\right.
\]
The automorphism \eqref{rot} with $z=\frac{\pi}{2}$ gives the lcs structure
\[
\left\{\begin{array}{l}
\vartheta' = e^4 \\[5pt]
\Omega' = \frac{\gamma\om_{34}}{\vartheta_{4}} e^{12}+ \frac{\omega_{34}}{\vartheta_{4}}  e^{13} + \frac{\omega_{14}}{\vartheta_{4}}  e^{14} - \frac{\omega_{34}}{\vartheta_{4}}  e^{24}
\end{array}\right.
\]
which we can handle as we did with the general case.

Finally, we prove that no automorphism transforms $\vt_1\coloneq -2\gamma e^1$ into $e^4$. The Gr\"obner basis approach as before applied to $A=(a_{jk})$ (with respect to the basis $\{e^j\}$), to which we ask to be a morphism of the Lie algebra and to transform $\vt_1$ into $\vt_2$, amounts to solving the ideal
\begin{eqnarray*}
B &=& \left( a_{12}^{2} + a_{13}^{2}, a_{12} a_{24} -  a_{13} a_{34}, a_{13} a_{24} + a_{12} a_{34}, a_{24}^{2} + a_{34}^{2}, a_{12} a_{32} + a_{13} a_{33}, \right. \\[5pt]
&& \left. a_{13} a_{32} -  a_{12} a_{33}, a_{24} a_{32} -  a_{33} a_{34}, a_{32}^{2} + a_{33}^{2}, a_{24} a_{33} + a_{32} a_{34}, a_{12} a_{41} + \frac{1}{2} a_{13}, \right. \\[5pt]
&& \left. a_{13} a_{41} - \frac{1}{2} a_{12}, a_{24} a_{41} - \frac{1}{2} a_{34}, a_{32} a_{41} + \frac{1}{2} a_{33}, \right. \\[5pt]
&& \left. a_{33} a_{41} - \frac{1}{2} a_{32}, a_{34} a_{41} + \frac{1}{2} a_{24}, a_{12} \gamma + a_{13}, a_{13} \gamma -  a_{12}, a_{24} \gamma -  a_{34}, \right. \\[5pt]
&& \left. a_{32} \gamma + a_{33}, a_{33} \gamma -  a_{32}, a_{34} \gamma + a_{24}, a_{41} \gamma + \frac{1}{2}, a_{11}, a_{21}, a_{22} -  a_{33}, a_{23} + a_{32}, a_{31}.\right)
\end{eqnarray*}
which gives that $A$ has determinant $0$.


\subsection{$\mathfrak{r}_2\mathfrak{r}_2$, $(0,-12,0,-34)$}
The closedness of a generic $1$-form $\vartheta=\sum_{j=1}^{4}\vt_j e^j$ gives the conditions $\vt_2=0=\vt_4$. The generic Lee form is then $\vt=\vt_1e^1+\vt_3e^3$, with $\vt_1^2+\vt_3^2\neq0$. Together with the equation $d_\vt\Omega=0$ for a generic $2$-form $\Omega=\sum_{1\leq j < k \leq 4} \omega_{jk} e^{jk}$, this yields
\begin{enumerate}
\item $\vartheta_1\omega_{23}+\vt_3\om_{12}=-\om_{23}$
\item $\vt_1\omega_{24}=-\omega_{24}$
\item $\vt_1\omega_{34}-\vartheta_3\omega_{14}=\omega_{14}$
\item $\vartheta_3\omega_{24}=-\omega_{24}$
\item $\vt_1^2+\vt_3^2\neq0$
\end{enumerate}

We assume first $\vt_1=0$. By (5), $\vt_3\neq 0$ and, by (2), $\om_{24}=0$. Moreover, (1) gives $\om_{12}=-\frac{\om_{23}}{\vt_3}$. The non-degeneracy condition \eqref{eq:non_deg} reads $\om_{23}(\vt_3\om_{14}-\om_{34})\neq 0$, which implies $\omega_{23}\neq 0$. Thus the lcs structure is
\begin{equation}\label{eq:105}
\vt=\vt_3e^3 \quad \textrm{and} \quad \Omega=-\frac{\omega_{23}}{\vt_3}e^{12}+\omega_{13}e^{13}+\om_{14}e^{14}+\om_{23}e^{23}+\omega_{34}e^{34}
\end{equation}
with the condition $\omega_{14}(\vartheta_3 + 1)=0$.

We make now the assumption that $\om_{14}=0$, which forces $\om_{34}\neq 0$. Then \eqref{eq:105} reduces to
\[
\vt=\vt_3e^3 \quad \textrm{and} \quad \Omega=-\frac{\omega_{23}}{\vt_3}e^{12}+\omega_{13}e^{13}+\om_{23}e^{23}+\omega_{34}e^{34}\;.
\]
In terms of the basis $\{e^1,e^2,e^3,e^4\}$ of $(\mathfrak{r}_2\mathfrak{r}_2)^*$, we consider the automorphism given by the matrix
\[
\left(\begin{array}{cccc}
1 & -\frac{\om_{13}}{\om_{23}} & 0 & 0 \\
0 & \frac{\vt_3}{\om_{23}} & 0 & 0 \\
0 & 0 & 1 & 0 \\
0 & 0 & 0 & \frac{1}{\om_{34}}
\end{array}\right) \;.
\]
In the new base, the lcs structure reads
\[
\vt'=\vt_3e^3 \quad \textrm{and} \quad \Omega'=-e^{12}+e^{34}+\vt_3e^{23}\;.
\]
Thus, every lcs structure with the hypotheses above $\vt_1=0$, $\om_{14}=0$ is equivalent to
\begin{equation}\label{eq:r2r2-theta-eps}
\left\{ \begin{array}{ccl}
         \vt & = & \ve e^3\\
        \Omega & = & -e^{12}+e^{34}+\ve e^{23}, \, \ve\neq 0
        \end{array} \right.
\end{equation}

We assume next that $\om_{14}\neq 0$; (3) implies that $\vt_3=-1$, hence $\om_{23}=\om_{12}$ and the non-degeneracy shows that $\om_{23}\neq 0$ and $\om_{14}+\om_{34}\neq 0$, the latter being equivalent to $\frac{\om_{34}}{\om_{14}}\neq -1$. The lcs structure is then.
\[
\vt=-e^3 \quad \textrm{and} \quad \Omega=\omega_{12}e^{12}+\omega_{13}e^{13}+\omega_{14}e^{14}+\om_{12}e^{23}+\omega_{34}e^{34}\;.
\]
Applying the automorphism
\[
\left(\begin{array}{cccc}
1 & 0 & 0 & 0 \\
0 & \frac{1}{\om_{12}} & 0 & 0 \\
0 & 0 & 1 & -\frac{\om_{13}}{\om_{14}} \\
0 & 0 & 0 & \frac{1}{\om_{14}}
\end{array}\right) \;.
\]
we obtain
\[
\vt'=-e^3 \quad \textrm{and} \quad \Omega'=e^{12}+e^{14}+e^{23}+\frac{\omega_{34}}{\omega_{14}}e^{34}\;.
\]
In this second case, every lcs structure is equivalent to
\[
\left\{ \begin{array}{ccl}
         \vt & = & - e^3\\
        \Omega & = & e^{12}+e^{14}+e^{23}+\ve e^{34}, \, \ve\neq -1
        \end{array} \right.
\]

The above forms are not equivalent, up to automorphisms of the Lie algebra. Indeed, take a generic $A=(a_{jk})$ in the basis $\{e^j\}$. By requiring that $A$ is a morphism of the Lie algebra sending $\Omega_1\coloneq e^{12}+e^{14}+e^{23}+\ve_1 e^{34}$ to $\Omega_2\coloneq e^{12}+e^{14}+e^{23}+\ve_2 e^{34}$, where $\ve_1\neq-1$ and $\ve_2\neq-1$, we get an ideal whose Gr\"obner basis is
\begin{eqnarray*}
B &=& (a_{32} a_{34}, a_{32} a_{44}, a_{34} a_{44} -  a_{34}, a_{44}^{2} -  a_{44}, a_{32} \ve_{2} + a_{32}, a_{44} \ve_{2} + a_{44} -  \ve_{2} - 1,\\[5pt]
&& a_{11} -  a_{44}, a_{12} + a_{34}, a_{13} + a_{44} - 1, a_{14} + a_{32}, a_{21}, a_{22} -  a_{44}, a_{23},\\[5pt]
&& a_{24} -  a_{44} + 1, a_{31} + a_{44} - 1, a_{33} -  a_{44}, a_{41}, a_{42} -  a_{44} + 1, a_{43}, \ve_{1} -  \ve_{2})
\end{eqnarray*}
In particular, $\ve_1=\ve_2$.

Moreover, the above normal form is not equivalent to \eqref{eq:r2r2-theta-eps} with $\ve=-1$. Indeed, take $A=(a_{jk})$ with respect to the basis $\{e^j\}$ and assume that it is a morphism of Lie algebras preserving $\theta\coloneq -e^3$ and sending $\Omega_1\coloneq -e^{12}+e^{34}-e^{23}$ to $\Omega_2\coloneq e^{12}+e^{14}+e^{23}+\alpha e^{34}$ for some $\alpha\neq-1$.
Then we are reduced to find the zeroes of the ideal with Gr\"obner basis $(1)$.

We move to the case $\vt_3=0$. By (5), $\vt_1\neq 0$ and, by (4), $\om_{24}=0$. Moreover, (3) gives $\om_{34}=\frac{\om_{14}}{\vt_1}$. The non-degeneracy condition \eqref{eq:non_deg} reads $\om_{14}(\om_{12}+\vt_1\om_{23})\neq 0$, which implies $\omega_{14}\neq 0$. The lcs structure is
\begin{equation}\label{eq:106}
\vt=\vt_1e^1 \quad \textrm{and} \quad \Omega=\om_{12}e^{12}+\omega_{13}e^{13}+\om_{14}e^{14}+\om_{23}e^{23}+\frac{\omega_{14}}{\vt_1}e^{34}
\end{equation}
with the condition $\omega_{23}(\vartheta_1 + 1)=0$ and $\vartheta_1\neq0$.

Under the automorphism
\[
\left(\begin{array}{cccc}
0 & 0 & 1 & 0 \\
0 & 0 & 0 & 1 \\
1 & 0 & 0 & 0 \\
0 & 1 & 0 & 0
\end{array}\right) \;,
\]
\eqref{eq:106} becomes
\[
\vt=\vt_1e^3 \quad \textrm{and} \quad \Omega=\frac{\omega_{14}}{\vt_1}e^{12}-\omega_{13}e^{13}
-\om_{23}e^{14}-\om_{14}e^{23}+\om_{12}e^{34}\;,
\]
which is equivalent to \eqref{eq:105}. This case deserves therefore no further analysis.

The last case is $\vt_1\vt_3\neq 0$. If we assume that $\omega_{24}\neq 0$, then (2) and (4) imply, respectively, $\vt_1=-1$ and $\vt_3=-1$. From (1) follows $\om_{12}=0$ and, from (3), $\om_{34}=0$. The non-degeneracy condition \eqref{eq:non_deg} is then $\om_{13}\om_{24}-\om_{14}\om_{23}\neq 0$. Thus the lcs structures is
\[
\vt=-e^1-e^3 \quad \textrm{and} \quad \Omega=\omega_{13}e^{13}+\om_{14}e^{14}+\om_{23}e^{23}+\om_{24}e^{24}\;.
\]
The automorphism
\[
\left(\begin{array}{cccc}
1 & -\frac{\om_{14}}{\om_{24}} & 0 & 0 \\
0 & 1 & 0 & 0 \\
0 & 0 & 1 & -\frac{\om_{23}}{\om_{24}} \\
0 & 0 & 0 & 1
\end{array}\right) \;
\]
transforms the lcs structure into
\[
\vt'=-e^1-e^3 \quad \textrm{and} \quad \Omega'=\left(\frac{\om_{13}\om_{24}-\om_{14}\om_{23}}{\om_{24}}\right)e^{13}+\om_{24}e^{24}\;.
\]
Finally, the automorphisms
\[
\left(\begin{array}{cccc}
1 & 0 & 0 & 0 \\
0 & \frac{1}{\om_{24}} & 0 & 0 \\
0 & 0 & 1 & 0 \\
0 & 0 & 0 & 1
\end{array}\right) \;
\]
followed by
\[
\left(\begin{array}{cccc}
0 & 0 & 1 & 0 \\
0 & 0 & 0 & 1 \\
1 & 0 & 0 & 0 \\
0 & -1 & 0 & 0
\end{array}\right)
\]
shows that, in this case, every lcs structure is equivalent to
\begin{equation}\label{eq:r2r2-theta-13}
\left\{ \begin{array}{ccl}
         \vt & = & -e^1-e^3\\
        \Omega & = & \ve e^{13}+e^{24}, \, \ve>0\,.
        \end{array} \right.
\end{equation}

The above forms are not equivalent, up to automorphisms of the Lie algebra. Indeed, take a generic $A=(a_{jk})$ in the basis $\{e^j\}$. By requiring that it is a morphism of the Lie algebra that sends $\Omega_1\coloneq \ve_1 e^{13}+e^{24}$ to $\Omega_2\coloneq \ve_2 e^{13}+e^{24}$, where $\ve_1>0$ and $\ve_2>0$, we get an ideal whose Gr\"obner basis is
\begin{eqnarray*}
B &=& (a_{22} a_{24}, a_{22} a_{33} -  a_{22}, a_{24} a_{33}, a_{33}^{2} -  a_{33}, a_{22} a_{42},  a_{24} a_{42} -  a_{33} + 1, a_{33} a_{42}, a_{22} a_{44} -  a_{33},\\[5pt]
&&  a_{24} a_{44}, a_{33} a_{44} -  a_{44}, a_{42} a_{44}, a_{22} \ve_{1} -  a_{22} \ve_{2}, a_{24} \ve_{1} + a_{24} \ve_{2}, a_{33} \ve_{1} - \frac{1}{2} \ve_{1} - \frac{1}{2} \ve_{2},\\[5pt]
&& a_{42} \ve_{1} + a_{42} \ve_{2}, a_{44} \ve_{1} -  a_{44} \ve_{2}, \ve_{1}^{2} -  \ve_{2}^{2}, a_{33} \ve_{2} - \frac{1}{2} \ve_{1} - \frac{1}{2} \ve_{2}, a_{11} -  a_{33}, a_{12}, \\[5pt]
&& a_{13} + a_{33} - 1, a_{14}, a_{21}, a_{23}, a_{31} + a_{33} - 1, a_{32}, a_{34}, a_{41}, a_{43}).
\end{eqnarray*}
In particular, it contains $(\ve_1 + \ve_2) a_{42}$ $(\ve_1 - \ve_2) a_{44}$, $a_{41}$, $a_{43}$, from which it follows that either $\ve_1=\ve_2$, or $A$ is not invertible.\\

Finally, we assume $\vt_1\vt_3\neq 0$ and $\om_{24}=0$. Equation (1) and (3) give
\[
\om_{12}=-\frac{1+\vt_1}{\vt_3}\om_{23} \quad \textrm{and} \quad \om_{34}=\frac{\vt_3+1}{\vt_1}\om_{14}.
\]
The non-degeneracy condition implies $\om_{14}\om_{23}\neq 0$ and $\vt_1+\vt_3\neq -1$. The generic lcs structure is, in this case,
\[
\vt=\vt_1e^1+\vt_3e^3 \quad \textrm{and} \quad \Omega=\frac{-1-\vt_1}{\vt_3}\om_{23}e^{12}+\omega_{13}e^{13}+\om_{14}e^{14}+\om_{23}e^{23}+\frac{\vt_3+1}{\vt_1}\om_{14}e^{34}\;.
\]
The automorphism 
\[
\left(\begin{array}{cccc}
1 & -\frac{\om_{13}}{\om_{23}} & 0 & 0 \\
0 & \frac{1}{\om_{23}} & 0 & 0 \\
0 & 0 & 1 & 0 \\
0 & 0 & 0 & \frac{1}{\om_{14}}
\end{array}\right) \;
\]
provides the lcs structure
\[
\vt'=\vt_1e^1+\vt_3e^3 \quad \textrm{and} \quad \Omega'=-\frac{1+\vt_1}{\vt_3}e^{12}+e^{14}+e^{23}+\frac{\vt_3+1}{\vt_1}e^{34}\;,
\]
giving the model
\[
\left\{ \begin{array}{ccl}
         \vt & = & \sigma e^1+\tau e^3\\
        \Omega & = & -\frac{1+\sigma}{\tau}e^{12}+e^{14}+e^{23}+\frac{\tau+1}{\sigma}e^{34}, \, \sigma\tau\neq 0, \sigma+\tau\neq -1, \sigma\leq\tau.
        \end{array} \right.
\]
Here, we can assume $\sigma\leq\tau$ up to the automorphism
\[
\left(\begin{array}{cccc}
0 & 0 & 1 & 0 \\
0 & 0 & 0 & 1 \\
1 & 0 & 0 & 0 \\
0 & 1 & 0 & 0
\end{array}\right).
\]

We show that the above form with $(\sigma,\tau)=(-1,-1)$ is not equivalent to \eqref{eq:r2r2-theta-13}. As before, take $A=(a_{jk})$ with respect to the basis $\{e^j\}$. The conditions for $A$ to be a morphism of the Lie algebra preserving $\theta=-e^1-e^3$ and sending $\Omega_1a_{33} \ve_{2} - \frac{1}{2} \ve_{1} - \frac{1}{2} \ve_{2}, a_{11} -  a_{33}, a_{12}\coloneq\alpha e^{13}+e^{24}$ (for some $\alpha\neq0$) to $\Omega_2\coloneq e^{14}+e^{23}$ are given by the ideal with Gr\"obner basis $(1)$.

To conclude, we have to show that the above Lee forms are not equivalent up to automorphisms of the Lie algebra. Set $\vt_1\coloneq\ve_1 e^3$ (for $\ve_1\neq0$ and $\ve_1\neq-1$), $\vt_2\coloneq\ve_2 e^3$ (for $\ve_2\neq0$ and $\ve_2\neq-1$), $\vt_3\coloneq-e^3$, $\vt_4\coloneq-e^1-e^3$, $\vt_5\coloneq\sigma_1e^1+\tau_1e^3$ (for $\sigma_1\tau_1\neq0$ and $\sigma_1+\tau_1\neq-1$ and $\sigma_1\leq\tau_1$, and $(\sigma_1,\tau_1)\neq(-1,-1)$), $\vt_6\coloneq\sigma_2e^1+\tau_2e^3$ (for $\sigma_2\tau_2\neq0$ and $\sigma_2+\tau_2\neq-1$ and $\sigma_2\leq\tau_2$, and $(\sigma_2,\tau_2)\neq(-1,-1)$). Let $B_{jk}$ be a Gr\"obner basis for the ideal generated by the conditions that the morphisms associated to $A=(a_{jk})$ with respect to the basis $\{e^j\}$ is a morphism of the Lie algebra and sends $\vt_j$ to $\vt_k$.
We have
\begin{eqnarray*}
B_{12} &=& (a_{11} a_{22} -  a_{22}, a_{13} a_{24} -  a_{24}, a_{12} a_{31} -  a_{11} a_{32}, a_{22} a_{31}, a_{22} a_{32}, a_{14} a_{33} -  a_{13} a_{34}, a_{24} a_{33},\\[5pt]
&&   a_{24} a_{34}, a_{11} a_{42}, a_{12} a_{42}, a_{22} a_{42}, a_{31} a_{42} -  a_{42}, a_{13} a_{44}, a_{14} a_{44}, a_{24} a_{44}, a_{33} a_{44} -  a_{44}, \\[5pt]
&&  a_{13} \ve_{1}, a_{24} \ve_{1}, a_{33} \ve_{1} -  \ve_{2}, a_{44} \ve_{1} -  a_{44} \ve_{2}, a_{13} \ve_{2}, a_{14} \ve_{2}, a_{24} \ve_{2}, a_{21}, a_{23}, a_{41}, a_{43})
\end{eqnarray*}
Hence, if $\ve_1\neq\ve_2$, then $A$ is not invertible.
We consider now
\begin{eqnarray*}
B_{56} &=& \left( a_{22} a_{24} \sigma_{1} -  a_{22} a_{24} \sigma_{2} + a_{22} a_{24} \tau_{1}, a_{22} a_{33} \sigma_{1} -  a_{22} a_{33} \sigma_{2} + a_{13} a_{22} \tau_{2}, \right.\\[5pt]
&& a_{22} a_{34} \sigma_{1} -  a_{22} a_{34} \sigma_{2} + a_{14} a_{22} \tau_{2}, a_{13} a_{42} \sigma_{1} + a_{33} a_{42} \sigma_{2} -  a_{13} a_{42} \tau_{2}, \\[5pt]
&& a_{14} a_{42} \sigma_{1} + a_{34} a_{42} \sigma_{2} -  a_{14} a_{42} \tau_{2}, a_{24} a_{42} \sigma_{1} -  a_{24} a_{42} \tau_{2}, a_{22} a_{44} \sigma_{1} -  a_{22} a_{44} \sigma_{2}, \\[5pt]
&& a_{42} a_{44} \sigma_{1} + a_{42} a_{44} \tau_{1} -  a_{42} a_{44} \tau_{2}, a_{24} a_{31} \sigma_{2} -  a_{24} a_{31} \tau_{1} -  a_{11} a_{24} \tau_{2}, \\[5pt]
&& a_{24} a_{32} \sigma_{2} -  a_{24} a_{32} \tau_{1} -  a_{12} a_{24} \tau_{2}, a_{24} a_{42} \sigma_{2} -  a_{24} a_{42} \tau_{1}, \\[5pt]
&& a_{31} a_{44} \sigma_{2} + a_{11} a_{44} \tau_{1} -  a_{11} a_{44} \tau_{2}, a_{32} a_{44} \sigma_{2} + a_{12} a_{44} \tau_{1} -  a_{12} a_{44} \tau_{2}, \\[5pt]
&& a_{42} a_{44} \sigma_{2}, a_{13} a_{22} \tau_{1} + a_{22} \sigma_{1} -  a_{22} \sigma_{2}, a_{13} a_{31} \tau_{1} -  a_{11} a_{33} \tau_{1} -  a_{31} \sigma_{2} + a_{11} \tau_{2}, \\[5pt]
&& a_{13} a_{32} \tau_{1} -  a_{12} a_{33} \tau_{1} -  a_{32} \sigma_{2} + a_{12} \tau_{2}, a_{22} a_{33} \tau_{1} -  a_{22} \tau_{2}, \\[5pt]
&& a_{13} a_{42} \tau_{1} -  a_{42} \sigma_{2}, a_{33} a_{42} \tau_{1} + a_{42} \sigma_{1} -  a_{42} \tau_{2}, \\[5pt]
&& a_{22} a_{44} \tau_{1} -  a_{22} a_{44} \tau_{2}, a_{22} a_{24} \tau_{2}, a_{11} a_{22} -  a_{22}, a_{13} a_{24} -  a_{24}, a_{12} a_{31} -  a_{11} a_{32}, a_{22} a_{31}, \\[5pt]
&& a_{22} a_{32}, a_{14} a_{33} -  a_{13} a_{34}, a_{24} a_{33}, a_{24} a_{34}, a_{11} a_{42}, a_{12} a_{42}, a_{22} a_{42}, a_{31} a_{42} -  a_{42}, a_{13} a_{44}, \\[5pt]
&& a_{14} a_{44}, a_{24} a_{44}, a_{33} a_{44} -  a_{44}, a_{11} \sigma_{1} + a_{13} \tau_{1} -  \sigma_{2}, \\[5pt]
&& \left. a_{31} \sigma_{1} + a_{33} \tau_{1} -  \tau_{2}, a_{21}, a_{23}, a_{41}, a_{43}\right).
\end{eqnarray*}
By solving the ideal, we get that $\det A=0$.
Next we have
\begin{eqnarray*}
B_{31} &=&
\left( a_{11} a_{22} -  a_{22}, a_{12} a_{31} -  a_{11} a_{32}, a_{22} a_{31}, a_{22} a_{32}, a_{11} a_{42}, a_{12} a_{42}, \right.\\[5pt]
&& a_{22} a_{42}, a_{31} a_{42} -  a_{42}, a_{14} a_{44}, a_{14} \ve_{1}, a_{44} \ve_{1} + a_{44}, \\[5pt]
&& \left. a_{13}, a_{21}, a_{23}, a_{24}, a_{33} + \ve_{1}, a_{41}, a_{43}\right).
\end{eqnarray*}
Since $\ve_1\neq-1$, we get that the only solutions have $\det(A)=0$.
As for $B_{34}$, we have
\begin{eqnarray*}
B_{34} &=& \left( a_{11} a_{22} -  a_{22}, a_{12} a_{31} -  a_{11} a_{32}, a_{22} a_{31}, a_{22} a_{32}, a_{11} a_{42}, \right.\\[5pt]
&& a_{12} a_{42}, a_{22} a_{42}, a_{31} a_{42} -  a_{42}, a_{13} - 1, a_{14} -  a_{34}, a_{21}, a_{23}, \\[5pt]
&=& \left. a_{24}, a_{33} - 1, a_{41}, a_{43}, a_{44}\right).
\end{eqnarray*}
We easily get that it implies $\det A=0$.
As for $B_{35}$, we get
\begin{eqnarray*}
B_{35} &=&
\left( a_{11} a_{22} -  a_{22}, a_{12} a_{31} -  a_{11} a_{32}, a_{22} a_{31}, a_{22} a_{32}, a_{24} a_{34}, a_{11} a_{42}, \right.\\[5pt]
&& a_{12} a_{42}, a_{22} a_{42}, a_{31} a_{42} -  a_{42}, a_{14} a_{44}, a_{24} a_{44}, a_{24} \sigma_{1} + a_{24}, a_{34} \sigma_{1} -  a_{14} \tau_{1}, a_{44} \sigma_{1}, \\[5pt]
&& \left. a_{24} \tau_{1}, a_{44} \tau_{1} + a_{44}, a_{13} + \sigma_{1}, a_{21}, a_{23}, a_{33} + \tau_{1}, a_{41}, a_{43}\right).
\end{eqnarray*}
We get that $\det A=0$.
Consider now $B_{41}$. We have
\begin{eqnarray*}
B_{41} &=&
(a_{14} a_{22} \ve_{1} -  a_{22} a_{34}, a_{12} a_{24} \ve_{1} -  a_{24} a_{32}, a_{14} a_{42} \ve_{1} + a_{14} a_{42}, a_{12} a_{44} \ve_{1} + a_{12} a_{44},\\[5pt]
&& a_{42} a_{44} \ve_{1} + 2 a_{42} a_{44}, a_{13} a_{22} + a_{22}, a_{13} a_{24} -  a_{24}, a_{22} a_{24}, a_{13} a_{32} -  a_{12} a_{33} -  a_{12} \ve_{1}, \\[5pt]
&& a_{22} a_{32}, a_{14} a_{33} -  a_{13} a_{34}, a_{22} a_{33} + a_{22} \ve_{1}, a_{24} a_{33}, a_{24} a_{34}, a_{12} a_{42}, a_{13} a_{42}, a_{22} a_{42}, \\[5pt]
&& a_{24} a_{42}, a_{33} a_{42} + a_{42} \ve_{1} + a_{42}, a_{13} a_{44}, a_{14} a_{44}, a_{22} a_{44}, a_{24} a_{44}, a_{33} a_{44} -  a_{44}, a_{11} + a_{13}, \\[5pt]
&& a_{21}, a_{23}, a_{31} + a_{33} + \ve_{1}, a_{41}, a_{43})
\end{eqnarray*}
from which we get that $\det A=0$.
As for $B_{45}$, we have
\begin{eqnarray*}
B_{45} &=&
(a_{22} a_{24} \sigma_{1} + 2 a_{22} a_{24}, a_{24} a_{32} \sigma_{1} -  a_{12} a_{24} \tau_{1} + a_{24} a_{32}, a_{22} a_{34} \sigma_{1} -  a_{14} a_{22} \tau_{1} + a_{22} a_{34},\\[5pt]
&& a_{24} a_{42} \sigma_{1} + a_{24} a_{42}, a_{34} a_{42} \sigma_{1} -  a_{14} a_{42} \tau_{1} -  a_{14} a_{42}, a_{22} a_{44} \sigma_{1} + a_{22} a_{44},\\[5pt]
&& a_{32} a_{44} \sigma_{1} -  a_{12} a_{44} \tau_{1} -  a_{12} a_{44}, a_{42} a_{44} \sigma_{1}, a_{22} a_{24} \tau_{1}, a_{24} a_{42} \tau_{1} + a_{24} a_{42}, \\[5pt]
&& a_{22} a_{44} \tau_{1} + a_{22} a_{44}, a_{42} a_{44} \tau_{1} + 2 a_{42} a_{44}, a_{13} a_{22} + a_{22} \sigma_{1} + a_{22}, a_{13} a_{24} -  a_{24},\\[5pt]
&&  a_{13} a_{32} -  a_{12} a_{33} + a_{32} \sigma_{1} -  a_{12} \tau_{1}, a_{22} a_{32}, a_{14} a_{33} -  a_{13} a_{34}, a_{22} a_{33} + a_{22} \tau_{1}, a_{24} a_{33},\\[5pt]
&&  a_{24} a_{34}, a_{12} a_{42}, a_{13} a_{42} + a_{42} \sigma_{1}, a_{22} a_{42}, a_{33} a_{42} + a_{42} \tau_{1} + a_{42}, a_{13} a_{44}, a_{14} a_{44}, a_{24} a_{44},\\[5pt]
&&  a_{33} a_{44} -  a_{44}, a_{11} + a_{13} + \sigma_{1}, a_{21}, a_{23}, a_{31} + a_{33} + \tau_{1}, a_{41}, a_{43})
\end{eqnarray*}
As before, under the assumptions, we find $\det A=0$.
Finally, we need to study $B_{15}$. We have
\begin{eqnarray*}
B_{15}&=& 
(a_{11} a_{22} -  a_{22}, a_{13} a_{24} -  a_{24}, a_{12} a_{31} -  a_{11} a_{32}, a_{22} a_{31}, a_{22} a_{32}, a_{14} a_{33} -  a_{13} a_{34},\\[5pt]
&& a_{24} a_{33}, a_{24} a_{34}, a_{11} a_{42}, a_{12} a_{42}, a_{22} a_{42}, a_{31} a_{42} -  a_{42}, a_{13} a_{44}, a_{14} a_{44}, a_{24} a_{44},\\[5pt]
&&  a_{33} a_{44} -  a_{44}, a_{13} \ve_{1} -  \sigma_{1}, a_{24} \ve_{1} -  a_{24} \sigma_{1}, a_{33} \ve_{1} -  \tau_{1}, a_{44} \ve_{1} -  a_{44} \tau_{1}, a_{33} \sigma_{1} -  a_{13} \tau_{1}, \\[5pt]
&&  a_{34} \sigma_{1} -  a_{14} \tau_{1}, a_{44} \sigma_{1}, a_{24} \tau_{1}, a_{21}, a_{23}, a_{41}, a_{43})
\end{eqnarray*}
This has no solution, if we require $\det A\neq0$.

\subsection{$\mathfrak{r}_2'$, $(0,0,-13+24,-14-23)$} The generic closed 1-form (and candidate for the Lee form) is seen to be $\vt=\vt_1e^1+\vt_2e^2$, $\vt_1^2+\vt_2^2\neq0$ If we impose, furthermore, the conformally closedness of the generic $2$-form $\Omega=\sum_{1\leq j < k \leq 4} \omega_{jk} e^{jk}$, we obtain the equations
\begin{enumerate}
\item $(\vt_1+1)\om_{23}-\om_{14}-\vt_2\om_{13}=0$
\item $(\vt_1+1)\om_{24}+\om_{13}-\vt_2\om_{14}=0$
\item $(\vt_1+1)\om_{34}=0$
\item $\vt_2\om_{34}=0$
\item $\vt_1^2+\vt_2^2\neq0$
\end{enumerate}

We start by assuming $\vt_2\neq 0$. (4) implies $\om_{34}=0$; if $\vt_1=-1$, then (1) gives $\om_{14}=-\vt_2\om_{13}$ and (2) gives $\om_{13}=\vt_2\om_{14}$; together, these conditions imply $\om_{13}=0$, which in turn says $\om_{14}=0$, contradicting non-degeneracy. Then $\vt_1\neq -1$ and combining (1) and (2) we obtain
\[
\om_{14}=\frac{\om_{13}+(1+\vt_1)\om_{24}}{\vt_2} \quad \textrm{and} \quad \om_{23}=\frac{(1+\vt_2^2)\om_{13}+(1+\vt_1)\om_{24}}{\vt_2(1+\vt_1)}.
\]
The non-degeneracy condition reads then
\[
\vt_2^2\om_{13}^2+(\om_{13}+(1+\vt_1)\om_{24})^2\neq 0\;,
\] 
which is equivalent to $\om_{13}^2+\om_{24}^2\neq 0$.
The generic lcs structure is then $\vt=\vt_1e^1+\vt_2e^2$ and
\[
\Omega=\om_{12}e^{12}+\omega_{13}e^{13}+\frac{\om_{13}+(1+\vt_1)\om_{24}}{\vt_2}e^{14}+\frac{(1+\vt_2^2)\om_{13}+(1+\vt_1)\om_{24}}{\vt_2(1+\vt_1)}e^{23}+\om_{24}e^{24}\;.
\]
In terms of the basis $\{e^1,e^2,e^3,e^4\}$ of $(\mathfrak{r}_2')^*$ we consider the automorphism 
\begin{equation}\label{aut:1}
\left(\begin{array}{cccc}
1 & 0 & x & -y \\
0 & 1 & y & x \\
0 & 0 & 1 & 0 \\
0 & 0 & 0 & 1
\end{array}\right) \;,
\end{equation}
where $x,y$ are parameters to be determined. In the new basis, the coefficient of $e^{12}$ in the new expression for $\Omega$ is
\begin{equation}\label{eq:109}
{\om_{12}+x\left(\frac{(\vt_1-\vt_2^2)\om_{13}+\vt_1(1+\vt_1)\om_{24}}{\vt_2(1+\vt_1)}\right)+y(\om_{13}+\om_{24})};
\end{equation}
one sees that the coefficients of $x$ and $y$ vanish simultaneously if and only if $\om_{13}=\om_{24}=0$, which would contradict the non-degeneracy hypothesis. This means that we can always choose $x$ and $y$ so that \eqref{eq:109} vanishes. The corresponding automorphism brings then $\Omega$ into
\begin{equation}\label{eq:111}
\Omega'=\omega_{13}e^{13}+\frac{\om_{13}+(1+\vt_1)\om_{24}}{\vt_2}e^{14}+\frac{(1+\vt_2^2)\om_{13}+(1+\vt_1)\om_{24}}{\vt_2(1+\vt_1)}e^{23}+\om_{24}e^{24}\;,
\end{equation}
while fixing $\vt$. Next, we consider an automorphism
\begin{equation}\label{aut:2}
\left(\begin{array}{cccc}
1 & 0 & 0 & 0 \\
0 & 1 & 0 & 0 \\
0 & 0 & \cos z & -\sin z \\
0 & 0 & \sin z & \cos z
\end{array}\right) \;,
\end{equation}
where $z$ is a parameter to be determined. In the expression for $\Omega'$ in the new basis, the coefficient of $e^{23}$ is
\begin{equation}\label{eq:110}
\frac{(1+\vt_2^2)\om_{13}+(1+\vt_1)\om_{24}}{\vt_2(1+\vt_1)}\cos z-\om_{24}\sin z;
\end{equation}
Assuming $\om_{24}\neq 0$, we choose $z$ such that $\tan z=\frac{(1+\vt_2^2)\om_{13}+(1+\vt_1)\om_{24}}{\vt_2(1+\vt_1)\om_{24}}$. If $\om_{24}=0$, we choose $z=-\frac{\pi}{2}$. In any case, under this automorphism the Lee form is unchanged, while
\[
\Omega''=\omega_{13}''e^{13}+\frac{\om_{13}''+(1+\vt_1)\om_{24}''}{\vt_2}e^{14}+\om_{24}''e^{24}\;,
\]
with
\[
\om_{13}''=\om_{13}\cos z-\frac{\om_{13}+(1+\vt_1)\om_{24}}{\vt_2}\sin z
\]
and
\[
\om_{24}''=\om_{24}\cos z+\frac{(1+\vt_2^2)\om_{13}+(1+\vt_1)\om_{24}}{\vt_2(1+\vt_1)}\sin z \,.
\]
The non-degeneracy guarantees $\om_{13}''\omega_{24}''\neq 0$. Considering the automorphism 
\begin{equation}\label{aut:5}
\left(\begin{array}{cccc}
1 & 0 & 0 & 0 \\
0 & 1 & 0 & 0 \\
0 & 0 & t & 0 \\
0 & 0 & 0 & t
\end{array}\right) \;,
\end{equation}
with $t=\frac{1}{\om_{13}''}$ we obtain
\begin{equation}\label{eq:112}
\Omega'''=e^{13}+\frac{1}{\vt_2}\left(1+(1+\vt_1)\frac{\om_{24}''}{\om_{13}''}\right)e^{14}+\frac{\om_{24}''}{\om_{13}''}e^{24}\;;
\end{equation}
A computation shows that $\frac{\om_{24}''}{\om_{13}''}=-\frac{1+\vt_2^2}{1+\vt_1}$, hence
\[
\Omega'''=e^{13}-\vt_2 e^{14}-\frac{1+\vt_2^2}{1+\vt_1}e^{24}\;.
\]
Finally, using the automorphism
\begin{equation}\label{aut:3}
\left(\begin{array}{cccc}
1 & 0 & 0 & 0 \\
0 & -1 & 0 & 0 \\
0 & 0 & 1 & 0 \\
0 & 0 & 0 & -1
\end{array}\right) \;
\end{equation}
we can assume that $\vt_2>0$. This gives the normal form
\begin{equation}\label{model}
\left\{ \begin{array}{ccl}
         \vt & = & \sigma e^1+\tau e^2\\
        \Omega & = & e^{13}-\tau e^{14}-\frac{1+\tau^2}{1+\sigma} e^{24}, \, \sigma\neq-1, \ \tau> 0.
        \end{array} \right.
\end{equation}

We consider now the case $\vt_2=0$. This implies $\vt_1\neq 0$, and conditions (1), (2) and (3) above become
\[
\left\{ \begin{array}{ccl}
        (\vt_1+1)\om_{23} & = & \om_{14}\\
        (\vt_1+1)\om_{24} & = & -\om_{13}\\
        (\vt_1+2)\om_{34} & = & 0
        \end{array} \right.
\]
The case $\vt_1=-1$ is excluded by the non-degeneracy, hence we consider the two cases $\om_{34}\neq 0$ and $\om_{34}=0$. Let us start with $\om_{34}\neq 0$, which implies $\vt_1=-2$; the above conditions give then
\[
\om_{23}=-\om_{14} \quad \textrm{and} \quad \om_{24}=\om_{13},
\]
and the non-degeneracy condition becomes $\om_{12}\om_{34}-\om_{13}^2-\om_{14}^2\neq 0$. The Lee form is $\vt=-2e^1$, while the generic lcs 2-form is then
\[
\Omega=\om_{12}e^{12}+\om_{13}(e^{13}+e^{24})+\om_{14}(e^{14}-e^{23})+\om_{34}e^{34}.
\]
The automorphism \eqref{aut:1} with $x=-\frac{\om_{14}}{\om_{34}}$ and $y=-\frac{\om_{13}}{\om_{34}}$
clearly fixes the Lee form, while the 2-form is transformed into
\[
\Omega'=\frac{\om_{12}\om_{34}-\om_{13}^2-\om_{14}^2}{\om_{34}}e^{12}+\om_{34}e^{34}.
\]
By the mean of the automorphism \eqref{aut:3} we can assume $\om_{34}>0$, hence \eqref{aut:5} with $t=\frac{1}{\om_{34}}$ allows us to conclude that the normal forms is
\[
\left\{ \begin{array}{ccl}
         \vt & = & -2e^1\\
        \Omega & = & \ve e^{12}+e^{34}, \, \ve\neq 0
\end{array} \right.
\]

The above forms are not equivalent for different values of $\ve$. Let us assume that $A=(a_{jk})$ is a morphism of the Lie algebra, with respect to the basis $\{e^j\}$, preserving $\theta=-2e^1$ and sending $\Omega_{1}\coloneq\ve_1 e^{12}+e^{34}$ to $\Omega_{2}\coloneq\ve_2 e^{12}+e^{34}$. These conditions amount to solve the ideal with Gr\"obner basis
\begin{eqnarray*}
B &=& (a_{13}^{2} -  a_{24}^{2}, a_{13} a_{14} + a_{23} a_{24}, a_{14}^{2} -  a_{23}^{2}, a_{13} a_{23} + a_{14} a_{24}, a_{13} a_{44} -  a_{24}, a_{14} a_{44} + a_{23}, \\[5pt]
&& a_{23} a_{44} + a_{14}, a_{24} a_{44} -  a_{13}, a_{44}^{2} - 1, a_{14} \varepsilon_2 + a_{14}, a_{23} \varepsilon_{2} + a_{23}, a_{11} - 1, a_{12}, a_{21}, \\[5pt]
&&  a_{22} -  a_{44}, a_{31}, a_{32}, a_{33} - 1, a_{34}, a_{41}, a_{42}, a_{43}, \varepsilon_{1} -  \varepsilon_{2}),
\end{eqnarray*}
from which it follows that $\ve_1=\ve_2$.

If $\om_{34}=0$, then we get $\om_{13}=-(\vt_1+1)\om_{24}$, $\om_{14}=(\vt_1+1)\om_{23}$ and non-degeneracy yields $\om_{23}^2+\om_{24}^2\neq 0$. The Lee form is $\vt=\vt_1e^1$, with $\vt_1\neq -1$, and the generic 2-form is 
\[
\Omega=\om_{12}e^{12}-(\vt_1+1)\om_{24}e^{13}+(\vt_1+1)\om_{23}e^{14}+\om_{23}e^{23}+\om_{24}e^{24}.
\]
Using the automorphism \eqref{aut:1} we see that non-degeneracy guarantees the existence of numbers $x,y\in\mathbb{R}$ such that the coefficient of $e^{12}$ in the new expression for $\Omega$ vanishes, while the others remain unaltered, i.e.
\[
\Omega'=-(\vt_1+1)\om_{24}e^{13}+(\vt_1+1)\om_{23}e^{14}+\om_{23}e^{23}+\om_{24}e^{24}.
\]
Using the automorphism \eqref{aut:2} with $\vt=-\frac{\pi}{2}$, we can without loss of generality assume that $\om_{24}\neq 0$. We consider the automorphism \eqref{aut:2} with $z$ such that $\tan z=\frac{\om_{23}}{\om_{24}}$. Again the Lee form is fixed, while
\[
\Omega''=-(\vt_1+1)\om_{24}''e^{13}+\om_{24}''e^{24}
\]
with $\om_{24}''=\om_{24}\cos z+\om_{23}\sin z$. Notice that non-degeneracy implies $(\om_{24}'')^2\neq 0$. Finally, \eqref{aut:5} with $t=-\frac{1}{(1+\vt_1)\om_{24}''}$ gives us the normal form
\[
\left\{ \begin{array}{ccl}
         \vt & = & \ve e^1\\
        \Omega & = & e^{13}-\frac{1}{\ve+1}e^{24}, \, \ve\notin\{0,-1\}
        \end{array} \right.
\]

In case $\ve=-2$, we show that the forms $\Omega_1\coloneq\alpha e^{12}+e^{34}$ with $\alpha\neq0$ and $\Omega_2\coloneq e^{13}+e^{24}$ are distinct. In fact, by computing the Gr\"obner basis of the ideal generated by the conditions that the generic matrix $A=(a_{jk})$ with respect to the basis $\{e^j\}$ gives an automorphism of the Lie algebra sending $\Omega_1$ to $\Omega_2$, we get that it is generated by 1.

To finish the proof that the above forms are actually different it still remains to show that there is no automorphism of the Lie algebra transforming one of the above Lee forms to another. Set $\vt_{1}\coloneq\sigma_1 e^1+\tau_1 e^2$, $\vt_{2}\coloneq\sigma_2 e^1+\tau_2 e^2$, (where $\sigma_1\neq-1$, $\sigma_2\neq-1$, $\tau_1>0$, $\tau_2>0$,) $\vt_3\coloneq-2 e^1$, $\vt_{4}\coloneq \ve_1 e^1$, $\vt_{5}\coloneq\ve_2 e^1$ (where $\ve_1,\ve_2\not\in\{0,-1,-2\}$), and consider the generic linear map associated to $A=(a_{jk})$ with respect to the basis $\{e^j\}$. Let $B_{jk}$ the Gr\"obner basis of the ideal obtained by requiring that $A$ gives an automorphism of the Lie algebra, and that $A$ sends $\vt_j$ to $\vt_k$.
We find that:
\begin{itemize}
\item $B_{12}$ contains $a_{41}$, $a_{42}$, and $(a_{43}^2+ a_{44}^2)(\sigma_1 - \sigma_2)$, from which it follows that if $A$ is invertible, then $\sigma_1=\sigma_2$. The ideal contains now $a_{12}a_{34}^2 + a_{12}a_{44}^2$. If $a_{34}=a_{44}=0$, we get that also $a_{43}=0$, whence $\det A=0$. Then assume that $a_{34}^2+a_{44}^2\neq0$. We get $a_{12}=0$, and $a_{21}=0$. The ideal contains now $a_{33}^2 - a_{44}^2$, from which $a_{44}=\pm a_{33}$. In case $a_{33}=a_{44}$, by the element $a_{44}\tau_1 - a_{33}\tau_2$ in the ideal, we conclude that $\tau_1=\tau_2$. In case $a_{44}=-a_{33}$, from $a_{33}\tau_1 + a_{33}\tau_2=0$, since $\tau_1>0$ and $\tau_2>0$, we get $a_{33}=a_{44}=0$, contradiction.
\item $B_{45}$ contains the element $(\varepsilon_1 - \varepsilon_2) \cdot (a_{43}^2 + a_{44}^2)$. Therefore, either $\varepsilon_1=\varepsilon_2$; or $a_{43}=a_{44}=0$. The ideal contains also $a_{41}$ and $a_{42}$, whence the latter case yields that $A$ is not invertible.
\item $B_{31}$ contains $a_{31}$, $a_{32}$, $\tau_1 (a_{34}^2 + a_{44}^2)$. Since $\tau_1>0$, it follows that $a_{31}=a_{32}=a_{34}=0$. Moreover, we see that $a_{41}=a_{42}=a_{44}=0$. The ideal contains now $a_{33}^2 + a_{43}^2$, from which it follows that $A$ is not invertible.
\item $B_{34}$ contains $a_{41}$, $a_{42}$, $(\ve_1 + 2)(a_{43}^2 + a_{44}^2)$, from which it follows that, since $\ve_1\neq-2$, $A$ is not invertible.
\item $B_{14}$ contains $a_{41}$, $a_{42}$, $\tau_1 (a_{34}^2 + a_{44}^2)$, $\varepsilon_1 (a_{34} - a_{43}) \cdot (a_{34} + a_{43})$. Since $\tau_1>0$ and $\ve_1\neq0$, we get that $a_{41}=a_{42}=a_{43}=a_{44}=0$, therefore $A$ is not invertible.
\end{itemize}

\subsection{$\mathfrak{n}_4$, $(0,14,24,0)$} Taking a generic $1$-form $\vartheta=\sum_{j=1}^{4}\vartheta_je^j$ and imposing closedness, we see that the conditions $\vt_2=0=\vt_3$ must be satisfied. The generic Lee form is therefore $\vt=\vt_1e^1+\vt_4e^4$, with $\vt_1^2+\vt_4^2\neq0$. This, together with the equation $d_\vt\Omega=0$ for a generic $2$-form $\Omega=\sum_{1\leq j < k \leq 4} \omega_{jk} e^{jk}$, provides us with the following set of equations:
\begin{enumerate}
\item $\vartheta_1\omega_{23}=0$
\item $\omega_{13}+\vartheta_1\omega_{24}+\vartheta_4\omega_{12}=0$
\item $\omega_{23}+\vartheta_1\omega_{34}+\vartheta_4\omega_{13}=0$
\item $\vartheta_4\omega_{23}=0$
\item $\vt_1^2+\vt_4^2\neq0$
\end{enumerate}

Assume $\vt_1=0$. Then $\vt_4\neq 0$, which, in view of (4), implies $\omega_{23}=0$; plugging this into (3) gives $\omega_{13}=0$, which gives, with (2), $\omega_{12}=0$. This contradicts however \eqref{eq:non_deg}.

We assume therefore $\vt_1\neq 0$, which implies $\omega_{23}=0$ in view of (1). Plugging this into (3) gives $\omega_{34}=-\frac{\vt_4\omega_{13}}{\vt_1}$, while from (2) follows $\omega_{24}=\frac{-\omega_{13}-\vt_4\omega_{12}}{\vt_1}$. Equation \eqref{eq:non_deg} reduces then to $\frac{\omega_{13}^2}{\vt_1}\neq 0$, ensuring that $\omega_{13}\neq 0$.

The generic lcs structure is then given by
\[
\vt=\vt_1e^1+\vt_4e^4 \quad \textrm{and} \quad \Omega=\om_{12}e^{12}+\omega_{13}e^{13}+\om_{14}e^{14}+\frac{-\om_{13}-\vt_4\om_{12}}{\vt_1}e^{24}-\frac{\vt_4\om_{13}}{\vt_1}e^{34}\;.
\]

In terms of the basis $\{e^1,e^2,e^3,e^4\}$ of $\mathfrak{n}_4^*$, we consider the automorphism
\[
\left(\begin{array}{cccc}
\frac{1}{\vt_1} & 0 & 0 & 0 \\
0 & 1 & 0 & 0 \\
0 & 0 & \vt_1 & 0 \\
-\vt_4 & 0 & 0 & \vt_1
\end{array}\right) \;.
\]
In the new basis, we have
\[
\vt'=e^1 \quad \textrm{and} \quad \Omega'=\frac{\omega_{12}}{\vt_1}e^{12}+\om_{13}e^{13}+\omega_{14}e^{14}-\om_{13}e^{24}\;.
\]
We consider now the automorphism
\[
\left(\begin{array}{cccc}
1 & -\frac{\omega_{12}}{\vt_1\omega_{13}} & \frac{\omega_{12}^2}{2\vt_1^2\omega_{13}^2} & 0 \\
0 & 1 & -\frac{\omega_{12}}{\vt_1\omega_{13}} & 0 \\
0 & 0 & 1 & 0 \\
0 & 0 & -\frac{\omega_{12}+\vt_1\omega_{14}}{\vt_1\omega_{13}} & 1
\end{array}\right) \;.
\]
In these new basis,
\[
\vt''=e^1 \quad \textrm{and} \quad \Omega''=\om_{13}(e^{13}-e^{24})\;.
\]
According to the sign of $\omega_{13}$ we consider the automorphisms
\[
\left(\begin{array}{cccc}
1 & 0 & 0 & 0 \\
0 & \sqrt{\pm\frac{1}{\om_{13}}} & 0 & 0 \\
0 & 0 & \pm\frac{1}{\om_{13}} & 0 \\
0 & 0 & 0 & \sqrt{\pm\frac{1}{\om_{13}}}
\end{array}\right)\;;
\]
the normal form for a lcs structure on $\mathfrak{n}_4$ is
\[
\left\{ \begin{array}{ccl}
         \vt & = & e^1\\
        \Omega_\pm & = & \pm(e^{13}- e^{24})
        \end{array} \right.
\]

The above structures are different because the Gr\"obner basis giving possible automorphisms interchanging them is
$$\left(a_{44}^{2} + 1, a_{11} - 1, a_{12}, a_{14}, a_{21}, a_{22} -a_{44}, a_{23}, a_{24}, a_{31}, a_{32}, a_{33} + 1, a_{34}, a_{41}, a_{43}\right) $$
which gives an ideal with empty variety.

\subsection{$\mathfrak{r}_4$, $(14+24,24+34,34,0)$} 

Taking a generic $1$-form $\vartheta=\sum_{j=1}^{4}\vartheta_je^j$ and imposing closedness, we see that the conditions $\vt_1=\vt_2=\vt_3=0$ must be satisfied. The generic Lee form is therefore $\vt=\vt_4e^4$, with $\vt_4\neq0$. This, together with the equation $d_\vt\Omega=0$ for a generic $2$-form $\Omega=\sum_{1\leq j < k \leq 4} \omega_{jk} e^{jk}$, provides us with the following set of equations:
\begin{enumerate}
\item $(\vt_4+2)\om_{12}=0$
\item $(\vt_4+2)\om_{13}+\om_{12}=0$
\item $(\vt_4+2)\om_{23}+\om_{13}=0$
\item $\vartheta_4\neq 0$
\end{enumerate}

Assume $\vt_4\neq -2$. Then from (1) follows $\om_{12}=0$; plugging this into (2) gives $\om_{13}=0$, which gives, together with (3), $\om_{23}=0$. But this contradicts \eqref{eq:non_deg}.

We assume therefore $\vt_4=-2$; then (2) gives $\om_{12}=0$ and (3) gives $\om_{13}=0$. The generic lcs structure is then given by
\[
\vt=-2e^4 \quad \textrm{and} \quad \Omega=\om_{14}e^{14}+\om_{23}e^{23}+\om_{24}e^{24}+\om_{34}e^{34}\;,
\]
with the non-degeneracy condition $\om_{14}\om_{23}\neq 0$.
In terms of the basis $\{e^1,e^2,e^3,e^4\}$ of $\mathfrak{r}_4^*$, we consider the automorphism
\[
\left(\begin{array}{cccc}
\frac{1}{\om_{14}} & 0 & 0 & 0 \\
0 & \frac{1}{\om_{14}} & 0 & 0 \\
0 & 0 & \frac{1}{\om_{14}} & 0 \\
0 & \frac{\om_{34}}{\om_{23}} & -\frac{\om_{24}}{\om_{23}} & 1
\end{array}\right) \;,
\]
which gives the normal forms on $\mathfrak{r}_4$:
\[
\left\{ \begin{array}{ccl}
         \vt & = & -2e^4\\
        \Omega & = & e^{14}+\sigma e^{23}, \quad \sigma\neq0 .
        \end{array} \right.
\]

The above structures are different. Indeed, the automorphisms associated to $A=(a_{jk})$ (with respect to the basis $\{e^j\}$) interchanging two forms $\Omega_1\coloneq e^{14}+\sigma_1 e^{23}$ and $\Omega_2\coloneq e^{14}+\sigma_2 e^{23}$ (where $\sigma_1,\sigma_2\neq0$) and preserving $\vt=-2e^4$, should belong to the variety associated to the ideal with Gr\"obner basis
\begin{eqnarray*}
B&=&\left( a_{32}^{2} + a_{42} \sigma_{2} -  a_{31}, a_{43} \sigma_{2} + a_{32}, a_{11} - 1, a_{12}, a_{13}, a_{14}, a_{21} -  a_{32}, a_{22} - 1, \right.\\[5pt]
&&\left.a_{23}, a_{24}, a_{33} - 1, a_{34}, a_{44} - 1, \sigma_{1} -  \sigma_{2}\right). \end{eqnarray*}
But the ideal contains $\sigma_{1} -  \sigma_{2}$.

\subsection{$\mathfrak{r}_{4,\mu}$, $(14,\mu 24+34,\mu 34,0)$}

We suppose first that $\mu=0$. A 1-form $\vartheta=\sum_{j=1}^{4}\vartheta_j e^j$ is closed if and only if $\vt_1=\vt_2=0$. The generic Lee form is then $\vt=\vt_3e^3+\vt_4e^4$ with $\vt_3^2+\vt_4^2\neq0$. A 2-form $\Omega=\sum_{1\leq j < k \leq 4} \omega_{jk}e^{jk}$ is conformally closed with respect to $\vt$ if and only if: 
\begin{enumerate}
\item $\vartheta_3\omega_{12}=0$
\item $(\vt_4+1)\omega_{12}=0$
\item $(\vt_4+1)\omega_{13}+\om_{12}-\vt_3\om_{14}=0$
\item $\vt_3\omega_{24}-\vt_4\omega_{23}=0$
\end{enumerate}

We assume first that $\vt_3\neq 0$. Then $\om_{12}=0$ follows from (1), so that $\om_{14}=\frac{(\vt_4+1)\om_{13}}{\vt_3}$ by (3) and $\om_{24}=\frac{\vt_4\om_{23}}{\vt_3}$ by (4). Under all these hypotheses, \eqref{eq:non_deg} yields $\om_{13}\om_{23}\neq 0$ and the lcs structure reads
\[
\vt=\vt_3e^3+\vt_4e^4 \quad \textrm{and} \quad \Omega=\om_{13}e^{13}+\frac{(\vt_4+1)\om_{13}}{\vt_3}e^{14}+\om_{23}e^{23}+\frac{\vt_4\om_{23}}{\vt_3}e^{24}+\om_{34}e^{34}\;.
\]
In terms of the basis $\{e^1,e^2,e^3,e^4\}$ of $(\fr_{4,0})^*$, we consider the automorphism
\[
\left(\begin{array}{cccc}
\frac{\vt_3}{\om_{13}} & 0 & 0 & 0 \\
0 & \frac{1}{\vt_3} & 0 & 0 \\
0 & 0 & \frac{1}{\vt_3} & 0 \\
0 & \frac{\om_{34}}{\om_{23}} & -\frac{\vt_4}{\vt_3} & 1
\end{array}\right) \;,
\]
which gives the normal form
\[
\left\{ \begin{array}{ccl}
         \vt & = & e^3\\
        \Omega & = & e^{13}+e^{14}+\ve e^{23}, \, \ve\neq 0
        \end{array} \right.
\]

The above forms are non-equivalent. Indeed, take $\Omega_1\coloneq e^{13}+e^{14}+\ve_1 e^{23}$ and $\Omega_2\coloneq e^{13}+e^{14}+\ve_2 e^{23}$ for some $\ve_1,\ve_2\neq0$. By requiring that the generic $A=(a_{jk})$ (with respect to the basis $\{e^j\}_j$) is an automorphism sending $\Omega_1$ to $\Omega_2$, we are reduced to solve the ideal with Gr\"obner basis
\begin{eqnarray*}
B&=&(a_{42} \ve_{2} -  a_{32} + a_{41}, a_{11} - 1, a_{12}, a_{13}, a_{14}, a_{21}, a_{22} - 1, a_{23}, a_{24}, a_{31}, a_{33} - 1,\\[5pt]
&&a_{34}, a_{43}, a_{44} - 1, \ve_{1} -  \ve_{2})\,,
\end{eqnarray*}
from which we get $\ve_1=\ve_2$.

Assume next that $\vt_3=0$; then $\vt_4\neq 0$ and (4) implies $\om_{23}=0$. If $\vt_4\neq -1$, then $\om_{12}=0$ follows from (2) and $\om_{13}=0$ follows from (3); this, however, contradicts the \eqref{eq:non_deg}. We can then suppose that $\vt_4=-1$; then $\om_{12}=0$ follows from (3) and \eqref{eq:non_deg} implies $\om_{13}\om_{24}\neq 0$. Thus
\[
\vt=-e^4 \quad \textrm{and} \quad \Omega=\om_{13}e^{13}+\om_{14}e^{14}+\omega_{24}e^{24}+\omega_{34}e^{34}\;.
\]
The automorphism
\[
\left(\begin{array}{cccc}
\frac{\om_{24}}{\om_{13}} & 0 & 0 & 0 \\
0 & \frac{1}{\om_{24}} & 0 & 0 \\
0 & -\frac{\om_{34}}{\om_{24}^2} & \frac{1}{\om_{24}} & 0 \\
0 & 0 & -\frac{\om_{14}}{\om_{13}} & 1
\end{array}\right) \;
\]
shows that a lcs structure on $\fr_{4,0}$ with $\vt_3=0$ is equivalent to 
\[
\left\{ \begin{array}{ccl}
         \vt & = & -e^4\\
        \Omega & = & e^{13}+e^{24}
        \end{array} \right.
\]

The two lcs structures for $\mathfrak{r}_{4,0}$ are not equivalent. Indeed, consider the ideal assuring that $A=(a_{jk})$ with respect to $\{e^j\}$ is a morphism of the Lie algebra transforming $\vt_1\coloneq e^3$ into $\vt_2\coloneq -e^4$. Then we have to solve the ideal with Gr\"obner basis
\begin{eqnarray*}
B&=&
\left( a_{11} a_{21}, a_{21}^{2}, a_{11} a_{31}, a_{11} a_{34}, a_{21} a_{34}, a_{34} a_{41} -  a_{31} a_{44} + a_{21}, a_{11} a_{44} -  a_{11}, \right.\\[5pt]
&& \left. a_{21} a_{44}, a_{12}, a_{13}, a_{14}, a_{22} -  a_{34}, a_{23}, a_{24}, a_{33}, a_{43} + 1\right) .
\end{eqnarray*}
By solving it, we get that $A$ is not invertible.

We continue with the case $\mu\neq 0$. Take a generic $1$-form $\vt=\sum_{j=1}^{4}\vt_je^j$; the closedness condition implies that $\vt_1=\vt_2=\vt_3=0$. Thus the Lee form is $\vt=\vt_4e^4$, with $\vt_4\neq 0$. We compute the 2-cocycles of the Lichnerowicz differential $d_\vt$:
\begin{itemize}
\item $d_\vt(e^{12})=(-1-\mu-\vt_4)e^{124}-e^{134}$
\item $d_\vt(e^{13})=(-1-\mu-\vt_4)e^{134}$
\item $d_\vt(e^{14})=0$
\item $d_\vt(e^{23})=(-2\mu-\vt_4)e^{234}$
\item $d_\vt(e^{24})=0$
\item $d_\vt(e^{34})=0$
\end{itemize}
If $\vt_4\neq -\mu-1$ and $\vt_4\neq -2\mu$, there are not enough cocycles to give a non-degenerate $\Omega$. The solution of the equations $\vt_4= -\mu-1$ and $\vt_4= -2\mu$ is $(\mu,\vt_4)=(1,-2)$.

We study first the case $\vt_4=-\mu-1$, $\mu\neq 1$; since $\vt_4\neq 0$, we must assume $\mu\neq -1$. Thus
\[
\vt=(-\mu-1)e^4 \quad \textrm{and} \quad \Omega=\om_{13}e^{13}+\om_{14}e^{14}+\omega_{24}e^{24}+\omega_{34}e^{34}\;,
\]
with $\om_{13}\om_{24}\neq 0$. We consider, in the basis $\{e^1,e^2,e^3,e^4\}$ of $(\mathfrak{r}_{4,\mu})^*$, the automorphism
\[
\left(\begin{array}{cccc}
\frac{\om_{24}}{\om_{13}} & 0 & 0 & 0 \\
0 & \frac{1}{\om_{24}} & 0 & 0 \\
0 & -\frac{\om_{34}}{\om_{24}^2} & \frac{1}{\om_{24}} & 0 \\
0 & \frac{\om_{14}\om_{34}}{2\om_{13}\om_{24}} & -\frac{\om_{14}}{\om_{13}} & 1
\end{array}\right) \;;
\]
this gives
\[
\vt'=(-\mu-1)e^4 \quad \textrm{and} \quad \Omega'=e^{13}+e^{24}\;.
\]
The given lcs structure is then equivalent to
\[
\left\{ \begin{array}{ccl}
         \vt & = & (-\mu-1)e^4\\
        \Omega & = & e^{13}+e^{24}, \, \mu\neq \pm 1
        \end{array} \right.
\]
Assuming next $\vt_4=-2\mu$, $\mu\neq1$, the lcs structure is
\[
\vt=-2\mu e^4 \quad \textrm{and} \quad \Omega=\om_{14}e^{14}+\om_{23}e^{23}+\omega_{24}e^{24}+\omega_{34}e^{34}\;,
\]
with $\om_{14}\om_{23}\neq 0$. The automorphism
\[
\left(\begin{array}{cccc}
\frac{1}{\om_{14}} & 0 & 0 & 0 \\
0 & 1 & 0 & 0 \\
0 & 0 & 1 & 0 \\
0 & \frac{\om_{34}}{\om_{23}} & -\frac{\om_{24}}{\om_{23}} & 1
\end{array}\right) \;
\]
gives
\[
\vt'=-2\mu e^4 \quad \textrm{and} \quad \Omega'=e^{14}+\om_{23}e^{23}\;.
\]
According to the sign of $\om_{23}$, the automorphism
\[
\left(\begin{array}{cccc}
1 & 0 & 0 & 0 \\
0 & \frac{1}{\sqrt{\pm\om_{23}}} & 0 & 0 \\
0 & 0 & \frac{1}{\sqrt{\pm\om_{23}}} & 0 \\
0 & 0 & 0 & 1
\end{array}\right) \;
\]
gives the lcs structure
\begin{equation}\label{eq:108}
\left\{ \begin{array}{ccl}
         \vt & = & -2\mu e^4\\
        \Omega_{\pm} & = & e^{14}\pm e^{23}, \, \mu\neq 1
        \end{array} \right.
\end{equation}

The two forms are different. Indeed, by imposing to a generic automorphism $A=(a_{jk})$ (with respect to the basis $\{e^j\}$) to swap $e^{14}+e^{23}$ and $e^{14}-e^{23}$, we get an ideal in $\mathbb{R}[a_{jk},\mu]$ whose Gr\"obner basis contain the element $a_{33}^2 + 1$.

We show that, in case $\mu\not\in\{-1,0,1\}$, the lcs structures with Lee forms $\vt_1\coloneq-(\mu+1)e^4$ and $\vt_2\coloneq-2\mu e^4$ are not equivalent. Indeed, the Gr\"obner basis for the ideal generated by the conditions that $A=(a_{jk})$ with respect to the basis $\{e^j\}$ is a morphism of the Lie algebra and it sends $\vt_1$ to $\vt_2$ is
\begin{eqnarray*}
B &=& 
(a_{12} a_{31} \mu -  a_{12} a_{31} + \frac{2}{3} a_{13} a_{31}, a_{22} a_{31} \mu + a_{21} a_{32} \mu -  a_{22} a_{31} -  a_{21} a_{32} + 2 a_{21} a_{33},\\[5pt]
&& a_{13} a_{31} \mu -  a_{13} a_{31},a_{12} a_{32} \mu -  a_{12} a_{32} + \frac{2}{3} a_{13} a_{32}, a_{13} a_{32} \mu -  a_{13} a_{32}, a_{12} a_{33} \mu -  a_{12} a_{33}, \\[5pt]
&& a_{23} a_{32} \mu -  a_{22} a_{33} \mu -  a_{23} a_{32} + a_{22} a_{33} + 2 a_{23} a_{33}, a_{21} a_{33} \mu, a_{23} a_{33} \mu,  \\[5pt]
&& a_{31} a_{33} \mu -  a_{22} a_{31} -  a_{21} a_{32} + a_{21} a_{33}, a_{33}^{2} \mu + a_{23} a_{32} -  a_{22} a_{33} -  a_{23} a_{33}, \\[5pt]
&& a_{12} \mu^{2} - \frac{1}{2} a_{12} \mu + a_{13} \mu - \frac{1}{2} a_{12}, a_{13} \mu^{2} - \frac{1}{2} a_{13} \mu - \frac{1}{2} a_{13}, a_{23} \mu^{2} -  a_{23} \mu, \\[5pt]
&& a_{31} \mu^{2} + a_{21} \mu -  a_{31} \mu + a_{21}, a_{32} \mu^{2} -  a_{22} \mu -  a_{32} \mu + 2 a_{33} \mu -  a_{22}, \\[5pt]
&& a_{33} \mu^{2} -  a_{23} \mu -  a_{33} \mu -  a_{23}, a_{21} \mu^{2} -  a_{21} \mu, a_{22} \mu^{2} -  a_{22} \mu + 2 a_{23} \mu, a_{11} a_{13}, a_{13}^{2},a_{11} a_{21}, \\[5pt]
&& , a_{12} a_{21} - \frac{1}{3} a_{13} a_{31}, a_{13} a_{21}, a_{21}^{2}, a_{11} a_{22} -  a_{11} a_{33}, a_{12} a_{22} + \frac{1}{3} a_{13} a_{32} -  a_{12} a_{33}, a_{13} a_{22}, \\[5pt]
&&  a_{21} a_{22}, a_{22}^{2} + a_{23} a_{32} -  a_{22} a_{33} - 2 a_{23} a_{33}, a_{11} a_{23}, a_{12} a_{23}, a_{13} a_{23}, a_{21} a_{23}, a_{22} a_{23}, a_{23}^{2}, \\[5pt]
&& a_{23} a_{31} + a_{21} a_{33}, a_{13} a_{33}, a_{11} a_{44} -  a_{11}, a_{12} a_{44} - 2 a_{12} \mu - 2 a_{13} \mu + a_{12} + a_{13}, \\[5pt]
&& a_{13} a_{44} - 2 a_{13} \mu + a_{13},a_{21} a_{44} -  a_{21} \mu, a_{22} a_{44} -  a_{22} \mu -  a_{23} \mu, a_{23} a_{44} -  a_{23} \mu, \\[5pt]
&& a_{31} a_{44} -  a_{31} \mu -  a_{21},a_{32} a_{44} -  a_{32} \mu -  a_{33} \mu + a_{22} + a_{23}, a_{33} a_{44} -  a_{33} \mu + a_{23},  \\[5pt]
&& a_{11} \mu -  a_{11}, a_{44} \mu + a_{44} - 2 \mu,a_{14}, a_{24}, a_{34})\,.
\end{eqnarray*}
By solving it, we are reduced to $a_{21}=a_{22}=a_{23}=a_{24}=0$, then $\det A=0$.

The last case is $\mu=1$, giving $\vt_4=-2$. Here then
\[
\vt=-2 e^4 \quad \textrm{and} \quad \Omega=\om_{13}e^{13}+\om_{14}e^{14}+\om_{23}e^{23}+\omega_{24}e^{24}+\omega_{34}e^{34}\;
\]
with $\om_{13}\om_{24}-\om_{14}\om_{23}\neq 0$. Notice that if $\om_{13}=0$, then we are in the previous case, and we obtain the model \eqref{eq:108} with $\mu=1$. We assume next that $\om_{13}\neq 0$. The automorphism
\[
\left(\begin{array}{cccc}
1 & 0 & 0 & 0 \\
0 & 1 & 0 & 0 \\
0 & 0 & 1 & 0 \\
\frac{\om_{34}}{\om_{13}} & 0 & -\frac{\om_{14}}{\om_{13}} & 1
\end{array}\right) \;
\]
transforms the lcs structure into
\[
\vt'=-2 e^4 \quad \textrm{and} \quad \Omega'=\om_{13}e^{13}+\om_{24}'e^{24}+\om_{23}e^{23}\;
\]
with $\om_{24}'=\frac{\om_{13}\om_{24}-\om_{14}\om_{23}}{\om_{13}}\neq 0$. The automorphism
\[
\left(\begin{array}{cccc}
\frac{\om_{24}'}{\om_{13}} & 0 & 0 & 0 \\
0 & \frac{1}{\om_{24}'} & 0 & 0 \\
0 & 0 & \frac{1}{\om_{24}'} & 0 \\
0 & 0 & 0 & 1
\end{array}\right) \;
\]
brings the lcs structure to
\[
\vt''=-2 e^4 \quad \textrm{and} \quad \Omega''=e^{13}+e^{24}+\frac{\om_{23}}{(\om_{24}')^2}e^{23}\;.
\]
The model for this case is then
\[
\left\{ \begin{array}{ccl}
         \vt & = & -2 e^4\\
        \Omega & = & e^{13}+e^{24}+\ve e^{23}, \, \ve\in\mathbb{R}
        \end{array} \right.
\]

The above forms are not equivalent. Indeed, take $\Omega_1\coloneq e^{13}+e^{14}+\ve_1 e^{23}$ and $\Omega_2\coloneq e^{13}+e^{14}+\ve_2 e^{23}$ for some $\ve_1,\ve_2\neq0$. By requiring the generic $A=(a_{jk})$ (with respect to the basis $\{e^j\}$) to being an automorphism swapping $\Omega_1$ and $\Omega_2$, we have to solve the ideal with Gr\"obner basis
\begin{eqnarray*}
B &=& 
(a_{31}^{2} a_{43}^{2} + a_{41}^{2} a_{43} \ve_{2} -  a_{11} a_{32}^{2} -  a_{33} a_{41}^{2} + a_{31} a_{33} a_{42} + a_{32} a_{42} \ve_{1} -  a_{41} a_{42}\ve_{2} + 2 a_{32} a_{41} \\[5pt]
&& -  a_{31} a_{42},a_{31} a_{33} a_{43} + a_{41} a_{43} \ve_{2} -  a_{33} a_{41} -  a_{42}\ve_{2} + a_{32}, a_{33} a_{41} a_{43} -  a_{31} a_{43}^{2} + a_{12} a_{32} \\[5pt]
&& -  a_{33} a_{42} + a_{42}, a_{11} a_{43}^{2} -  a_{12}^{2}, a_{11} a_{32} \ve_{1} -  a_{31} a_{43}\ve_{1} -  a_{42}\ve_{1}^{2} + a_{11} a_{31} -  a_{41}\ve_{1} -  a_{31},\\[5pt]
&& a_{33} a_{42} \ve_{1} -  a_{32} a_{43}\ve_{1} + a_{33} a_{41} -  a_{31} a_{43} -  a_{32}, a_{33} a_{43} \ve_{1} + a_{33} - 1, a_{43}^{2} \ve_{1} + a_{12} + a_{43},\\[5pt]
&& a_{42}\ve_{1}\ve_{2} + a_{31} a_{33} -  a_{32}\ve_{1} + a_{41} \ve_{2} -  a_{31}, a_{43} \ve_{1}\ve_{2} -  a_{33} \ve_{1} +\ve_{2}, a_{12} a_{31} -  a_{11} a_{32} + a_{42} \ve_{1} \\[5pt]
&& + a_{41}, a_{11} a_{33} -  a_{43}\ve_{1} - 1, a_{12} a_{33} + a_{43}, a_{33}^{2} + a_{43} \ve_{2} -  a_{33}, a_{12}\ve_{1} -  a_{43}\ve_{1} + a_{11} - 1, \\[5pt]
&& a_{11}\ve_{2} - \ve_{1}, a_{12}\ve_{2} -  a_{33} + 1, a_{13}, a_{14}, a_{21}, a_{22} -  a_{33}, a_{23}, a_{24}, a_{34}, a_{44} - 1 )\,.
\end{eqnarray*}
from which we get $\ve_1=\ve_2$.

\subsection{$\mathfrak{r}_{4,\alpha,\beta}$, $(14,\alpha 24,\beta 34,0)$}

Since $\alpha\beta\neq 0$, the only closed 1-form is $e^4$ and the Lee form ist $\vt=\vt_4e^4$ with $\vt_4\neq 0$. For a generic 2-form $\Omega=\sum_{1\leq j < k \leq 4} \omega_{jk} e^{jk}$, the conformally closedness $d_\vt\Omega=0$ provides the following equations:
\begin{enumerate}[(1)]
\item $(\vt_4+1+\alpha)\om_{12}=0$;
\item $(\vt_4+1+\beta)\om_{13}=0$;
\item $(\vt_4+\alpha+\beta)\om_{23}=0$.
\end{enumerate}

We consider first the case $\alpha=\beta\neq 1$. If $\vt_4\notin\{-2\alpha,-\alpha-1\}$ the above conditions imply $\om_{12}=\om_{13}=\om_{23}=0$, which contradicts non-degeneracy \eqref{eq:non_deg}. For $\vt_4=-2\alpha$ we get the lcs structure
\[
\vt=-2\alpha e^4, \quad \Om=\om_{14}e^{14}+\om_{23}e^{23}+\om_{24}e^{24}+\om_{34}e^{34}
\]
with $\om_{14}\om_{23}\neq 0$. The automorphism
\begin{equation}\label{eq:r4alphabeta:1}
\left(\begin{array}{cccc}
x & 0 & 0 & 0 \\
0 & y & 0 & 0 \\
0 & 0 & z & 0 \\
a & b & c & 1
\end{array}\right)
\end{equation}
with $b=\frac{\om_{34}}{\om_{23}}$, $c=-\frac{\om_{24}}{\om_{23}^2}$, $x=\frac{1}{\om_{14}}$, $y=\frac{1}{\om_{23}}$, $z=1$ and $a=0$ gives the normal form
\[
\left\{ \begin{array}{ccl}
         \vt & = & -2\alpha e^4\\
        \Omega & = & e^{14}+e^{23}
        \end{array} \right.
\]
For $\vt_4=-1-\alpha$ we get $\om_{23}=0$ and the lcs structure
\[
\vt=(-1-\alpha)e^4, \quad \Om=\om_{12}e^{12}+\om_{13}e^{13}+\om_{14}e^{14}+\om_{24}e^{24}+\om_{34}e^{34}
\]
with non-degeneracy $\Delta=\om_{12}\om_{34}-\om_{13}\om_{24}\neq 0$. Using the automorphism
\begin{equation}\label{eq:r4alphabeta:2}
\left(\begin{array}{cccc}
1 & 0 & 0 & 0 \\
0 & 0 & 1 & 0 \\
0 & 1 & 0 & 0 \\
0 & 0 & 0 & 1
\end{array}\right)
\end{equation}
we can assume that $\om_{12}\neq 0$. The automorphism \eqref{eq:r4alphabeta:1} with $a=\frac{\om_{24}}{\om_{12}}$, $b=-\frac{\om_{14}}{\om_{12}}$, $c=0$ and $x=y=z=1$ transforms the structure into
\[
\vt'=(-1-\alpha)e^4 \quad \textrm{and} \quad \Om'=\om_{12}e^{12}+\om_{13}e^{13}+\frac{\Delta}{\om_{12}}e^{34}\,.
\]
Since $\alpha=\beta$, we can use the automorphism
\begin{equation}\label{eq:r4alphabeta:3}
\left(\begin{array}{cccc}
1 & 0 & 0 & 0 \\
0 & 1 & 0 & 0 \\
0 & p & 1 & 0 \\
0 & 0 & 0 & 1
\end{array}\right)
\end{equation}
with $p=-\frac{\om_{13}}{\om_{12}}$ to obtain $\vt''=\vt$ and $\Om''=\om_{12}e^{12}+\frac{\Delta}{\om_{12}}e^{34}$. Finally, the automorphism \eqref{eq:r4alphabeta:1} with $a=b=c=0$, $x=\frac{1}{\om_{12}}$, $y=1$ and $z=\frac{\om_{12}}{\Delta}$ gives the normal form
\[
\left\{ \begin{array}{ccl}
         \vt & = & (-1-\alpha)e^4\\
        \Omega & = & e^{12}+e^{34}
        \end{array} \right.
\]

For $\alpha=\beta=1$ and $\vt_4\neq -2$ non-degeneracy does not hold. Hence we get the lcs structure
\[
\vt=-2e^4, \quad \Om=\om_{12}e^{12}+\om_{13}e^{13}+\om_{14}e^{14}+\om_{23}e^{23}+\om_{24}e^{24}+\om_{34}e^{34}
\]
with $\om_{12}\om_{34}-\om_{13}\om_{24}+\om_{14}\om_{23}\neq 0$. Since $\alpha=\beta=1$, $\mathrm{SO}(3)$ sits into the automorphism group of $(\fr_{4,1,1})^*$ via the map $A\mapsto\textrm{diag}(A,1)$ and this action is transitive on spheres of a given radius contained in the abelian ideal $\{e^1,e^2,e^3\}$. This action lifts to an action on $\Lambda^2(\fr_{4,1,1})^*$ which is also transitive on spheres of a gives radius contained in $\{e^{12},e^{13},e^{23}\}$. This means that, for a convenient choice of basis in the ideal $\{e^1,e^2,e^3\}$, the gives lcs structure is
\[
\vt'=\vt \quad \mathrm{and} \quad \Om'=\om_{12}'e^{12}+\om_{14}'e^{14}+\om_{24}'e^{24}+\om_{34}'e^{34}\,,
\]
with $\om_{12}'\om_{34}'\neq 0$. The automorphism \eqref{eq:r4alphabeta:1} with $a=\frac{\om_{24}'}{\om_{12}'}$, $b=-\frac{\om_{14}'}{\om_{12}'}$, $c=0$, $x=\frac{1}{\om_{12}'}$, $y=1$ and $z=\frac{1}{\om_{34}'}$ gives the normal form
\[
\left\{ \begin{array}{ccl}
         \vt & = & -2e^4\\
        \Omega & = & e^{12}+e^{34}
        \end{array} \right.
\]

We proceed with the case $\alpha\neq \beta$. We assume first $\beta\neq 1$. The non-degeneracy condition forces $\vt_4\in\{-1-\alpha, -1-\beta, -\alpha-\beta\}$. If $\vt_4=-1-\alpha$, then $\vt_4+1+\beta\neq 0$ and $\vt_4+\alpha+\beta\neq 0$, which imply $\om_{13}=\om_{23}=0$. The lcs structure is then
\[
\vt=(-1-\alpha)e^4, \quad \Om=\om_{12}e^{12}+\om_{14}e^{14}+\om_{24}e^{24}+\om_{34}e^{34}
\]
with $\om_{12}\om_{34}\neq 0$. The automorphism \eqref{eq:r4alphabeta:1} with $a=\frac{\om_{24}}{\om_{12}}$, $b=-\frac{\om_{14}}{\om_{12}}$, $x=\frac{1}{\om_{12}}$, $y=1$, $z=\frac{1}{\om_{34}}$ and $c=0$ gives the normal form
\[
\left\{ \begin{array}{ccl}
         \vt & = & (-1-\alpha)e^4\\
        \Omega & = & e^{12}+e^{34}
        \end{array} \right.
\]

If $\vt_4=-1-\beta$, then $\vt_4+1+\alpha\neq 0$ and $\vt_4+\alpha+\beta\neq 0$, which imply $\om_{12}=\om_{23}=0$. The lcs structure is then
\[
\vt=(-1-\beta)e^4, \quad \Om=\om_{13}e^{13}+\om_{14}e^{14}+\om_{24}e^{24}+\om_{34}e^{34}
\]
with $\om_{13}\om_{24}\neq 0$. The automorphism \eqref{eq:r4alphabeta:1} with $a=\frac{\om_{34}}{\om_{13}}$, $c=-\frac{\om_{14}}{\om_{13}}$, $x=\frac{1}{\om_{13}}$, $y=\frac{1}{\om_{24}}$, $z=1$ and $b=0$ gives the normal form
\[
\left\{ \begin{array}{ccl}
         \vt & = & (-1-\beta)e^4\\
        \Omega & = & e^{13}+e^{24}
        \end{array} \right.
\]

If $\vt_4=-\alpha-\beta$, then $\vt_4+1+\alpha\neq 0$ and $\vt_4+1+\beta\neq 0$, which implies $\om_{12}=\om_{13}=0$. The lcs structure is then
\[
\vt=(-\alpha-\beta)e^4, \quad \Om=\om_{14}e^{14}+\om_{23}e^{23}+\om_{24}e^{24}+\om_{34}e^{34}
\]
with $\om_{14}\om_{23}\neq 0$. The automorphism \eqref{eq:r4alphabeta:1} with $b=\frac{\om_{34}}{\om_{23}}$, $c=-\frac{\om_{24}}{\om_{23}}$, $x=\frac{1}{\om_{14}}$, $y=\frac{1}{\om_{23}}$, $z=1$ and $a=0$ gives the normal form
\[
\left\{ \begin{array}{ccl}
         \vt & = & (-\alpha-\beta)e^4\\
        \Omega & = & e^{14}+e^{23}
        \end{array} \right.
\]

We continue with the case $\alpha\neq\beta=1$. If $\vt_4\notin\{-1-\alpha,-2\}$ then $\om_{12}=\om_{13}=\om_{23}=0$ and non-degeneracy does not hold. We consider the case $\vt_4=-1-\alpha$; this implies $\om_{13}=0$, and the lcs structure is
\[
\vt=(-1-\alpha)e^4, \quad \Om=\om_{12}e^{12}+\om_{14}e^{14}+\om_{23}e^{23}+\om_{24}e^{24}+\om_{34}e^{34}
\]
with $\Delta=\om_{12}\om_{34}+\om_{14}\om_{23}\neq 0$. In particular, either $\om_{12}$ or $\om_{23}$ must be non-zero. Assuming $\omega_{12}=0$, then $\omega_{23}\neq0$. Apply the automorphism
\[
\left(\begin{array}{cccc}
1 & 0 & 1 & 0 \\
0 & 1 & 0 & 0 \\
0 & 0 & 1 & 0 \\
0 & 0 & 0 & 1
\end{array}\right)
\]
to get $\Omega'=-\omega_{23}e^{12} + (\omega_{14} + \omega_{34})e^{14} + \omega_{23}e^{23} + \omega_{24}e^{24} + \omega_{34}e^{34}$, hence it suffices to study the case $\om_{12}\neq0$. In this case, the automorphism \eqref{eq:r4alphabeta:1}
with $a=\frac{\om_{24}}{\om_{12}}$, $b=-\frac{\om_{14}}{\om_{12}}$, $c=0$ and $x=y=z=1$ gives the structure
\[
\vt'=(-1-\alpha)e^4, \quad \Om'=\om_{12}e^{12}+\om_{23}e^{23}+\frac{\Delta}{\om_{12}}e^{34}\,.
\]
Since $\beta=1$, we can use the automorphism
\begin{equation}\label{eq:r4alphabeta:4}
\left(\begin{array}{cccc}
1 & 0 & 0 & 0 \\
0 & 1 & 0 & 0 \\
q & 0 & 1 & 0 \\
0 & 0 & 0 & 1
\end{array}\right)
\end{equation}
with $q=\frac{\om_{23}}{\om_{12}}$ to get $\vt''=\vt'$ and $\Om''=\om_{12}e^{12}+\frac{\Delta}{\om_{12}}e^{34}$. Finally, the automorphism \eqref{eq:r4alphabeta:1} with $a=b=c=0$, $x=\frac{1}{\om_{12}}$, $y=1$ and $z=\frac{\om_{12}}{\Delta}$ gives the normal form
\[
\left\{ \begin{array}{ccl}
         \vt & = & (-1-\alpha)e^4\\
        \Omega & = & e^{12}+e^{34}
        \end{array} \right.
\]

At last, we tackle the case $\vt_4=-2$; this implies $\om_{12}=\om_{23}=0$, and the lcs structure is
\[
\vt=-2e^4, \quad \Om=\om_{13}e^{13}+\om_{14}e^{14}+\om_{24}e^{24}+\om_{34}e^{34}
\]
with $\om_{13}\om_{24}\neq 0$. The automorphism \eqref{eq:r4alphabeta:1} with $a=\frac{\om_{34}}{\om_{13}}$, $c=-\frac{\om_{14}}{\om_{13}}$, $x=\frac{1}{\om_{13}}$, $y=\frac{1}{\om_{24}}$, $z=1$ and $b=0$ gives the normal form
\[
\left\{ \begin{array}{ccl}
         \vt & = & -2e^4\\
        \Omega & = & e^{13}+e^{24}
        \end{array} \right.
\]

To conclude, we have to show that there is no automorphism of the Lie algebra interchanging the possible lcs structures. As before, we denote by $B$ a Gr\"obner basis for the ideal generated by the conditions that $A=(a_{jk})$ with respect to $\{e^j\}$ yields a morphism of the Lie algebra transforming $\vt_1$ into $\vt_2$.
\begin{itemize}
\item For $\alpha\neq\beta$, consider the case $\vt_1\coloneq(-1-\alpha)e^4$ and $\vt_2\coloneq(-1-\beta)e^4$. Then $B$ contains $a_{31}(\alpha \beta - 1)$, $(\alpha -  \beta) a_{32}$, $\beta(\alpha -  \beta)a_{33}$ and $a_{34}$, from which it follows that $\det A=0$.
\item For $\beta\neq1$, consider the case $\vt_1\coloneq(-1-\alpha)e^4$ and $\vt_2\coloneq(-\alpha-\beta)e^4$. Then $B$ contains $(\beta - 1)\alpha a_{31}$, $-a_{32}(- \alpha^{2} + \beta)$, $\beta(\beta - 1)a_{33}$, $a_{34}$, from which it follows that $\det A=0$.
\item For $\alpha\neq\beta\neq1$, consider the case $\vt_1\coloneq(-1-\beta)e^4$ and $\vt_2\coloneq(-\alpha-\beta)e^4$. Then $B$ contains $a_{31}(\beta^{2} -  \alpha)$, $-a_{32}(- \alpha^{2} -  \alpha \beta + \beta^{2} + \beta)$, $\beta(\alpha - 1)a_{33}$, $a_{34}$, from which we get $\det A=0$.
\end{itemize}

\subsection{$\hat{\mathfrak{r}}_{4,\beta}$, $(14,-24,\beta 34,0)$}
Since the only closed element is $e^4$, the Lee form ist $\vt=\vt_4e^4$ with $\vt_4\neq 0$. For a generic 2-form $\Omega=\sum_{1\leq j < k \leq 4} \omega_{jk} e^{jk}$, the conformally closedness $d_\vt\Omega=0$ provides the following equations:
\begin{enumerate}[(1)]
\item $\vt_4\om_{12}=0$;
\item $(\vt_4+1+\beta)\om_{13}=0$;
\item $(\vt_4-1+\beta)\om_{23}=0$.
\end{enumerate}

The first equation implies $\om_{12}=0$. If $\vt_4\notin\{-1-\beta,1-\beta\}$ then $\om_{13}=\om_{23}=0$, which contradicts \eqref{eq:non_deg}. We start by assuming $\vt_4=-1-\beta$; since $\vt_4\neq 0$, we exclude the case $\beta=-1$. Then $\om_{23}=0$ and the generic lcs structure is 
\[
\vt=(-1-\beta) e^4 \quad \textrm{and} \quad \Om=\om_{13}e^{13}+\om_{14}e^{14}+\om_{24}e^{24}+\om_{34}e^{34}
\]
with $\om_{13}\om_{24}\neq 0$. The automorphism \eqref{eq:r4alphabeta:1} with $a=\frac{\om_{34}}{\om_{13}}$, $c=-\frac{\om_{14}}{\om_{13}}$, $x=\frac{1}{\om_{13}}$, $y=\frac{1}{\om_{24}}$, $z=1$, and $b=0$ gives the normal form (on $\hat{\mathfrak{r}}_{4,\beta}$ with $\beta\neq -1$)
\[
\left\{ \begin{array}{ccl}
         \vt & = & (-1-\beta) e^4\\
        \Omega & = & e^{13}+e^{24}
        \end{array} \right.
\]

We go on with $\vt_4=1-\beta$; we get $\om_{13}=0$ and the generic lcs structure is 
\[
\vt=(1-\beta) e^4 \quad \textrm{and} \quad \Om=\om_{14}e^{14}+\om_{23}e^{23}+\om_{24}e^{24}+\om_{34}e^{34}
\]
with $\om_{14}\om_{23}\neq 0$. The automorphism \eqref{eq:r4alphabeta:1} with $b=\frac{\om_{34}}{\om_{23}}$, $c=-\frac{\om_{24}}{\om_{23}}$, $x=\frac{1}{\om_{14}}$, $y=\frac{1}{\om_{23}}$, $z=1$, and $a=0$ gives the normal form
\[
\left\{ \begin{array}{ccl}
         \vt & = & (1-\beta) e^4\\
        \Omega & = & e^{14}+e^{23}
        \end{array} \right.
\]

We show that, in case $\beta\neq-1$, the lcs structures with Lee forms $\vt_1\coloneq(-1-\beta)e^4$ and $\vt_2\coloneq(1-\beta)e^4$ are not equivalent. Consider the ideal containing the conditions so that $A=(a_{jk})$ with respect to the basis $\{e^j\}$ is a morphism of the Lie algebra $\hat{\mathfrak{r}}_{4,\beta}$ transforming $\vt_1$ into $\vt_2$. We compute a Gr\"obner basis for it:
\begin{eqnarray*}
B &=&
(a_{13} \beta^{2} - 2 a_{13} \beta -  a_{13}, a_{23} \beta^{2} + a_{23}, a_{31} \beta^{2} + a_{31}, a_{32} \beta^{2} + 2 a_{32} \beta -  a_{32}, a_{12} a_{13},a_{13} a_{21},\\[5pt]
&&  a_{12} a_{23}, a_{13} a_{23}, a_{21} a_{23}, a_{12} a_{31}, a_{13} a_{31}, a_{21} a_{31}, a_{12} a_{32}, a_{13} a_{32},a_{21} a_{32}, a_{23} a_{32}, a_{31} a_{32},\\[5pt]
&&   a_{13} a_{33}, a_{23} a_{33}, a_{31} a_{33}, a_{32} a_{33}, a_{12} a_{44} + a_{12}, a_{13} a_{44} -  a_{13} \beta + 2 a_{13}, a_{21} a_{44} + a_{21},  \\[5pt]
&&  a_{23} a_{44} -  a_{23} \beta, a_{31} a_{44} -  a_{31} \beta, a_{32} a_{44} + a_{32} \beta, a_{33} a_{44} + a_{33}, a_{12} \beta, a_{21} \beta, a_{33} \beta,\\[5pt]
&&  a_{44} \beta + a_{44} -  \beta + 1,a_{11}, a_{14}, a_{22}, a_{24}, a_{34} ).
\end{eqnarray*}
By solving it, we get $\det A=0$.

\subsection{$\mathfrak{r}^\prime_{4,\gamma,\delta}$, $(14,\gamma 24+\delta 34,-\delta 24+\gamma 34,0)$}
The generic $1$-form $\vartheta$ has differential $d\vt=\vt_1e^{14} + (\gamma\vt_2 - \delta\vt_3)e^{24} + (\delta\vt_2 + \gamma\vt_3)e^{34}$. Then $d\vt=0$ if and only if $\vt_1=0$ and
\[ 
\left(\begin{array}{cc}\gamma&-\delta\\\delta&\gamma\end{array}\right) \left(\begin{array}{c}\vartheta_2\\\vartheta_3\end{array}\right)\;=\;\left(\begin{array}{c}0\\0\end{array}\right) \;.
\]
Since the matrix above is always invertible we get $\vartheta_2=\vartheta_3=0$ and the generic Lee form is $\vartheta=\vartheta_4 e^4$.

The condition $d_\vartheta\Omega=0$ yields
$$ \left(\begin{array}{ccc}
-\vartheta_4-(\gamma+1) & \delta & 0 \\
-\delta & -\vartheta_4-(\gamma+1) & 0 \\
0 & 0 & -2\gamma-\vartheta_4
\end{array}\right)
\left(\begin{array}{c}\omega_{12}\\\omega_{13}\\\omega_{23}\end{array}\right) \;=\;
\left(\begin{array}{c}0\\0\\0\end{array}\right) \;,$$
from which we get $\omega_{12}=\omega_{13}=0$. The non-degeneracy \eqref{eq:non_deg} becomes then $\om_{14}\om_{23}\neq 0$, which implies $\om_{23}\neq 0$. Hence we must have $\vt_4=-2\gamma$; in particular $\gamma\neq 0$.

The generic lcs structure is then
\[
\vt=-2\gamma e^4 \quad \textrm{and} \quad \Om=\om_{14}e^{14} + \om_{23}e^{23} + \om_{24}e^{24} + \om_{34}e^{34}\,.
\]
The automorphism
\[
\left(\begin{array}{cccc}
\frac{1}{\om_{14}} & 0 & 0 & 0 \\
0 & \frac{1}{\sqrt{\pm \om_{23}}} & 0 & 0 \\
0 & 0 & \frac{1}{\sqrt{\pm \om_{23}}} & 0 \\
0 & \frac{\om_{34}}{\om_{23}} & -\frac{\om_{24}}{\om_{23}} & 1
\end{array}\right)
\]
gives the normal form on $\fr'_{4,\gamma,\delta}$ with $\gamma\neq 0$
\[
\left\{\begin{array}{ccl}
         \vt & = & -2\gamma e^4\\
        \Omega & = & e^{14}\pm e^{23}
        \end{array} \right.
\]

The two forms above are different. Indeed, by requiring that the generic matrix $A=(a_{jk})$ is an automorphism (with respect to the basis $\{e^j\}$) swapping $e^{14}+e^{23}$ and $e^{14}-e^{23}$, one is reduced to solve an ideal whose Gr\"obner basis contains $a_{32}^2\delta + a_{33}^2\delta + \delta$, which is empty since $\delta>0$.

\subsection{$\mathfrak{d}_{4}$, $(14,-24,-12,0)$}

We take a generic $1$-form $\vartheta=\sum_{j=1}^{4}\vt_je^j$; imposing closedness, we obtain that $\vt_1=\vt_2=\vt_3=0$. Thus the Lee form is $\vt=\vt_4e^4$, with $\vt_4\neq 0$. For a generic 2-form $\Omega=\sum_{1\leq j < k \leq 4} \omega_{jk} e^{jk}$, the conformally closedness $d_\vt\Omega=0$ provides the following equations:
\begin{enumerate}[(1)]
\item $\vt_4\om_{12}+\om_{34}=0$;
\item $(\vt_4-1)\om_{23}=0$;
\item $(\vt_4+1)\om_{13}=0$.
\end{enumerate}

We assume $\vt_4\neq\pm 1$; then (2) and (3) imply $\om_{13}=0=\om_{23}$ and the non-degeneracy condition \eqref{eq:non_deg} becomes $\om_{12}\neq 0$. The generic lcs structure under these hypotheses is
\[
\vt=\vt_4e^4 \quad \textrm{and} \quad \Omega=\om_{12}(e^{12}-\vt_4e^{34})+\om_{14}e^{14}+\om_{24}e^{24}\;.
\]
In terms of the basis $\{e^1,e^2,e^3,e^4\}$ of $\mathfrak{d}_4^*$, we consider the automorphism
\begin{equation}\label{aut:d_4:1}
\left(\begin{array}{cccc}
1 & 0 & y & 0 \\
0 & 1 & x & 0 \\
0 & 0 & 1 & 0 \\
-x & -y & -xy+z & 1
\end{array}\right) \;,
\end{equation}
with $z=0$ and $x,y$ to be determined.
The Lee form is fixed, while the transformed $2$-form is
\[
\Omega'=\om_{12}(e^{12}-\vt_4e^{34})+(\om_{14}-y\om_{12}(1+\vt_4))e^{14}+(\om_{24}+x\om_{12}(1-\vt_4))e^{24}\;.
\]
Imposing the vanishing of the coefficients of $e^{14}$ and $e^{24}$ gives the equations
\[
\om_{14}-y\om_{12}(1+\vt_4)=0 \quad \textrm{and} \quad \om_{24}+x\om_{12}(1-\vt_4)=0\;.
\]
Since we assumed $\vt_4\neq\pm 1$, both equations have a solution
(namely, take $x=\frac{\omega_{24}}{\omega_{12}(\vartheta_4-1)}$ and $y=\frac{\omega_{14}}{\omega_{12}(\vartheta_4+1)}$) and we obtain $\Omega'=\om_{12}(e^{12}-\vt_4e^{34})$. The automorphism
\begin{equation}\label{aut:d_4:2}
\left(\begin{array}{cccc}
a & 0 & 0 & 0 \\
0 & b & 0 & 0 \\
0 & 0 & ab & 0 \\
0 & 0 & 0 & 1
\end{array}\right) \;,
\end{equation}
with $a=\frac{1}{\om_{12}}$ and $b=1$ gives the lcs structure $\vt=\vt_4 e^4$, $\Omega=e^{12}-\vt_4 e^{34}$ with $\vt_4\notin\{0,1,-1\}$. Then, the automorphism
\[
\left(\begin{array}{cccc}
0 & 1 & 0 & 0 \\
-1 & 0 & 0 & 0 \\
0 & 0 & 1 & 0 \\
0 & 0 & 0 & -1
\end{array}\right)
\]
gives the normal form
\[
\left\{ \begin{array}{ccl}
         \vt & = & \ve e^4\\
        \Omega & = & e^{12}-\ve e^{34}, \, \ve>0, \, \ve\neq 1
        \end{array} \right.
\]

Assume next $\vt_4=1$; then $\om_{13}=0$ by (2), the generic lcs structure is
\[
\vt=e^4 \quad \textrm{and} \quad \Omega=\om_{12}(e^{12}-e^{34})+\om_{14}e^{14}+\om_{23}e^{23}+\om_{24}e^{24}\;,
\]
and the non-degeneracy yields $\om_{12}^2+\om_{14}\om_{23}\neq 0$. We consider again the automorphism \eqref{aut:d_4:1} with $x=0$, which transforms $\Omega$ into
\begin{equation}\label{eq:448}
\Omega'=(\om_{12}-y\om_{23})(e^{12}-e^{34})+(\om_{14}-2y\om_{12}+y^2\om_{23})e^{14}+\om_{23}e^{23}+(\om_{24}+z\om_{23})e^{24}\;.
\end{equation}
If $\om_{23}=0$ then $\om_{12}\neq 0$ and 
\[
\Omega'=\om_{12}(e^{12}-e^{34})+(\om_{14}-2y\om_{12})e^{14}+\om_{24}e^{24}\,;
\]
choosing $y=\frac{\om_{14}}{2\om_{12}}$ and $z=0$ gives $\Om'=\om_{12}(e^{12}-e^{34})+\om_{24}e^{24}$. If $\om_{24}\neq0$ use \eqref{aut:d_4:2} with $a=\frac{\om_{24}}{\om_{12}}$ and $b=\frac{1}{\om_{24}}$; if $\om_{24}=0$ then use \eqref{aut:d_4:2} with $a=1$ and $b=\frac{1}{\om_{12}}$.
This gives the normal form
\[
\left\{ \begin{array}{ccl}
         \vt & = & e^4\\
        \Omega & = & e^{12}-e^{34}+\ve e^{24}, \ve\in\{0,1\}
        \end{array} \right.
\]
On the other hand, if $\om_{23}\neq 0$ we may set $y=\frac{\om_{12}}{\om_{23}}$ and $z=-\frac{\om_{24}}{\om_{23}}$ in \eqref{eq:448} and get
\[
\Omega'=\frac{\om_{14}\om_{23}-\om_{12}^2}{\om_{23}}e^{14}+\om_{23}e^{23}
\]
According to the sign of $\om_{14}\om_{23}-\om_{12}^2$, we choose the automorphism \eqref{aut:d_4:2} with $a=\pm\frac{\om_{23}}{\om_{14}\om_{23}-\om_{12}^2}$ and $b=\frac{\sqrt{\pm(\om_{14}\om_{23}-\om_{12}^2)}}{\om_{23}}$ 
to obtain the normal form
\[
\left\{ \begin{array}{ccl}
         \vt & = & e^4\\
        \Omega & = & \pm e^{14}+ e^{23}
        \end{array} \right.
\]

Finally we consider the case $\vt_4=-1$; then $\om_{23}=0$ by (3), the generic lcs structure is
\[
\vt=-e^4 \quad \textrm{and} \quad \Omega=\om_{12}(e^{12}+e^{34})+\om_{13}e^{13}+\om_{14}e^{14}+\om_{24}e^{24}\;,
\]
and the non-degeneracy yields $\om_{12}^2-\om_{14}\om_{23}\neq 0$. We consider the automorphism
\[
\left(\begin{array}{cccc}
0 & 1 & 0 & 0 \\
1 & 0 & 0 & 0 \\
0 & 0 & -1 & 0 \\
0 & 0 & 0 & -1
\end{array}\right) \;
\]
which sends $\vt=-e^4$ to $\vt'=e^4$ and $\Omega$ to
\[
\Omega'=\om_{12}'(e^{12}-e^{34})+\om_{14}'e^{14}+\om_{23}'e^{23}+\om_{24}'e^{24},
\]
with $\om_{12}'=-\om_{12}$, $\om_{23}'=-\om_{13}$, $\om_{14}'=-\om_{24}$ and $\om_{24}'=-\om_{14}$. The non-degeneracy condition reads $(\om_{12}')^2+\om_{14}'\om_{23}'$ and we are back to the previous case.

We claim that the forms $\Omega_1=e^{12}-e^{34}$, $\Omega_2=e^{12}-e^{34}+e^{24}$, $\Omega_3=e^{14}+e^{23}$, and $\Omega_4=-e^{14}+e^{23}$ are distinct. Indeed, arguing as before, we get an ideal with Gr\"obner basis containing either $1$ or $a_{33}^2+1$.

Finally, we show that there is no automorphisms of the Lie algebra interchanging the Lee forms $\vt_1\coloneq\ve_1 e^4$, $\vt_2\coloneq\ve_2 e^4$, and $\vt_3\coloneq e^4$, where $\ve_1\neq\ve_2$ and $0<\ve_1\neq1$, $0<\ve_2\neq1$. This would be equivalent to solve the ideal with Gr\"obner basis $B_{jk}$, in case identifying $\vt_j$ with $\vt_k$, with the further condition $\det A\neq0$. We note:
\begin{itemize}
\item $B_{12}$ contains $a_{31}$, $a_{32}$, $a_{34}$, $a_{33}(a_{44}\varepsilon_2 - \varepsilon_1)$, $a_{44}\varepsilon_1 - \varepsilon_2$, which yields $a_{31}=a_{32}=a_{33}=a_{34}=0$.
\item $B_{13}$ contains $a_{31}$, $a_{32}$, $a_{34}$, $a_{33}(a_{44} - \varepsilon_1)$, $a_{44}\varepsilon_1 - 1$, which yields $a_{31}=a_{32}=a_{33}=a_{34}=0$.
\end{itemize}


\subsection{$\mathfrak{d}_{4,\lambda}$, $(\lambda 14,(1-\lambda)24,-12+34,0)$, $\lambda\geq\frac{1}{2}$}

Take a generic $1$-form $\vartheta=\sum_{j=1}^{4}\vt_je^j$ and a generic $2$-form $\Omega=\sum_{1\leq j < k \leq 4} \omega_{jk}e^{jk}$. 

Assume first $\lambda\neq1$. We compute $d\vartheta=\lambda\vt_1e^{14}+(1-\lambda)\vt_2e^{24}-\vt_3e^{12}+\vt_3e^{34}$, hence $d\vt=0$ if and only if $\vt_1=\vt_2=\vt_3=0$; the generic Lee form is 
$\vt=\vt_4e^4$ with $\vt_4\neq 0$.
We compute the 2-cocycles of the Lichnerowicz differential $d_\vt$:
\begin{itemize}
\item $d_\vt(e^{12})=(-1-\vt_4)e^{124}$
\item $d_\vt(e^{13})=(-1-\lambda-\vt_4)e^{134}$
\item $d_\vt(e^{14})=0$
\item $d_\vt(e^{23})=(\lambda-2-\vt_4)e^{234}$
\item $d_\vt(e^{24})=0$
\item $d_\vt(e^{34})=-e^{124}$
\end{itemize}
For $\om_{13} e^{13}+\om_{24} e^{24}$ to be a $d_\vt$-cocycle one needs $\vt_4=-(1+\lambda)$. For $\om_{14} e^{14}+\om_{23} e^{23}$ to be a cocycle one needs $\vt_4=\lambda-2$. This happens simultaneously if and only if $\lambda=\frac{1}{2}$, giving $\vt_4=-\frac{3}{2}$. 

We begin with the case $\lambda\neq\frac{1}{2}$ and $\vt_4=-(1+\lambda)$. The generic lcs structure is
\[
\vt=-(1+\lambda)e^4 \quad \textrm{and} \quad \Omega=\om_{12}(e^{12}+\lambda e^{34})+\om_{13}e^{13}+\om_{14}e^{14}+\om_{24}e^{24}\,.
\]
with \eqref{eq:non_deg} reducing to $\lambda\om_{12}^2-\om_{13}\om_{24}\neq 0$.
Assume first $\om_{13}\neq 0$. We consider the automorphism
\begin{equation}\label{d_4lambda:aut0}
\left(\begin{array}{cccc}
1 & 0 & -\frac{y}{\lambda-1} & 0 \\
0 & 1 & -\frac{x}{\lambda} & 0 \\
0 & 0 & 1 & 0 \\
x & y & \frac{xy(1-2\lambda)}{2\lambda(\lambda-1)}+z & 1
\end{array}\right)
\end{equation}
with $x=\lambda\frac{\om_{12}}{\om_{13}}$, $y=0$ and $z=-\frac{\om_{14}}{\om_{13}}$. This leaves $\vt$ invariant, while $\Om'=\om_{13}e^{13}+\frac{\om_{13}\om_{24}-\lambda\om_{12}^2}{\om_{13}}e^{24}$. According to the sign of $\om_{13}\om_{24}-\lambda\om_{12}^2$, the automorphism 
\begin{equation}\label{d_4lambda:aut1}
\left(\begin{array}{cccc}
a & 0 & 0 & 0 \\
0 & b & 0 & 0 \\
0 & 0 & ab & 0 \\
0 & 0 & 0 & 1
\end{array}\right)
\end{equation}
with $a=\frac{\sqrt{\pm(\om_{13}\om_{24}-\lambda\om_{12}^2)}}{\om_{13}}$ and $b=\frac{\om_{13}}{\om_{13}\om_{24}-\lambda\om_{12}^2}$
gives the normal form ($\lambda\notin\{\frac{1}{2},1\}$)
\[
\left\{ \begin{array}{ccl}
         \vt & = & -(\lambda+1)e^4\\
        \Omega & = & \pm e^{13}+e^{24}
        \end{array} \right.
\]
If $\om_{13}=0$ then $\om_{12}\neq 0$ and we consider the automorphism \eqref{d_4lambda:aut0} with $x=\frac{\om_{24}}{2\om_{12}}$, $y=\frac{(\lambda-1)\om_{14}}{\om_{12}}$ and $z=0$. This leaves $\vt$ invariant, while $\Om'=\om_{12}(e^{12}+\lambda e^{34})$. The automorphism \eqref{d_4lambda:aut1} with $a=1$ and $b=\frac{1}{\om_{12}}$ provides the normal form ($\lambda\notin\{\frac{1}{2},1\}$)
\[
\left\{ \begin{array}{ccl}
         \vt & = & -(\lambda+1)e^4\\
        \Omega & = & e^{12}+\lambda e^{34}
        \end{array} \right.
\]

We continue with the case $\lambda\neq\frac{1}{2}$, $\vt_4=\lambda-2$. The generic lcs structure is
\[
\vt=(\lambda-2)e^4, \quad \Om=\om_{12}(e^{12}-(\lambda-1)e^{34})+\om_{14}e^{14}+\om_{23}e^{23}+\om_{24}e^{24}
\]
with non-degeneracy condition amounting to $(\lambda-1)\om_{12}^2-\om_{14}\om_{23}\neq 0$. Assuming $\om_{23}\neq 0$ we consider the automorphism \eqref{d_4lambda:aut0} with $x=0$, $y=(1-\lambda)\frac{\om_{12}}{\om_{23}}$ and $z=-\frac{\om_{24}}{\om_{23}}$.
This leaves $\vt$ invariant, while $\Om'=\frac{\om_{14}\om_{23}-(\lambda-1)\om_{12}^2}{\om_{23}}e^{14}+\om_{23}e^{23}$. According to the sign of $\om_{14}\om_{23}-(\lambda-1)\om_{12}^2$, the automorphism \eqref{d_4lambda:aut1} with $a=\frac{\om_{23}}{\om_{14}\om_{23}-(\lambda-1)\om_{12}^2}$ and $b=\frac{\sqrt{\pm(\om_{14}\om_{23}-(\lambda-1)\om_{12}^2)}}{\om_{23}}$
gives the normal form ($\lambda\notin\{\frac{1}{2},1,2\}$)
\[
\left\{ \begin{array}{ccl}
         \vt & = & (\lambda-2)e^4\\
        \Omega & = & e^{14}\pm e^{23}
        \end{array} \right.
\]
We consider next the case $\om_{23}=0$; then $\om_{12}\neq 0$ and we take the automorphism \eqref{d_4lambda:aut0} with $x=\lambda\frac{\om_{24}}{\om_{12}}$, $y=-\frac{\om_{14}}{2\om_{12}}$ and $z=0$. This gives $\Omega'=\om_{12}(e^{12}-(\lambda-1)e^{34})$, while leaving $\vt$ invariant. We choose again $a=1$ and $b=\frac{1}{\om_{12}}$ in \eqref{d_4lambda:aut1} to obtain the normal form ($\lambda\notin\{\frac{1}{2},1,2\}$)
\[
\left\{ \begin{array}{ccl}
         \vt & = & (\lambda-2)e^4\\
        \Omega & = & e^{12}-(\lambda-1)e^{34}
        \end{array} \right.
\]

The last case is $\lambda\neq\frac{1}{2}$, $\vt_4\not\in\{-(\lambda+1),\lambda-2\}$. Here the generic lcs structure is
\[
\vt = \vt_4 e^4, \quad \Omega=\omega_{12}(e^{12}-(\vt_4 + 1)e^{34})+ \omega_{14}e^{14} + \omega_{24} e^{24}\,.
\]
The non-degeneracy condition \eqref{eq:non_deg} reads $\om_{12}^2(\vt_4+1)\neq0$, forcing $\vt_4\neq -1$. In case $\vt_4\not\in\{-\lambda,\lambda-1\}$, apply the automorphism \eqref{d_4lambda:aut0}
with $x=\frac{\lambda\om_{24}}{\om_{12}(\vt_4+1-\lambda)}$, $y=-\frac{\om_{14}(\lambda-1)}{\om_{12}(\lambda+\vt_4)}$ and $z=0$
to get $\vt'=\vt$ and $\Om'=\omega_{12}(e^{12} -(\vartheta_4 + 1)e^{34})$.
Use now \eqref{d_4lambda:aut1} with $a=1$ and $b=\frac{1}{\om_{12}}$ to get the normal form
($\lambda\notin\{\frac{1}{2},1\}$)
\[
\left\{ \begin{array}{ccl}
         \vt & = & \ve e^4\\
        \Omega & = & e^{12} - (\ve + 1)e^{34}, \, \ve\notin\{-1,-\lambda-1,\lambda-2,\lambda-1,-\lambda\}
        \end{array} \right.
\]
In case $\vt_4=-\lambda$, the automorphism \eqref{d_4lambda:aut0} with $x=\frac{\lambda\om_{24}}{\om_{12}(2\lambda-1)}$ and $y=z=0$ fixes $\vt$, while $\Omega'=\omega_{12}(e^{12}-(1-\lambda)e^{34})+\omega_{14}e^{14}$. If $\om_{14}\neq 0$, apply the automorphism \eqref{d_4lambda:aut1} with $a=\frac{1}{\om_{14}}$ and $b=\frac{\om_{14}}{\om_{12}}$; if $\om_{14}=0$, apply \eqref{d_4lambda:aut1} with $a=\frac{1}{\om_{12}}$ and $b=1$. This gives the normal form ($\lambda\not\in\{\frac{1}{2},1\}$)
\[
\left\{ \begin{array}{ccl}
         \vt & = & -\lambda e^4 \\
        \Omega & = & e^{12}-(1-\lambda)e^{34}+\ve e^{14}, \, \ve\in\{0,1\}
        \end{array} \right.
\]
Finally, when $\vt_4=\lambda-1$, use the automorphism \eqref{d_4lambda:aut0}
with $y=-\frac{\om_{14}(\lambda-1)}{\om_{12}(2\lambda-1)}$ and $x=z=0$ to get $\vt'=\vt$ and $\Omega'=\omega_{12}(e^{12}-\lambda e^{34})+\omega_{24}e^{24}$. If $\om_{24}\neq 0$, apply the automorphism \eqref{d_4lambda:aut1} with $a=\frac{\om_{24}}{\om_{12}}$ and $b=\frac{1}{\om_{24}}$; if $\om_{24}=0$, apply \eqref{d_4lambda:aut1} with $a=1$ and $b=\frac{1}{\om_{12}}$. This gives the normal form ($\lambda\not\in\{\frac{1}{2},1\}$)
\[
\left\{ \begin{array}{ccl}
         \vt & = & (\lambda-1) e^4 \\
        \Omega & = & e^{12}-\lambda e^{34}+\ve e^{24}, \, \ve\in\{0,1\}
        \end{array} \right.
\]

We turn now to the issue of uniqueness, modulo automorphisms of the Lie algebra. We first notice that, in case $\lambda\not\in\{\frac{1}{2},1\}$, the lcs structures on $\mathfrak{d}_{4,\lambda}$ have only one Lee form: $\vt_1\coloneq -(\lambda+1)e^4$, $\vt_2\coloneq(\lambda-2)e^4$, $\vt_3\coloneq-\lambda e^4$ and $\vt_4\coloneq(\lambda-1)e^4$. We look now at the different lcs structures with same Lee form. For $\vt_1$, we have $\Omega_1\coloneq-e^{12}+\lambda e^{34}$, $\Omega_2\coloneq e^{13}+e^{24}$ and $\Omega_{3}\coloneq-e^{13}+e^{24}$.
For $\vt_2$, we have $\Omega_4\coloneq e^{12}-(\lambda-1)e^{34}$, $\Omega_5\coloneq e^{14}+e^{23}$ and $\Omega_6\coloneq e^{14}-e^{23}$. For $\vt_3$, $\Omega_7\coloneq e^{12}-(1-\lambda)e^{34}$ and and $\Omega_8\coloneq e^{12}-(1-\lambda)e^{34}+e^{14}$. For $\vt_4$, $\Omega_9\coloneq e^{12}-\lambda e^{34}$ and $\Omega_{0}\coloneq e^{12}-\lambda e^{34}+e^{24}$. In each case, we consider 
the ideal for $A=(a_{jk})$ in the basis $\{e^j\}$ to be a morphism of the Lie algebra sending $\Omega_j$ into $\Omega_k$, and we compute a Gr\"obner basis for it, $B_{jk}$. Then we get that $B_{12}$, $B_{13}$, $B_{45}$, $B_{46}$, $B_{78}$, $B_{90}$ are equal to $(1)$, and $B_{23}$ and $B_{56}$ contain $a_{33}^2+1$.

We tackle now the case $\lambda=\frac{1}{2}$, $\vt_4=-\frac{3}{2}$. In this case the automorphism \eqref{d_4lambda:aut0} becomes
\begin{equation}\label{d_4lambda:aut2}
\left(\begin{array}{cccc}
1 & 0 & 2y & 0 \\
0 & 1 & -2x & 0 \\
0 & 0 & 1 & 0 \\
x & y & z & 1
\end{array}\right) \,.
\end{equation}
The generic lcs structure is
\[
\vt=-\frac{3}{2}e^4, \quad \Om=\om_{12}(e^{12}+\frac{1}{2}e^{34})+\om_{13}e^{13}+\om_{14}e^{14}+\om_{23}e^{23}+\om_{24}e^{24},
\]
with \eqref{eq:non_deg} giving $\om_{12}^2-2\om_{13}\om_{24}+2\om_{14}\om_{23}\neq 0$. Assume $\om_{13}=\om_{23}=0$; then $\om_{12}\neq 0$ and the automorphism \eqref{d_4lambda:aut2} with $x=\frac{\om_{24}}{2\om_{12}}$, $y=-\frac{\om_{14}}{2\om_{12}}$ and $z=0$, followed by \eqref{d_4lambda:aut1} with $a=1$ and $b=\frac{1}{\om_{12}}$, gives the normal form ($\lambda=\frac{1}{2}$)
\[
\left\{
\begin{array}{ccl}
\vt & = & -\frac{3}{2} e^4\\
\Omega & = & e^{12}+\frac{1}{2}e^{34}
\end{array} \right.
\]
If we assume that either $\om_{13}$ or $\om_{23}$ are non-zero, using the automorphism (which exists only for $\lambda=\frac{1}{2}$)
\begin{equation}\label{d_4lambda:aut3}
\left(\begin{array}{cccc}
0 & 1 & 0 & 0 \\
-1 & 0 & 0 & 0 \\
0 & 0 & 1 & 0 \\
0 & 0 & 0 & 1
\end{array}\right)
\end{equation}
we can assume $\om_{13}\neq 0$. The automorphism \eqref{d_4lambda:aut2} with $x=\frac{\om_{12}}{2\om_{13}}$, $y=0$ and $z=-\frac{\om_{14}}{\om_{13}}$ gives $\vt'=\vt$ and
\[
\Om'=\om_{13}e^{13}+\frac{2\om_{13}\om_{24}-\om_{12}^2-2\om_{14}\om_{23}}{2\om_{13}}e^{24}+\om_{23}e^{23}.
\]
The automorphism
\begin{equation}\label{d_4lambda:aut4}
\left(\begin{array}{cccc}
1 & 0 & 0 & 0 \\
w & 1 & 0 & 0 \\
0 & 0 & 1 & 0 \\
0 & 0 & 0 & 1
\end{array}\right)
\end{equation}
with $w=-\frac{\om_{23}}{\om_{13}}$ gives $\vt''=\vt'$ and $\Om''=\om_{13}e^{13}+\frac{2\om_{13}\om_{24}-\om_{12}^2-2\om_{14}\om_{23}}{2\om_{13}}e^{24}$ and, finally, the automorphism \eqref{d_4lambda:aut1} with $a=\frac{\sqrt{\pm(2\om_{13}\om_{24}-\om_{12}^2-2\om_{14}\om_{23})}}{\sqrt{2}\om_{13}}$ and $b=\frac{2\om_{13}}{2\om_{13}\om_{24}-\om_{12}^2-2\om_{14}\om_{23}}$ gives the normal form ($\lambda=\frac{1}{2}$)
\[
\left\{
\begin{array}{ccl}
\vt & = & -\frac{3}{2} e^4\\
\Omega & = & \pm e^{13}+e^{24}
\end{array} \right.
\]

If $\lambda=\frac{1}{2}$ and $\vt_4\neq-\frac{3}{2}$ the generic lcs structure is
\[
\vt=\vt_4e^4, \quad \Om=\om_{12}(e^{12}-(\vt_4+1)e^{34})+\om_{14}e^{14}+\om_{24}e^{24}.
\]
The non-degeneracy yields $(\vt_4+1)\om_{12}^2\neq 0$, hence $\om_{12}\neq 0$ and $\vt_4\neq -1$. We consider an automorphism of the form \eqref{d_4lambda:aut2}
where $x$, $y$ and $z$ are parameters to be determined. While the Lee form is fixed, the 2-form transforms into
\[
\Om'=\om_{12}(e^{12}-(\vt_4+1)e^{34})+(\om_{14}-y\om_{12}(2\vt_4+1))e^{14}+(\om_{24}+x\om_{12}(2\vt_4+1))e^{24}.
\]
If $\vt_4\neq -\frac{1}{2}$, we choose $x=-\frac{\om_{24}}{\om_{12}(2\vt_4+1)}$, $y=\frac{\om_{14}}{\om_{12}(2\vt_4+1)}$ and $z=0$; then $\Om'=\om_{12}(e^{12}-(\vt_4+1)e^{34})$. The automorphism \eqref{d_4lambda:aut1} with $a=1$ and $b=\frac{1}{\om_{12}}$ gives the normal form ($\lambda=\frac{1}{2}$)
\[
\left\{ \begin{array}{ccl}
         \vt & = & \ve e^4\\
        \Omega & = & e^{12}-(\ve+1)e^{34},\, \ve\notin\{-\frac{3}{2},-1,-\frac{1}{2},0\}
        \end{array} \right.
\]
If $\vt_4=-\frac{1}{2}$ the above automorphism will not work. If $\omega_{14}=\omega_{24}=0$, then apply \eqref{d_4lambda:aut1} with $a=\frac{1}{\om_{12}}$ and $b=1$ to get $\vt'=-\frac{1}{2}e^4$ and $\Omega'=e^{12} - \frac{1}{2}e^{34}$.
Assuming either $\om_{14}$ or $\om_{24}$ are non-zero, using the automorphism \eqref{d_4lambda:aut3} we can suppose that this is the case for $\om_{14}$. We consider then the automorphism \eqref{d_4lambda:aut4} with $w=-\frac{\om_{24}}{\om_{14}}$, giving $\Om'=\om_{12}(e^{12}-\frac{1}{2}e^{34})+\om_{14}e^{14}$. Apply then \eqref{d_4lambda:aut1} with $a=\frac{1}{\om_{14}}$ and $b=\frac{\om_{14}}{\om_{12}}$. At last, we get the normal forms ($\lambda=\frac{1}{2}$)
\[
\left\{
\begin{array}{ccl}
\vt & = & -\frac{1}{2} e^4\\
\Omega & = & e^{12}-\frac{1}{2}e^{34}+\ve e^{14},\, \ve\in\{0,1\}
\end{array} \right.
\]

We consider now the uniqueness of the above normal forms, in case $\lambda=\frac{1}{2}$. First of all, as for the Lee forms, we have to prove that there is no automorphism $A=(a_{jk})$ (with respect to $\{e^j\}$) transforming $\vt_1\coloneq\varepsilon_1 e^4$ into $\vt_2\coloneq\ve_2e^4$, where $\ve_1,\ve_2\neq-1$. If it existed, then its entries should satisfy the ideal with Gr\"obner basis containing, in particular, $a_{31}$, $a_{32}$, $a_{34}$, $a_{33}(\varepsilon_1 - \varepsilon_2)$, hence either $\ve_1=\ve_2$ or $A$ is singular. Now, we focus on the lcs structures with same Lee forms. In case $\vt=-\frac{3}{2}e^4$, we have to distinguish $\Omega_1\coloneq e^{12}+\frac{1}{2}e^{34}$, $\Omega_2\coloneq e^{13}+e^{24}$, and $\Omega_3\coloneq-e^{13}+e^{24}$. As before, consider an associated Gr\"obner basis $B_{jk}$ for the pair $(\Omega_j,\Omega_k)$. We get that $B_{12}$ and $B_{13}$ contain $1$, and $B_{23}$ contains $a_{33}^2 + 1$. In case $\vt=-\frac{1}{2}e^4$, we have to distinguish $\Omega_1\coloneq e^{12}-\frac{1}{2}e^{34}$ and $\Omega_2\coloneq e^{12}-\frac{1}{2}e^{34}+e^{14}$. A computation for the associated ideal gives the Gr\"obner basis $(1)$.

Finally, we consider the case $\lambda=1$. The generic Lee form is now $\vt=\vt_2e^2+\vt_4e^4$, with $\vt_2^2+\vt_4^2\neq 0$, and the condition $d_\vt\Om=0$ for a 2-form $\Om$ yields the equations
\begin{enumerate}
\item $\vt_2\om_{13}=0$;
\item $(\vt_4+1)\om_{12}+\om_{34}-\vt_2\om_{14}=0$;
\item $(\vt_4+2)\om_{13}=0$;
\item $(\vt_4+1)\om_{23}+\vt_2\om_{34}=0$.
\end{enumerate}

Suppose first $\vt_2=0$; then $\vt_4\neq 0$; if $\vt_4\notin\{-1,-2\}$ then the above equations imply $\om_{13}=\om_{23}=0$ and $\om_{34}=-(\vt_4+1)\om_{12}$. The generic lcs structure is then
\[
\vt=\vt_4e^4 \quad \textrm{and} \quad \Om=\om_{12}(e^{12}-(\vt_4+1)e^{34})+\om_{14}e^{14}+\om_{24}e^{24}
\]
and \eqref{eq:non_deg} becomes $\om_{12}\neq 0$. The automorphism
\begin{equation}\label{d_4lambda:aut5}
\left(\begin{array}{cccc}
1 & 0 & t & 0 \\
0 & 1 & -s & 0 \\
0 & 0 & 1 & 0 \\
s & 0 & u+\frac{1}{2}st & 1
\end{array}\right)
\end{equation}
with $s=-\frac{\om_{24}}{\vt_4\om_{12}}$, $t=\frac{\om_{14}}{(\vt_4+1)\om_{12}}$ and $u=0$
fixes $\vt$ and gives $\Om'=\om_{12}(e^{12}-(\vt_4+1)e^{34})$. Then \eqref{d_4lambda:aut1} with $a=1$ and $b=\frac{1}{\om_{12}}$ gives the normal form ($\lambda=1$)
\[
\left\{
\begin{array}{ccl}
\vt & = & \ve e^4\\
\Omega & = & e^{12}-(\ve+1)e^{34}, \, \ve\notin\{-2,-1,0\}
\end{array} \right.
\]
If $\vt_4=-1$ then $\om_{13}=\om_{34}=0$, the generic lcs structure is
\[
\vt=-e^4 \quad \textrm{and} \quad \Om=\om_{12}e^{12}+\om_{14}e^{14}+\om_{23}e^{23}+\om_{24}e^{24}
\]
and the non-degeneracy yields $\om_{14}\om_{23}\neq 0$.
The automorphism \eqref{d_4lambda:aut5} with $s=0$, $t=\frac{\om_{12}}{\om_{23}}$ and $u=-\frac{\om_{24}}{\om_{23}}$ fixes $\vt$ and gives $\Om'=\om_{14}e^{14}+\om_{23}e^{23}$. The automorphism \eqref{d_4lambda:aut1} with $a=\frac{1}{\om_{14}}$ and $b=\sqrt{\pm\frac{\om_{14}}{\om_{23}}}$
gives the normal form ($\lambda=1$)
\[
\left\{
\begin{array}{ccl}
\vt & = & -e^4\\
\Omega & = & e^{14}\pm e^{23}
\end{array} \right.
\]
If $\vt_4=-2$ then $\om_{23}=0$ and $\om_{34}=\om_{12}$; the lcs structure is
\[
\vt=-2e^4 \quad \textrm{and} \quad \Om=\om_{12}(e^{12}+e^{34})+\om_{13}e^{13}+\om_{14}e^{14}+\om_{24}e^{24}
\]
and \eqref{eq:non_deg} becomes $\om_{12}^2-\om_{13}\om_{24}\neq 0$. Assuming $\om_{13}\neq 0$, we consider the automorphism \eqref{d_4lambda:aut5} with $s=\frac{\om_{12}}{\om_{13}}$, $u=-\frac{\om_{14}}{\om_{13}}$ and $t=0$ and obtain $\vt'=\vt$ and $\Om'=\om_{13}e^{13}+\frac{\om_{13}\om_{24}-\om_{12}^2}{\om_{13}}e^{24}$. The automorphism \eqref{d_4lambda:aut1} with $a=\frac{\sqrt{\pm(\om_{13}\om_{24}-\om_{12}^2})}{\om_{13}}$ and $b=\frac{\om_{13}}{\om_{13}\om_{24}-\om_{12}^2}$ gives the normal form ($\lambda=1$)
\[
\left\{
\begin{array}{ccl}
\vt & = & -2e^4\\
\Omega & = & \pm e^{13}+e^{24}
\end{array} \right.
\]
If $\om_{13}=0$ then $\om_{12}\neq 0$ and \eqref{d_4lambda:aut5} with $s=\frac{\om_{24}}{2\om_{12}}$, $t=-\frac{\om_{14}}{\om_{12}}$ and $u=0$ gives $\vt'=\vt$ and $\Om'=\om_{12}(e^{12}+e^{34})$. Then \eqref{d_4lambda:aut1} with $a=1$ and $b=\frac{1}{\om_{12}}$ provides the normal form ($\lambda=1$)
\[
\left\{
\begin{array}{ccl}
\vt & = & -2e^4\\
\Omega & = & e^{12}+e^{34}
\end{array} \right.
\]

If $\vt_2\neq 0$ then $\om_{13}=0$, $\om_{14}=\frac{(\vt_4+1)(\vt_2\om_{12}-\om_{23})}{\vt_2^2}$ and $\om_{34}=-\frac{(\vt_4+1)\om_{23}}{\vt_2}$. The generic lcs structure is
$\vt=\vt_2e^2+\vt_4e^4$ and
\[
\Om=\om_{12}e^{12}+\frac{(\vt_4+1)(\vt_2\om_{12}-\om_{23})}{\vt_2^2}e^{14}+\om_{23}e^{23}+\om_{24}e^{24}-\frac{(\vt_4+1)\om_{23}}{\vt_2}e^{34}.
\]
The non-degeneracy forces $(\vt_4+1)\om_{23}^2\neq0$, which implies $\vt_4\neq -1$ and $\om_{23}\neq 0$. 
We consider \eqref{d_4lambda:aut5} with $t=\frac{\om_{12}}{\om_{23}}$, $u=-\frac{\om_{24}}{\om_{23}}$ and $s=0$ to obtain $\vt'=\vt$ and
\[
\Om'=-\frac{(\vt_4+1)\om_{23}}{\vt_2^2}e^{14}+\om_{23}e^{23}-\frac{(\vt_4+1)\om_{23}}{\vt_2}e^{34}.
\]
The automorphism \eqref{d_4lambda:aut1} with $a=-\frac{\vt_2^2}{(\vt_4+1)\om_{23}}$ and $b=\frac{1}{\vt_2}$ gives the normal form ($\lambda=1$)
\[
\left\{
\begin{array}{ccl}
\vt & = & e^2+\ve e^4\\
\Omega & = & e^{14}-\frac{1}{\ve+1}e^{23}+e^{34},\, \ve\neq -1
\end{array} \right.
\]

We turn now to the uniqueness of the models. First, we show that the Lee forms are not equivalent. 
In the case $\lambda=1$, we have first of all to show that the Lee forms $\vt_1\coloneq \ve_1 e^4$, $\vt_2\coloneq \ve_2e^4$, $\vt_3\coloneq e^2+\ve_3e^4$, and $\vt_4\coloneq e^2+\ve_4 e^4$, where $\ve_1,\ve_2\neq0$ and $\ve_3,\ve_4\neq-1$, are not equivalent under automorphisms of the Lie algebra. We set the ideal for $A=(a_{jk})$ in the basis $\{e^j\}$ to represent a morphism of the Lie algebra sending $\vt_j$ into $\vt_k$, and we compute a Gr\"obner basis $B_{jk}$ for it:
\begin{itemize}
\item $B_{12}$ contains $a_{31}\varepsilon_2$, $a_{32}$, $a_{34}$, and $a_{33}(\varepsilon_1 - \varepsilon_2)$;
\item $B_{34}$ contains $a_{33}^2(\varepsilon_3 - \varepsilon_4)$, $a_{32}$, $a_{34}$, $a_{31}^2$;
\item $B_{13}$ contains $a_{11}$, $a_{12}$, $a_{13}$, $a_{14}$.
\end{itemize}
We now show that the lcs structures with same Lee forms are non-equivalent, too. In the case of $\vt=-2e^4$, we have to distinguish $\Omega_1\coloneq e^{12}+e^{34}$, $\Omega_2\coloneq e^{13}+e^{24}$ and $\Omega_3\coloneq -e^{13}+e^{24}$. As before, we compute a Gr\"obner basis $B_{jk}$ for the ideal of morphism $A=(a_{jk})$ of Lie algebra, in the basis $\{e^j\}$, moving $\Omega_j$ into $\Omega_k$: $B_{12}$ and $B_{13}$ contain $1$, while $B_{23}$ contains $a_{33}^2 + 1$.
Finally, in the case of $\vt=-e^4$, we have to distinguish $\Omega_1\coloneq e^{14}+e^{23}$ from $\Omega_2\coloneq e^{14}-e^{23}$. A Gr\"obner basis for the ideal of morphism $A=(a_{jk})$ of Lie algebra, in the basis $\{e^j\}$, moving $\Omega_1$ into $\Omega_2$, contains $a_{33}^2 + 1$.

\subsection{$\mathfrak{d}'_{4,\delta}$, $(\frac{\delta}{2}14+24,-14+\frac{\delta}{2}24,-12+\delta 34,0)$, $\delta\geq 0$}

We take a generic $1$-form $\vartheta=\sum_{j=1}^{4}\vt_je^j$; a computation shows that $d\vt=0$ if and only if $\vt_1=\vt_2=\vt_3=0$. Thus the generic Lee form is $\vt=\vt_4e^4$ with $\vt_4\neq 0$. We consider a $2$-form $\Omega=\sum_{1\leq j < k \leq 4} \omega_{jk}e^{jk}$ and impose $d_\vt\Omega=0$. We obtain the following equations:
\begin{enumerate}
\item $\om_{12}(\delta+\vt_4)+\om_{34}=0$
\item $\omega_{23}-\om_{13}\left(\frac{3\delta}{2}+\vt_4\right)=0$
\item $\omega_{13}+\om_{23}\left(\frac{3\delta}{2}+\vt_4\right)=0$
\end{enumerate}
Equations (2) and (3) imply $\om_{13}=0=\om_{23}$, while equation (1) gives $\om_{34}=-(\delta+\vt_4)\om_{12}$; in particular, \eqref{eq:non_deg} reduces to $\om_{12}^2(\delta+\vt_4)\neq 0$, saying that $\vt_4\neq -\delta$. It follows that the generic lcs structure is given by
\[
\vt=\vt_4e^4 \quad \textrm{and} \quad \Omega=\om_{12}(e^{12}-(\delta+\vt_4)e^{34})+\om_{14}e^{14}+\omega_{24}e^{24}\;.
\]

We consider, in terms of the basis $\{e^1,e^2,e^3,e^4\}$ of $(\mathfrak{d}'_{4,\delta})^*$, the automorphism given by the matrix
\[
\left(\begin{array}{cccc}
1 & 0 & -\frac{2(\delta\ell+2c)}{\delta^2+4} & 0 \\
0 & 1 & \frac{2(\delta c-2\ell)}{\delta^2+4} & 0 \\
0 & 0 & 1 & 0 \\
-c & -\ell & \frac{2\ell^2+2c^2}{\delta^2+4} & 1
\end{array}\right) \;,
\]
where $c,\ell\in\mathbb{R}$ are parameters to be determined. The Lee form remains unaltered under this change of basis. Imposing that the coefficients of the basis vectors $e^{14}$ and $e^{24}$ in the transformed expression for $\Omega$ vanish gives two equations:
\begin{equation}\label{eq:lin_sys}
\left\{ \begin{array}{ccl}
         c(-4\omega_{12}(\delta+\vt_4))+\ell(\omega_{12}(\delta^2+4)-2\delta\omega_{12}(\delta+\vt_4)) & = & \omega_{14}(\delta^2+4)\\
        c(-\omega_{12}(\delta^2+4)+2\delta\omega_{12}(\delta+\vt_4))+\ell(-4\omega_{12}(\delta+\vt_4)) & = & \omega_{24}(\delta^2+4)
        \end{array} \right. 
\end{equation}
The matrix of the linear system \eqref{eq:lin_sys} is
\[
\left(\begin{array}{cc}
-4\omega_{12}(\delta+\vt_4) & \omega_{12}(\delta^2+4)-2\delta\omega_{12}(\delta+\vt_4) \\
-\omega_{12}(\delta^2+4)+2\delta\omega_{12}(\delta+\vt_4) & -4\omega_{12}(\delta+\vt_4)
\end{array}\right) \;,
\]
whose determinant $\omega_{12}^2(16(\delta+\vt_4)^2+((\delta^2+4)^2-2\delta(\delta+\vt_4))^2)$ is always positive. Hence \eqref{eq:lin_sys} has a unique solution and the transformed lcs structure is
\[
\vt'=\vt_4e^4 \quad \textrm{and} \quad \Omega'=\om_{12}(e^{12}-(\delta+\vt_4)e^{34})\;.
\]
According to the sign of $\om_{12}$, we consider the automorphism: 
\[
\left(\begin{array}{cccc}
\frac{1}{\sqrt{\pm\om_{12}}} & 0 & 0 & 0 \\
0 & \frac{1}{\sqrt{\pm\om_{12}}} & 0 & 0 \\
0 & 0 & \pm\frac{1}{\om_{12}} & 0 \\
0 & 0 & 0 & 1
\end{array}\right).
\]
Doing so, we see that every lcs structure on $\mathfrak{d}'_{4,\delta}$ is equivalent to
\[
\left\{ \begin{array}{ccl}
         \vt & = & \ve e^4\\
        \Omega_{\pm} & = & \pm(e^{12}-(\delta+\ve)e^{34}), \, \ve\not\in \{0,-\delta\}
        \end{array} \right.
\]

In the case $\delta=0$, we can further apply the automorphism
$$
\left(\begin{array}{cccc}
1 & 0 & 0 & 0 \\
0 & -1 & 0 & 0 \\
0 & 0 & -1 & 0 \\
0 & 0 & 0 & -1
\end{array}\right)
$$
and so we see that every lcs structure on $\mathfrak{d}'_{4,0}$ is equivalent to
\[
\left\{ \begin{array}{ccl}
         \vt & = & \ve e^4\\
        \Omega & = & e^{12}-\ve e^{34}, \, \ve>0
        \end{array} \right.
\]

The above structure can not be further reduced. Indeed, consider the generic linear morphism with matrix $A=(a_{jk})$ in the basis $\{e^j\}$. The Gr\"obner basis of the ideal associated to the condition of being a morphism of the Lie algebra and to the condition that it transforms $\vt_1\coloneq \ve_1e^4$ to $\vt_2\coloneq \ve_2e^4$, where $\ve_j\not\in\{0,-\delta\}$, contains $a_{31}$, $a_{32}$, $a_{34}$, which are then zero. Then, we get the condition $a_{33}(\ve_1^2-\ve_2^2)=0$, hence $\ve_1=-\ve_2$ is the only non trivial case. The monomial $a_{33}\delta\ve_1$ also appears, proving that no further reduction of the Lee form is possible in the case $\delta\neq0$. Consider now, besides the condition that $A$ yields a morphism of Lie algebra, the assumption that it moves $\Omega_+$ into $\Omega_-$. The computation of the Gr\"obner basis yields the elements $a_{31}$, $a_{32}$, $a_{33} + 1$, $a_{34}$, which give a first reduction of $A$. We then have the elements $a_{41}^2 + a_{42}^2$, $(a_{44}-1)\ve$. We then get $a_{21}^2 + a_{22}^2 + 1$, concluding the proof of the claim.

\subsection{$\fh_4$, $(\frac{1}{2}14+24,\frac{1}{2}24,-12+34,0)$}

A generic $1$-form $\vartheta=\sum_{j=1}^{4}\vt_je^j$ is closed if and only if $\vt_1=\vt_2=\vt_3=0$. Thus the generic Lee form is $\vt=\vt_4e^4$ with $\vt_4\neq 0$. We consider a $2$-form $\Omega=\sum_{1\leq j < k \leq 4} \omega_{jk}e^{jk}$ and impose $d_\vt\Omega=0$. We obtain the following equations:
\begin{enumerate}
\item $(\vt_4+1)\om_{12}+\om_{34}=0$
\item $\left(\vt_4+\frac{3}{2}\right)\om_{13}=0$
\item $\left(\vt_4+\frac{3}{2}\right)\om_{23}+\om_{13}=0$
\end{enumerate}

We assume first that $\vt_4\notin\left\{-\frac{3}{2},-1\right\}$. Then $\om_{13}=0$ by (2), $\om_{23}=0$ by (3) and $\om_{34}=-(\vt_4+1)\om_{12}$ by (1). The generic lcs structure is therefore
\begin{equation}\label{eq:h4:1}
\vt=\vt_4e^4, \quad \Om=\om_{12}(e^{12}-(\vt_4+1)e^{34})+\om_{14}e^{14}+\om_{24}e^{24}
\end{equation}
and \eqref{eq:non_deg} gives $\om_{12}\neq 0$. Assume further that $\vt_4\neq-\frac{1}{2}$; then the automorphism
\begin{equation}\label{aut:h4:1}
\left(\begin{array}{cccc}
1 & 0 & 2c & 0 \\
a & 1 & 4c-2b+2ac & 0\\
0 & 0 & 1 & 0 \\
b & c & 0 & 1
\end{array}\right)
\end{equation}
with $b=\frac{4\om_{14}(\vt_4+1)-\om_{24}(2\vt_4+1)}{(2\vt_4+1)^2\om_{12}}$, $c=\frac{\om_{14}}{(2\vt_4+1)\om_{12}}$ and $a=0$ leaves $\vt$ invariant, while giving $\Om'=\om_{12}(e^{12}-(\vt_4+1)e^{34})$. We consider next the automorphism
\begin{equation}\label{aut:h4:2}
\left(\begin{array}{cccc}
s & 0 & 0 & 0 \\
0 & s & 0 & 0 \\
0 & 0 & s^2 & 0 \\
0 & 0 & 0 & 1
\end{array}\right)
\end{equation}
with $s=\frac{1}{\sqrt{\pm\om_{12}}}$, according to the sign of $\om_{12}$. We get the normal form
\[
\left\{ \begin{array}{ccl}
         \vt & = & \ve e^4\\
        \Omega & = & \pm(e^{12}-(\ve+1)e^{34}), \, \ve\not\in\left\{-\frac{3}{2},-1,-\frac{1}{2},0\right\}
        \end{array} \right.
\]

The above two forms $\Omega_+=e^{12}-(\ve+1)e^{34}$ and $\Omega_-=-(e^{12}-(\ve+1)e^{34})$ can not be reduced one to the other. Indeed, consider the generic matrix $A=(a_{jk})$ and its associated linear map in the basis $\{e^j\}$. The condition for being a morphism of the Lie algebra and for transforming $\Omega_+$ into $\Omega_-$ yields an ideal; if we compute a Gr\"obner basis, we notice that it contains $a_{22}^2 + 1$, proving the claim.

Consider now the case $\vt_4=-\frac{1}{2}$. If $\omega_{14}\neq0$, the automorphism \eqref{aut:h4:1} with $a=-\frac{\om_{24}}{\om_{14}}$ and $b=c=0$ gives $\vt'=\vt$ and $\Om'=\om_{12}(e^{12}-\frac{1}{2}e^{34})+\om_{14}e^{14}$. Using \eqref{aut:h4:2}, according to the sign of $\omega_{12}$, with $s=\frac{1}{\sqrt{\pm\om_{12}}}$ gives 
$\vt''=-\frac{1}{2}e^4$ and $\Om''=\pm(e^{12}- \frac{1}{2}e^{34}) + \frac{\omega_{14}}{\sqrt{\pm\omega_{12}}}e^{14}$. Using again \eqref{aut:h4:2} with $s=-1$, we get
the normal form
\[
\left\{ \begin{array}{ccl}
         \vt & = & -\frac{1}{2}e^4\\
        \Omega & = & \pm(e^{12}-\frac{1}{2}e^{34})+\sigma e^{14}, \ \sigma\in\bR,\sigma\geq0
        \end{array} \right.\,.
\]
If $\om_{14}=0$, the automorphism \eqref{aut:h4:1} with $c=\frac{\om_{24}}{2\om_{12}}$ and $a=b=0$ gives $\vt'=\vt$ and $\Om'=\om_{12}(e^{12}-\frac{1}{2}e^{34})$, which gives no further normal form.

The above forms $\Omega_1=\varepsilon_1(e^{12}-\frac{1}{2}e^{34})+\sigma_1 e^{14}$ and $\Omega_2=\varepsilon_2(e^{12}-\frac{1}{2}e^{34})+\sigma_2 e^{14}$, for $\sigma_1,\sigma_2\in\mathbb{R}$, $\varepsilon_1,\varepsilon_2\in\{1,-1\}$, can not be transformed into one another. Indeed, arguing as before, we find an automorphism of the form $A=(a_{jk})$ with respect to the basis $\{e^j\}$. We notice that we are reduced to
$$
\left(\begin{array}{cccc}
a_{11} & 0 & a_{13} & 0 \\
a_{21} & a_{11} & a_{23} & 0 \\
0 & 0 & a_{11}^{2} & 0 \\
a_{41} & a_{42} & a_{43} & 1
\end{array}\right)
$$
with further conditions which include, in particular, $-a_{11}^2\varepsilon_1 + \varepsilon_2=0$. Then we get that $\varepsilon_1=\varepsilon_2$. By continuing, we have that $\varepsilon_1 (\sigma_1 - \sigma_2)  (\sigma_1 + \sigma_2)=0$. Since $\varepsilon_1\neq0$, then either $\sigma_1=\sigma_2$, or $\sigma_1=-\sigma_2$, concluding the claim.

If $\vt_4=-\frac{3}{2}$ then $\om_{13}=0$ by (3) and $\om_{34}=\frac{1}{2}\om_{12}$ by (1). The generic lcs form is then 
\[
\vt=-\frac{3}{2}e^4, \quad \Om=\om_{12}\left(e^{12}+\frac{1}{2}e^{34}\right)+\om_{14}e^{14}+\om_{23}e^{23}+\om_{24}e^{24}
\]
with $\om_{12}^2+2\om_{14}\om_{23}\neq 0$. If $\om_{23}\neq 0$, the automorphism \eqref{aut:h4:1} with
$a=-\frac{2\om_{23}\om_{24}}{\om_{12}^2+2\om_{14}\om_{23}}$, $c=\frac{\om_{12}}{2\om_{23}}$ and $b=0$ gives $\vt'=\vt$ and $\Om'=\frac{\om_{12}^2+2\om_{14}\om_{23}}{2\om_{23}}e^{14}+\om_{23}e^{23}$. The automorphism \eqref{aut:h4:2} with $s=\frac{2\omega_{23}}{\omega_{12}^2 + 2\omega_{14}\omega_{23}}$ gives the normal form
\[
\left\{ \begin{array}{ccl}
         \vt & = & -\frac{3}{2}e^4\\
         & & \\
        \Omega & = & e^{14}+\sigma e^{23}, \ \sigma\in\mathbb{R}\,.
        \end{array} \right.
\]
If $\om_{23}=0$ then $\om_{12}\neq 0$ and we are back at \eqref{eq:h4:1}. Finally if $\vt_4=-1$ we get $\omega_{13}=\omega_{23}=\omega_{34}=0$, whence $\Omega$ is degenerate.

For different $\sigma\in\bR$, the above normal forms are different. Indeed, trying to find an automorphism of the Lie algebra transforming $\Omega_1\coloneq e^{14}+\sigma_1 e^{23}$ into $\Omega_2\coloneq e^{14}+\sigma_2 e^{23}$ for some $\sigma_1,\sigma_2\in\mathbb{R}$, we have to solve the Gr\"obner ideal
\begin{eqnarray*}
B &=& \left(a_{21} a_{42}, a_{42} \sigma_{2}, a_{43} \sigma_{2} + a_{21}, a_{11} - 1, a_{12}, a_{13} - 2 a_{42}, a_{14}, a_{22} - 1, \right.\\[5pt]
&& \left. a_{23} + 2 a_{41} - 4 a_{42}, a_{24}, a_{31}, a_{32}, a_{33} - 1, a_{34}, a_{44} - 1, \sigma_{1} -  \sigma_{2}\right)
\end{eqnarray*}
which contains, in particular, $\sigma_1 - \sigma_2$.

Finally, we have to prove that there is no automorphism transforming one Lee form $\vt_1\coloneq\varepsilon_1 e^4$ to another $\vt_2\coloneq\varepsilon_2 e^4$. For an automorphism $A=(a_{jk})$ with respect to the basis $\{e^j\}$ we are reduced to
$$
\left(\begin{array}{cccc}
a_{11} & 0 & a_{13} & 0 \\
a_{21} & a_{11} & a_{23} & 0 \\
0 & 0 & a_{33} & 0 \\
a_{41} & a_{42} & a_{43} & 1
\end{array}\right)
$$
with the further conditions
\begin{eqnarray*}
B &=& (a_{21} a_{42}^{2} - \frac{1}{2} a_{13} a_{41} + a_{13} a_{42} - \frac{1}{2} a_{23} a_{42}, a_{33} a_{42}^{2} - \frac{1}{4} a_{13}^{2}, a_{11}^{2} -  a_{33}, a_{11} a_{13} - 2 a_{33} a_{42},\\[5pt]
&&  a_{13} a_{21} -  a_{11} a_{23} - 2 a_{33} a_{41} + 4 a_{33} a_{42},a_{11} a_{41} -  a_{21} a_{42} -  a_{13} + \frac{1}{2} a_{23}, a_{11} a_{42} - \frac{1}{2} a_{13},\\[5pt] 
&& \varepsilon_{1} -  \varepsilon_{2})
\end{eqnarray*}
among which there appears $\varepsilon_1=\varepsilon_2$.
\end{proof}

\end{document}